\documentclass[a4paper]{amsart}

\AtBeginDocument{%
   \def\MR#1{}
}

\usepackage{amsmath,amssymb}
\usepackage{amsthm}
\usepackage{mathrsfs}
\usepackage[shortlabels]{enumitem}
\usepackage{mathtools}
\usepackage[all,cmtip,2cell]{xy}
\usepackage{array}
\usepackage{braket}
\usepackage{stmaryrd}

\theoremstyle{plain}
\newtheorem{thm}{Theorem}[section]
\newtheorem{prop}[thm]{Proposition}
\newtheorem{lem}[thm]{Lemma}
\newtheorem{claim}[thm]{Claim}

\newtheorem*{thm*}{Theorem}
\newtheorem{mainthm}{Theorem}
\newtheorem{maincor}{Corollary}

\theoremstyle{definition}
\newtheorem{defi}[thm]{Definition}

\newtheorem{dsc}[thm]{}

\theoremstyle{remark}
\newtheorem{rem}[thm]{Remark}

\newtheorem{ques}[thm]{Question}
\newtheorem{ex}[thm]{Example}

\numberwithin{equation}{section}

\newcommand{\Z}{\mathbb{Z}}
\newcommand{\Q}{\mathbb{Q}}
\newcommand{\R}{\mathbb{R}}
\newcommand{\C}{\mathbb{C}}
\renewcommand{\P}{\mathbb{P}}
\newcommand{\F}{\mathbb{F}}
\newcommand{\T}{\mathbb{T}}

\renewcommand{\a}{\alpha}
\renewcommand{\b}{\beta}
\renewcommand{\d}{\delta}
\newcommand{\D}{\Delta}
\newcommand{\e}{\varepsilon}
\newcommand{\E}{\mathbb{E}}
\renewcommand{\F}{\mathbb{F}}
\newcommand{\g}{\gamma}
\newcommand{\G}{\Gamma}
\renewcommand{\L}{\Lambda}
\newcommand{\s}{\sigma}
\renewcommand{\t}{\tau}
\newcommand{\vp}{\varphi}

\newcommand{\ds}{\displaystyle}
\newcommand{\mc}{\mathcal}
\newcommand{\ms}{\mathscr}

\newcommand{\mb}{\mathbb}

\newcommand{\emp}{\varnothing}

\newcommand{\ol}{\overline}
\newcommand{\wt}{\widetilde}

\newcommand{\ra}{\Rightarrow}

\newcommand{\hra}{\hookrightarrow}
\newcommand{\epm}{\twoheadrightarrow}
\newcommand{\dra}{\dashrightarrow}

\DeclareMathOperator{\rk}{rk}
\DeclareMathOperator{\id}{id}

\DeclareMathOperator{\Sym}{\mathrm{Sym}}

\DeclareMathOperator{\Spec}{\mathrm{Spec}}

\DeclareMathOperator{\Proj}{\mathrm{Proj}}

\DeclareMathOperator{\Exc}{\mathrm{Exc}}

\DeclareMathOperator{\Pic}{\mathrm{Pic}}

\DeclareMathOperator{\red}{\mathrm{red}}
\DeclareMathOperator{\NE}{\mathrm{NE}}

\DeclareMathOperator{\Bl}{\mathrm{Bl}}

\DeclareMathOperator{\codim}{\mathrm{codim}}
\DeclareMathOperator{\Sing}{\mathrm{Sing}}
\DeclareMathOperator{\Bs}{\mathrm{Bs}}

\DeclareMathOperator{\pr}{pr}

\DeclareMathOperator{\Hom}{Hom}

\DeclareMathOperator{\Ker}{Ker}
\DeclareMathOperator{\Cok}{Cok}
\DeclareMathOperator{\Supp}{Supp}
\DeclareMathOperator{\Ext}{Ext}

\DeclareMathOperator{\hd}{hd}
\DeclareMathOperator{\Tr}{Tr}
\DeclareMathOperator{\Eu}{Eu}

\DeclareMathOperator{\lgth}{length}
\DeclareMathOperator{\Hilb}{Hilb}
\DeclareMathOperator{\Jac}{Jac}

\title[{\tiny Relative linear extensions of sextic del Pezzo fibrations}]{Relative linear extensions of sextic del Pezzo fibrations over curves}
\author[T. FUKUOKA]{Takeru Fukuoka}
\date{\today}
\address{Graduate School of Mathematical Sciences\\The University of Tokyo\\3-8-1 Komaba\\Meguro-ku, Tokyo 153-8914, Japan}
\email{tfukuoka@ms.u-tokyo.ac.jp}
\subjclass[2010]{14E25, 14E30.}

\begin{document}
\maketitle
\begin{abstract}
In this paper, we study a sextic del Pezzo fibration over a curve comprehensively. 
We obtain certain formulae of several basic invariants of such a fibration. 
We also establish the embedding theorem of such a fibration which asserts that every such a fibration is a relative linear section of a Mori fiber space with the general fiber $(\mathbb{P}^{1})^{3}$ and that with the general fiber $(\mathbb{P}^{2})^{2}$. 
As an application of this embedding theorem, we classify singular fibers of such a fibrations and answer a question of T.~Fujita about existence of non-normal fibers. 

\end{abstract}
\tableofcontents

\section{Introduction}
\subsection{Motivations}
A smooth del Pezzo surface $S$ of degree $d$ is defined to be a smooth projective surface whose anti-canonical divisor $-K_{S}$ is ample with $(-K_{S})^{2}=d$. 
It is a famous result that for any integer $d \in \{1,\ldots,9\}$, there exists a certain variety $V_{d}$ such that every del Pezzo surface $S$ of degree $d$ is an (weighted) complete intersection of $V_{d}$. 
For example, when $d=3$ (resp. $d=4$), we take $V_{3}=\P^{3}$ (resp. $V_{4}=\P^{4}$) and every del Pezzo surface of degree $3$ (resp. $4$) is a cubic hypersurface on $\P^{3}$ (resp. a complete intersection of two quadrics on $\P^{4}$). 
When $d=6$, we can take not only $V_{6}=(\P^{1})^{3}$ but also $V_{6}=(\P^{2})^{2}$. 
Then every del Pezzo surface of degree $6$ is a hyperplane section of $(\P^{1})^{3}$ and also a codimension $2$ linear section of $(\P^{2})^{2}$ with respect to the Segre embeddings. 
These descriptions are classic and useful to study del Pezzo surfaces. 

In this paper, we mainly discuss about how to relativize these description for del Pezzo fibrations. 
A relativization of those embeddings is important to study del Pezzo fibrations and actually have been used in several researches (e.g. \cite{Takeuchi09}). 
As we will see in the next subsection, relativizations of these descriptions for del Pezzo fibrations of degree $d$ was established when $d \neq 6$. 
One of the main results of this paper is to give a relativization when $d=6$. 

\subsection{Known results}
In this paper, we employ the following definition for del Pezzo fibrations as in the context of Mori theory. 
Let $X$ be a smooth projective $3$-fold whose canonical divisor $K_{X}$ is not nef. 
By virtue of Mori theory, $X$ has an extremal contraction $\vp \colon X \to C$, which is a surjective morphism onto a normal projective variety $C$ with connected fibers satisfying that $\rho(X/C)=1$ and $-K_{X}$ is $\vp$-ample. 
When $\dim C=1$, we call $\vp$ a \emph{del Pezzo fibration}, which is one of the final outputs of the minimal model program.  
In this case, a general $\vp$-fiber $F$ is a del Pezzo surface. 
Then the \emph{degree} of a del Pezzo fibration $\vp \colon X \to C$ is defined to be $(-K_{F})^{2}$. 

Let $\vp \colon X \to C$ be a del Pezzo fibration of degree $d$. 
In the paper \cite{Mori82}, Mori proved that $1 \leq d \leq 9$ and $d \neq 7$. 
Moreover, he proved that if $d=9$ then $\vp$ is $\P^{2}$-bundle, and if $d=8$ then there exists an embedding of $X$ into a $\P^{3}$-bundle over $C$ containing $X$ as a quadric fibration \cite[Theorem~(3.5)]{Mori82}. 
When $d=1$ or $2$, Fujita proved that there exists an weighted projective space bundle containing $X$ as a relative weighted hypersurface \cite[(4.1),(4.2)]{Fujita90}. 

Now we assume that $3 \leq d \leq 6$. 
Then $\vp \colon X\to C$ has a natural embedding into the $\P^{d}$-bundle $p \to \P_{C}(\vp_{\ast}\mc O(-K_{X})) \to C$. 
D'Souza \cite[(2.2.1) and (2.3.1)]{DSouza88} and Fujita \cite[(4.3) and (4.4)]{Fujita90} proved that if $d=3$ or $4$, then $X$ is a relative complete intersection in $\P_{C}(\vp_{\ast}\mc O(-K_{X}))$.  
More precisely, when $d=4$ for example, they proved that there is a rank $2$ vector bundle $\mc E$ on $C$ such that $X$ is the zero scheme of a global section of $\mc O_{\P_{C}(\vp_{\ast}\mc O(-K_{X}))}(2) \otimes p^{\ast}\mc E$. 
When $d =5$ or $6$, $\P_{C}(\vp_{\ast}\mc O(-K_{X}))$ does not contain $X$ as a relative complete intersection and hence it seems to be difficult to treat such an $X$ as a submanifold of $\P_{C}(\vp_{\ast}\mc O(-K_{X}))$. 
When $d=5$ and $C=\P^{1}$, however, K.~Takeuchi claimed that $X$ is relatively defined in $\P_{\P^{1}}(\vp_{\ast}\mc O(-K_{X}))$ by the Pfaffian of the $4 \times 4$ diagonal minors of a $5 \times 5$ skew-symmetric matrix \cite[Theorem~(3.3)~(v)]{Takeuchi09}. 

\subsection{Main Results}

In this paper, we mainly treat a \emph{sextic} del Pezzo fibration $\vp \colon X \to C$, i.e., that of degree $6$.

\subsubsection{Associated coverings}

For every sextic del Pezzo fibration $\vp \colon X \to C$, we will define smooth projective curves $B,T$ with a double covering structure $\vp_{B} \colon B \to C$ and a triple covering structure $\vp_{T} \colon T \to C$ respectively associated to $\vp$. 
These coverings $\vp_{B}$ and $\vp_{T}$ are deeply related to the relative Hilbert scheme of twisted cubics and conics respectively (see Lemma~\ref{lem-indep}). 
In particular, when all $\vp$-fiber are normal, the coverings $B$ and $T$ coincide with the coverings $\mathscr{Z}_{3}$ and $\mathscr{Z}_{2}$ over $C$ that are defined by Kuznetsov \cite{Kuznetsov17} (see Proposition~\ref{prop-Kuz}). 
We refer to Definition~\ref{defi-BT} for the precise definition. 

\subsubsection{Formulae for invariants $(-K_{X})^{3}$ and $h^{1,2}(X)$}
For a sextic del Pezzo fibration $X \to C$, the associated coverings are closely related to the invariants $(-K_{X})^{3}$ and $h^{1,2}(X)$. 

\begin{mainthm}\label{mainthm-inv}
Let $\vp \colon X \to C$ be a sextic del Pezzo fibration. 
Let $\vp_{B} \colon B \to C$ be the double covering and $\vp_{T} \colon T \to C$ the triple covering associated to $\vp$ (see Definition~\ref{defi-BT}). 
Then the following assertions hold.  

\begin{enumerate}
\item $\mc J(X) \times \Jac(C)$ is isomorphic to $\Jac(B) \times \Jac(T)$ as complex tori, where $\mc J(X)$ is the intermediate Jacobian of $X$. Moreover, if $C=\P^{1}$, then 
$\mc J(X)$ is isomorphic to  $\Jac(B) \times \Jac(T)$ as principally polarized abelian varieties. 
\item It holds that $(-K_{X})^{3}=22-(6g(B)+4g(T)+12g(C))$. 
\end{enumerate}
\end{mainthm}

This theorem shows that the invariants $(-K_{X})^{3}$ and $h^{1,2}(X)$ can be interpreted by using the genera of two curves $B$ and $T$. 

\subsubsection{Relative linear extensions}

Let us recall that a smooth sextic del Pezzo surface is a hyperplane section of $(\P^{1})^{3}$ and also a codimension $2$ linear section of $(\P^{2})^{2}$ under the Segre embeddings. 
In the following two theorems, we relativize these embeddings for every sextic del Pezzo fibration.

\begin{mainthm}\label{mainthm-P13}
Let $\vp \colon X \to C$ be a sextic del Pezzo fibration and $\vp_{B} \colon B \to C$ the double covering associated to $\vp$. 
Set $\mc L:=\Cok(\mc O_{C} \to {\vp_{B}}_{\ast}\mc O_{B}) \otimes \mc O(-K_{C})$. 
Then there exists a smooth projective 4-fold $Y$, an extremal contraction $\vp_{Y} \colon Y \to C$ and a divisor $H_{Y}$ on $Y$ satisfying the following conditions. 
\begin{enumerate}
\item Every smooth fiber of $\vp_{Y}$ is isomorphic to $(\P^{1})^{3}$. 
\item $\mc O_{Y}(K_{Y}+2H_{Y})=\vp_{Y}^{\ast}\mc L$ holds. 
\item $Y$ contains $X$ as a member of $|\mc O(H_{Y}) \otimes \vp_{Y}^{\ast}\mc L^{-1}|$. 
\end{enumerate}
\end{mainthm}

\begin{mainthm}\label{mainthm-P22}
Let $\vp \colon X \to C$ be a sextic del Pezzo fibration and $\vp_{T} \colon T \to C$ the triple covering associated to $\vp$. 
Set $\mc G:=\Cok (\mc O_{C} \to {\vp_{T}}_{\ast}\mc O_{T}) \otimes \mc O(-K_{C})$. 
Then there exists a smooth projective 5-fold $Z$,  an extremal contraction $\vp_{Z} \colon Z \to C$ and a divisor $H_{Z}$ on $Z$ satisfying the following conditions. 
\begin{enumerate}
\item Every smooth fiber of $\vp_{Z}$ is isomorphic to $(\P^{2})^{2}$. 
\item $\mc O_{Z}(K_{Z}+3H_{Z})=\vp_{Z}^{\ast}\det \mc G$ holds. 
\item There exists a section $s \in H^{0}(Z,\mc O_{Z}(H_{Z}) \otimes \vp_{Z}^{\ast}\mc G^{\vee})$ such that $X$ is isomorphic to the zero scheme of $s$. 
\end{enumerate}
\end{mainthm}

\begin{rem}
Note that the sheaf $\mc L$ (resp. $\mc G$) in Theorem~\ref{mainthm-P13} (resp. Theorem~\ref{mainthm-P22}) is invertible (resp. locally free of rank $2$). 
\end{rem}

One of the most different points from the case where the degree is not $6$ is that $\vp_{Y}$ and $\vp_{Z}$ in Theorems~\ref{mainthm-P13} and \ref{mainthm-P22} may have singular fibers, and must have when $C=\P^{1}$ by the invariant cycle theorem (cf. \cite[II, Theorem~4.18]{VoisinBook}). 
We will classify singular fibers of $\vp_{Y}$ and $\vp_{Z}$ in Theorem~\ref{mainthm-singfib}. 
Moreover, as an application of Theorems~\ref{mainthm-P13} and \ref{mainthm-P22}, we will classify singular fibers of sextic del Pezzo fibrations in Subsections~1.4 and 1.5.

\subsubsection{Classification of singular fibers of sextic del Pezzo fibrations}

Let us recall Fujita's result about singular fibers of del Pezzo fibrations \cite{Fujita90}. 
\begin{thm}[{\cite[(4,10)]{Fujita90}}]
Let $\vp \colon X \to C$ be a del Pezzo fibration. 
If $\vp$ is not of degree $6$, then every fiber of $\vp$ is normal. 
\end{thm}
However, singular fibers of sextic del Pezzo fibrations was yet to be classified. 
Indeed, Fujita proposed the following question. 
\begin{ques}[{\cite[(3,7)]{Fujita90}}]\label{ques-nonnormal}
Does there exist del Pezzo fibrations of degree $6$ containing non-normal fibers?
\end{ques}

Another main result of this paper is a classification of singular fibers of sextic del Pezzo fibrations $\vp \colon X \to C$. 
For the proof, we will use the embeddings $X \hra Y$ and $X \hra Z$ as in Theorems~\ref{mainthm-P13} and \ref{mainthm-P22}. 
In summary, we will show the following theorem.

\begin{mainthm}\label{mainthm-singfib}
Let $\vp \colon X \to C$ be a sextic del Pezzo fibration. 
Let $\vp_{B} \colon B \to C$ and $\vp_{T} \colon T \to C$ be the coverings associated to $\vp$. 
Let  $X \hra Y$ and $X \hra Z$ be the embeddings as in Theorems~\ref{mainthm-P13} and \ref{mainthm-P22} respectively. 
For $t \in C$, we set $B_{t}:=\vp_{B}^{-1}(t)$ and $T_{t}:=\vp_{T}^{-1}(t)$. 

Then for every $t \in C$, the numbers $(\#(B_{t})_{\red},\#(T_{t})_{\red})$ determines the isomorphism classes of $Y_{t}$, $Z_{t}$ and the possibilities of those of $X_{t}$ as in Table~1. 

\begin{table}[h]\label{table-main}
\begin{tabular}{|c|c||c|c|c|}
\hline 
 $\#(B_{t})_{\red}$ & $\#(T_{t})_{\red}$  &$X_{t}$& $Y_{t}$ &$Z_{t}$ \\ 
\hline \hline 
$2$ & $3$  & (2,3)&$(\P^{1})^{3}$& $(\P^{2})^{2}$ \\ 
\hline
$2$ & $2$ &(2,2)&$\P^{1} \times \Q^{2}_{0}$&$(\P^{2})^{2}$ \\
\hline
$2$ & $1$ & (2,1)&$\P^{1,1,1}$& $(\P^{2})^{2}$ \\
\hline
$1$ & $3$ &(1,3)&$(\P^{1})^{3}$& $\P^{2,2}$\\
\hline
$1$ & $2$&(1,2) or (n2)&$\P^{1} \times \Q^{2}_{0}$& $\P^{2,2}$\\
\hline
$1$ & $1$ &(1,1) or (n4)& $\P^{1,1,1}$&$\P^{2,2}$ \\
\hline
\end{tabular}
\caption{The singular fibers of $\vp$, $\vp_{Y}$ and $\vp_{Z}$}
\end{table}

For the definitions of ($i$,$j$), (n2) and (n4), we refer to Theorem~\ref{thm-GordP6}. 
$\Q^{2}_{0}$ denotes the quadric cone. 
For the definition of $\P^{1,1,1}$ (resp. $\P^{2,2}$), we refer to Definition~\ref{defi-111} (resp. Definition~\ref{defi-22}).

In particular, if $X_{t}$ is normal, then the isomorphism class of $X_{t}$ is determined by the pair $(\#(B_{t})_{\red},\#(T_{t})_{\red})$ and the number of lines in $X_{t}$ is equals to $\#(B_{t})_{\red} \times \#(T_{t})_{\red}$. 

\end{mainthm}

As applications of Theorems~\ref{mainthm-inv}~and~\ref{mainthm-singfib}, we have the following properties of a sextic del Pezzo fibration $\vp \colon X \to C$. 

\setcounter{mainthm}{5}
\begin{maincor}\label{maincor-inv}
Let $\vp \colon X \to C$ be a sextic del Pezzo fibration. 
\begin{enumerate}
\item It holds that $(-K_{X})^{3} \leq 22$ and $(-K_{X})^{3} \neq 20$. 
\item If $(-K_{X})^{3}>0$, then $C \simeq \P^{1}$. In particular, $X$ is rational since $\vp$ admits a section (cf. \cite[Theorem~4.2]{Manin66} and \cite{Swinnerton-Dyer72}). 
\item It holds that $(-K_{X/C})^{3} \leq 0$ and the equality holds if and only if $\vp$ is a smooth morphism. 
\end{enumerate}
\end{maincor}

\subsubsection{Existence of non-normal fibers}

We will give an answer to Question~\ref{ques-nonnormal} by presenting sextic del Pezzo fibrations with non-normal fibers. 
More precisely, as a consequence of Examples \ref{ex-(2,j)}, \ref{ex-(1,j)} and \ref{ex-nonnormal}, we will show the following theorem. 
\begin{mainthm}\label{mainthm-ex}
Let $X_{0}$ be a sextic Gorenstein del Pezzo surface, which is possibly non-normal. 
Suppose that $X_{0}$ is not a cone over an irreducible curve of arithmetic genus $1$. 
Then there exists a sextic del Pezzo fibration $\vp \colon X \to C$ containing $X_{0}$. 
\end{mainthm}
In particular, Theorem~\ref{mainthm-ex} gives an affirmative answer for Question~\ref{ques-nonnormal}. 

\subsubsection{Relative double projection}
Our proof of the main results are based on the \emph{relative double projection} from a section of $\vp$. 
This is a relativization of the \emph{double projection} from a general point $x$ of a sextic del Pezzo surface $S$, which is given as follows. 
Under the embedding $S \hra \P^{6}$ given by the anti-canonical system, we consider the projection $\P^{6} \dra \P^{3}$ from the tangent plane $\mathbb{T}_{x}S=\P^{2} \subset \P^{6}$ of $S$ at $x$. 
The proper image of $S$ under this birational map is a smooth quadric surface $\Q^{2}$ and 
the map $S \dra \Q^{2}$ is birational. 
This birational map is what is called the \emph{double projection} from the point $x$ on $S$. 
In Proposition~\ref{prop-2ray68}, we will make a relativization of this birational map $S \dra \Q^{2}$ for a sextic del Pezzo fibration. 

\subsection{Organization of this paper}
We organize this paper as follows. 

In Section~2, we will establish a relativization of the double projection from a point on a sextic del Pezzo surface (=Proposition~\ref{prop-2ray68}). 

In Section~3, we will collect some preliminary results for quadric fibrations to define the associated coverings and prove Theorem~\ref{mainthm-inv}. 
Furthermore, we will see the following two statements: a characterization of a certain nef vector bundle of rank $3$ on a quadric surface (=Proposition~\ref{prop-Spi2}), and a variant of the Hartshorne-Serre correspondence on a family of surfaces with a multi-section (=Theorem~\ref{thm-univext}). 
These two statements will be necessary for our proving Theorem~\ref{mainthm-P22}. 
As Theorem~\ref{thm-univext} is formulated in a slightly general form, we will postpone its proof to Appendix~\ref{app-relext}.

In Section~4 and 5, we will prove Theorems~\ref{mainthm-P13} and \ref{mainthm-P22} respectively. 

In Section~6, using Theorems~\ref{mainthm-P13} and \ref{mainthm-P22}, we will classify the singular fibers of the sextic del Pezzo fibrations and prove Theorem~\ref{mainthm-singfib} and Corollary~\ref{maincor-inv}. 

In Section~7, we will prove Theorem~\ref{mainthm-ex} by using results in Section~6 and giving explicit examples. 

In Section~8, we will mention some recent related works. 
Since 2016, families of sextic del Pezzo surfaces have been studied more actively in different contexts from this paper (cf. \cite{AHTVA16,Kuznetsov17}). 
We will compare some results to our main results. 

\subsection{Notation and Conventions}\label{NC}
Throughout of this paper, we work over the complex number field $\C$. 
We basically adopt the terminology of \cite{Hartshorne77,KM98}. 

(1) We use the following notation for a surjective morphism $f \colon X \to Y$ among varieties $X$ and $Y$. 
\begin{itemize}
\item If $t \in Y$ is a point, then $X_{t}$ stands for the scheme-theoretic fiber $f^{-1}(t)$ of $t$ under the morphism $f$. 
\item For Cartier divisors $D$ and $D'$ on $X$, $D \sim_{Y} D'$ or $D \sim_{f} D'$ means that there exists a Cartier divisor $A$ on $Y$ such that $D \sim D' + f^{\ast}A$. 
\end{itemize}


(2) For a product $X_{1} \times X_{2}$, $\pr_{i}$ denotes the $i$-th projection. 

(3) For a closed subscheme $Y \subset X$, the ideal sheaf of $Y$ is denoted by $\mc I_{Y/X}$ or $\mc I_{Y}$. 
The normal sheaf is denoted by $\mc N_{Y/X}$. 

(4) For a locally free sheaf $\mc E$ on $X$, we set $\P_{X}(\mc E)=\Proj_{X} \Sym \mc E$. 
We say that $\mc E$ is nef if $\mc O_{\P(\mc E)}(1)$ is nef. 

(5) For $n \in \Z_{\geq 0}$, the Hirzebruch surface is defined to be $\F_{n}:=\P_{\P^{1}}(\mc O \oplus \mc O(n))$. $h$ denotes a tautological divisor, $f$ a fiber of $\F_{n} \to \P^{1}$ and $C_{0} \in |h-nf|$ the negative section of $\F_{n}$. 


(6) For a smooth projective curve $C$, $g(C)$ denotes the genus of $C$ and $\Jac(C)$ the Jacobian of $C$. 

(7) For a smooth projective 3-fold $X$, $\mc J(X)$ denotes the intermediate Jacobian of $X$. When $h^{1,0}(X)=0$, we regard $\mc J(X)$ as a principally polarized abelian variety as in \cite{CG72}. 

(8) $\Q^{n}$ denotes the non-singular quadric hypersurface in $\P^{n+1}$. 

(9) $\Q^{2}_{0}$ denotes a singular quadric in $\P^{3}$ with an ordinary double point. 

(10) On $\P^{m} \times \P^{n}$, $\mc O(a,b)$ is defined to be $\pr_{1}^{\ast}\mc O_{\P^{m}}(a) \otimes \pr_{2}^{\ast}\mc O_{\P^{n}}(b)$. 

(11) For an irreducible and reduced quadric surface $Q \subset \P^{3}$, $\mc O_{Q}(1)$ denotes the very ample line bundle with respect to the embedding. 
Moreover, under a fixed linear embedding $Q \hra \Q^{3}$, we define
\[\ms S_{Q}:=\ms S_{\Q^{3}}|_{Q} \]
where $\ms S_{\Q^{3}}$ is the globally generated spinor bundle on $\Q^{3}$ (see \cite{Kapranov88}). 
In this paper, we follow Kapranov's definition for the spinor bundle, which is the dual of what was defined by Ottaviani in \cite{Ottaviani88}. 

(12) We say that $\vp \colon X \to C$ is a \emph{del Pezzo fibration} if $\vp$ is an extremal contraction from a non-singular projective 3-fold $X$ onto a smooth projective curve $C$, i.e., $-K_{X}$ is $\vp$-ample and $\rho(X/C)=1$, as in Subsection~1.1. 

(13) We say that $q \colon Q \to C$ is a \emph{quadric fibration} if $q$ is a del Pezzo fibration of degree $8$. In particular, we assume that $Q$ is a smooth projective $3$-fold and $\rho(Q/C)=1$. The general fiber of $q$ is actually isomorphic to $\Q^{2}$ by the following theorem, which we will use throughout this paper: 

\begin{thm}[\cite{Manin66,Mori82}]
For a del Pezzo fibration $\vp \colon X \to C$, the following assertions hold.
\begin{enumerate}
\item $\vp$ has a section. 
\item If $\vp$ is of degree $9$, then $\vp$ is a $\P^{2}$-bundle. 
\item If $\vp$ is of degree $8$, then every $\vp$-fiber is an irreducible and reduced quadric surface in $\P^{3}$. 
Moreover, there exists an embedding into a $\P^{3}$-bundle $f \colon \F \to C$ with a tautological divisor $H$ such that $Q \in |2H+f^{\ast}L|$ for some $L \in \Pic(C)$. 
\item There exists an exact sequence $0 \to \Pic(C) \overset{\vp^{\ast}}{\to}\Pic(X) \overset{{(-.l)}}{\to} \Z \to 0$, where $l$ is a line in general $\vp$-fiber. 
\end{enumerate}
\end{thm}
\begin{proof}
(1) is proved by \cite[Theorem~4.2]{Manin66} and \cite{Swinnerton-Dyer72}. 
(2) and (3) follow from \cite[Theorem~(3.5)]{Mori82}. 
(4) follows from \cite[Theorems~(3.2) and (3.5)]{Mori82}. 
\end{proof}

\subsection{Acknowledgement}
The author wishes to express his deepest gratitude to Professor Hiromichi Takagi, his supervisor, for his valuable comments, suggestions and encouragement. 
The author is grateful to Professor Asher Auel for variable comments and introducing the papers, especially \cite{ABB14} and \cite{LT16}.  
The author is also grateful to Doctor Akihiro Kanemitsu for discussions on nef vector bundles and variable suggestions. 
The author is also grateful to Doctor Naoki Koseki for introducing the paper \cite{BPS80}. 
The author is also grateful to Professor Shigeru Mukai for variable comments on Lemma~\ref{lem-indep}. 
The author is also grateful to Doctor Masaru Nagaoka for variable discussions on non-normal del Pezzo surfaces, especially Example \ref{ex-nonnormal}. 
The author is also grateful to Professor Yuji Odaka for variable comments on Corollary~\ref{maincor-inv}~(3). 
During this work, the author also benefited from discussions with many people. 
Here the author shows my gratitude to Doctors Sho Ejiri, Makoto Enokizono, Wahei Hara and Yosuke Matsuzawa for discussions. 
This work was supported by the Program for Leading Graduate Schools, MEXT, Japan.
This work was also partially supported by JSPS's Research Fellowship for Young Scientists (JSPS KAKENHI Grant Number 18J12949). 

\section{Relative double projections}

For the proof of Theorems \ref{mainthm-P13} and \ref{mainthm-P22}, the relative double projection from a section of a sextic del Pezzo fibration plays an important role. 
We devote this section to prove Proposition~\ref{prop-2ray68} that establishes this technique. 

\subsection{Double projection from a smooth point on a smooth sextic del Pezzo surface}\label{subsec-doubleprojdP6}
First of all, we review the double projection from a general point on a smooth sextic del Pezzo surface $S$ precisely. 
Recall that $S$ is a linear section of $(\P^{1})^{3}$ and hence $S$ has three different conic bundle structures onto $\P^{1}$. 
If we take a general point $x \in S$, then the three fibers $G_{1},G_{2}$ and $G_{3}$ passing through $x$ are smooth rational curves with self-intersection $0$. 
On the blow-up $\Bl_{x}S$ of $S$ at $x$, the proper transform $G_{i}'$ of $G_{i}$ is a $(-1)$-curve for each $i$. 
Blowing $G_{1}',G_{2}'$ and $G_{3}'$ down to points $y_{1}$, $y_{2}$ and $y_{3}$ respectively, we obtain a smooth quadric surface $\Q^{2}$. 
The extracted curve $E$ by the blow-up $\mu \colon \Bl_{x}S \to S$ is transformed into a conic on $\Q^{2}$ passing through the three points $y_{1},y_{2},y_{3}$. 
In particular, $y_{1},y_{2},y_{3}$ are not colinear. 
This construction is summarized as the following diagram:
\begin{align}\label{dia-68gen}
\xymatrix{
&\ar[ld]_{\mu}\Bl_{x}S \ar@{=}[r] & \Bl_{y_{1},y_{2},y_{3}}\Q^{2} \ar[rd]^{\s} & \\
S&&&\Q^{2}.
}
\end{align}
Noting that $E=\s^{\ast}\s_{\ast}E-(G_{1}'+G_{2}'+G_{3}')$ and $G_{1}'+G_{2}'+G_{3}' \sim \mu^{\ast}(-K_{S})-3E$, 
we obtain $\s^{\ast}\s_{\ast}E = E+G_{1}'+G_{2}'+G_{3}' \sim \mu^{\ast}(-K_{S})-2E$. 
Therefore, the birational map $S \dra \Q^{2} \subset \P^{3}$ is the double projection from $x$. 

\subsection{Relativizations of double projections}

The following proposition is a method of relativizing the double projection of a sextic del Pezzo surface. 
The proof will be done by Takeuchi's 2-ray game argument. 

\begin{prop}\label{prop-2ray68}
Let $\vp \colon X \to C$ be a sextic del Pezzo fibration and $C_{0}$ a $\vp$-section. 
Let $\mu \colon \wt{X} =\Bl_{C_{0}}X \to X$ be the blow-up of $X$ along $C_{0}$ and $E=\Exc(\mu)$ and $\wt{\vp}:=\vp \circ \mu$.
Then the following assertions hold. 
\begin{enumerate}
\item $\mc O(-K_{\wt{X}})$ is $\vp$-globally generated. Moreover, if $\psi_{X} \colon \wt{X} \to W$ denotes the first part of the Stein factorization of the induced morphism $X \to \P_{C}(\vp_{\ast}\mc O(-K_{\wt{X}}))$ over $C$, then $\psi_{X}$ is an isomorphism or a flopping contraction. 
\item 
When $\psi_{X}$ is an isomorphism, let $\chi \colon \wt{X} \to \wt{Q}$ denote the identity map. 
Then there exists the unique contraction $\s \colon \wt{Q} \to Q$ of the another $K_{\wt{Q}}$-negative ray over $C$. 
When $\psi_{X}$ is a flopping contraction, let $\chi \colon \wt{X} \dra \wt{Q}$ denote the flop of $\psi_{X}$. 
Then there exists the contraction $\s \colon \wt{Q} \to Q$ of the $K_{\wt{Q}}$-negative ray over $C$. 

In both cases, we set the morphisms as in the following commutative diagram:
\begin{align}
\xymatrix{
&E\ar[ld]&\ar[ld]_{\mu} \ar@{}[l]|{\subset}\wt{X} \ar@{-->}[rr]^{\chi} \ar[rd]^{\psi_{X}} \ar[dd]_{\wt{\vp}}&&\wt{Q} \ar[rd]^{\s} \ar[ld]_{\psi_{Q}} \ar[dd]^{\wt{q}} \ar@{}[r]|{\supset}&G \ar[rd]&  \\
C&\ar@{}[l]|{\subset}X\ar[rd]_{\vp}&&W\ar[rd] \ar[ld]&&Q\ar[ld]^{q} \ar@{}[r]|{\supset} & T\\
&&C \ar@{=}[rr]&&C,&&
}\label{dia-2ray68}
\end{align}
\item The following assertions hold. 
\begin{enumerate}
\item $\s$ is the blow-up along a non-singular curve $T \subset Q$. 
\item $\deg (q|_{T})=3$ and $q$ is a quadric fibration. 
\end{enumerate}
Moreover, if we set $G = \Exc(\s)$, then there exists a divisor $\a$ on C such that 
\[G \sim -K_{\wt{Q}}-2E_{\wt{Q}}+\wt{q}^{\ast}\a.\]
\item The following equalities hold. 
\begin{align*}
(-K_{Q})^{3}&=\frac{1}{3}(4(-K_{X})^{3}+16g(T)-48g(C)+32), \\
-K_{Q}.T&=\frac{1}{6}((-K_{X})^{3}+22g(T)-54g(C)+6(-K_{X}.C_{0})+32) \text{ and } \\
\deg \a&=\frac{1}{6}(-(-K_{X})^{3}+2g(T)-18g(C)+6(-K_{X}.C_{0})+16). 
\end{align*}
\item The proper transform $E_{Q} \subset Q$ of $E=\Exc(\mu)$ contains $T$. 
Moreover, it holds that $-K_{Q}=2E_{Q}-q^{\ast}\a$.
\end{enumerate}
\end{prop}
\begin{proof}
(1) 
Fix a point $t \in C$. 
Then $\wt{X}_{t}=\wt{\vp}^{-1}(t)$ is the blow-up of $X_{t}=\vp^{-1}(t)$ at a smooth point. 
For each $t \in C$, $-K_{X_{t}}$ is very ample and hence $-K_{\wt{X}_{t}}$ is globally generated and big. 
The proof of (1) is completed by showing that $\psi_{X}$ does not contract a prime divisor. 

\begin{claim}\label{claim-(-2)}
If $X_{t}$ is smooth, then the number of $(-2)$-curves on $\wt{X}_{t}$ is at most $2$. 
Moreover, if $(-2)$-curves exist, then all of $(-2)$-curves are disjoint. 
\end{claim}
\begin{proof}
Let $\mu_{t} \colon \wt{X}_{t} \to X_{t}$ denote the blow-up and $E_{t}=\Exc(\mu_{t})$. 
Let $\e \colon X_{t} \to \P^{2}$ be the blow-up at three points, $h$ the pull-back of a line by $\e$ and $e_{1},e_{2},e_{3}$ the exceptional divisors of $\e$. 
Let $C$ be a $(-2)$-curve on $\wt{X}_{t}$. 
Then there exists $a \in \Z_{\geq 0}$, $b_{j} \in \Z$ and $m \in \Z_{> 0}$ such that 
$\mu_{t\ast}C \sim ah-\sum_{j=1}^{3} b_{j}e_{j}$ and $\mu_{t}^{\ast}\mu_{t\ast}C=C+mE_{t}$. 
Since $0=E_{t}.(\mu_{t}^{\ast}\mu_{t\ast}C)=E_{t}.C-m$, 
it holds that $(\mu_{t\ast}C)^{2}=\mu_{t}^{\ast}\mu_{t\ast}C.C=C^{2}+mE_{t}.C=m^{2}-2$ and $-K_{X_{t}}.(\mu_{t\ast}C)=(-K_{\wt{X}_{t}}+E_{t}).C=E_{t}.C=m$. 
Hence we obtain equalities $m=3a-\sum_{j=1}^{3} b_{j}$ and $m^{2}-2=a^{2}-\sum_{j=1}^{3} b_{j}^{2}$. The Cauchy-Schwartz inequality 
$3\sum_{j=1}^{3} b_{j}^{2} \geq  \left( \sum_{j=1}^{3} b_{j} \right)^{2}$ 
implies that 
$3(a^{2}-m^{2}+2) \geq (3a-m)^{2}$, which implies $m=1$. 
Therefore, every $(-2)$-curve on $X_{t}$ is the proper transform of a $(-1)$-curve in $X_{t}$ passing through $C_{0} \cap X_{t}$, which implies the assertion. 
\end{proof}

Assume that $\psi_{X}$ contracts a divisor $G$. 
Since $\Pic(\wt{X})=\Z[-K_{\wt{X}}] \oplus \Z[E] \oplus \wt{\vp}^{\ast}\Pic(C)$, 
there exists $x,y \in \Z$ such that $G \equiv_{C} x(-K_{\wt{X}})+yE$. 
Note that $\psi|_{\wt{X}_{t}} \colon \wt{X}_{t} \to \t(\wt{X}_{t})$ is the contraction of the $(-2)$-curves for general $t$. 
Let $n$ be the number of the $(-2)$-curves on $\wt{X}_{t}$. 
Then Claim~\ref{claim-(-2)} implies that $0=-K_{\wt{X}_{t}}G|_{\wt{X}_{t}}=5x+y$, $-2n=(G|_{\wt{X}_{t}})^{2}=5x^{2}+2xy-y^{2}$ and $n \in \{1,2\}$. 
These equalities gives $x=\pm \sqrt{\frac{n}{15}} \text{ and } y=\mp \sqrt{\frac{5n}{3}}$, which is a contradiction. 
We complete the proof of (1). 

(2) 
Note that $\rho(\wt{Q})=3$. Since $R^{1}\wt{q}_{\ast}\mc O_{\wt{Q}}=0$ and $\wt{Q}$ and $C$ are smooth, we have $N_{1}(\wt{Q}/C)=2$ by \cite[Corollary~(12.1.5)]{KM92}. 
Hence $\ol{\NE}(\wt{Q}/C) \subset N_{1}(\wt{Q}/C)=\R^{2}$ is spanned by two rays. 
If $\chi$ is identity, then $-K_{\wt{Q}}$ is ample over $C$ and the assertion holds by the relative contraction theorem. 
If $\chi$ is a flop, then $\psi_{Q}$ is the contraction of an extremal ray, say $R_{1}$. 
If $R_{2}$ denotes the another ray, then we have $K_{\wt{Q}}.R_{2}<0$ since a general $\wt{q}$-fiber is a del Pezzo surface. 
We complete the proof of (2). 

(3) 
Since $\rho(\wt{Q})=3$, it holds that $\rho(Q)=2$ and $\dim Q \geq 2$. 
If $\dim Q=2$, then $Q$ is a $\P^{1}$-bundle over $C$. 
Let $s \subset Q$ be a section of $Q \to C$ and set $G:=q^{\ast}s$. 
Then there exists $x,y \in \Z$ such that $G\equiv_{C}x(-K_{\wt{Q}})+yE_{\wt{Q}}$. 
Then for a general $\wt{q}$-fiber $F_{\wt{Q}}$, we have $G^{2}F_{\wt{Q}}=0$ and $-K_{\wt{Q}}GF_{\wt{Q}}=2$, which implies $5x^{2}+2xy-y^{2}=0$ and $5x+y=2$. 
Solving the above equalities, we get $x=\pm \frac{6+\sqrt{6}}{15}$ and $y=\mp \sqrt{\frac{2}{3}}$, which contradicts $x,y \in \Z$. 

Therefore, we obtain $\dim Q=3$ and hence $\s$ is divisorial. 
Set $G=\Exc(\s)$. Since every $\wt{\vp}$-fiber is integral and $\chi$ is isomorphic in codimension $1$, every $\wt{q}$-fiber is integral. 
Hence $T:=\s(G)$ is not contained in any $q$-fiber. 
By \cite[Theorem~(3.3)]{Mori82}, $T$ is a non-singular curve, $Q$ is smooth and $\s$ is the blow-up of $Q$ along $T$. 
Set $m:=\deg(q|_{T})$. 
Then there exists $x,y \in \Z$ and $\a \in \Pic(C)$ such that 
$G \sim x(-K_{\wt{Q}})+yE_{\wt{Q}}+\wt{q}^{\ast}\a$. 
For a general $\wt{q}$-fiber $F_{\wt{Q}}$, $\s|_{F_{\wt{Q}}}$ is a contraction of disjoint finitely many $(-1)$-curves. 
Hence we have $G^{2}.F=-m$ and $-K_{\wt{Q}}.F_{\wt{Q}}.G=m$, which implies $5x^{2}+2xy-y^{2}=-m$ and $5x+y=m$. 
Hence we obtain $x=\frac{\pm \sqrt{6m^{2}+30m}+6m}{30}$ and $y = \mp \sqrt{\frac{m(m+5)}{6}}$, which implies $m \in \{1,3\}$. 

When $m=1$, we have $(x,y)=(0,1)$ since $x,y \in \Z$. 
Then $G|_{F_{\wt{Q}}} \equiv E|_{F_{\wt{Q}}}$ holds for general $F_{\wt{Q}}$ and hence we obtain $G=E_{\wt{Q}}$. 
If $\chi$ is an identity, then it is a contradiction since $\mu$ is the contraction of another ray. 
If $\chi$ is a flop, then $E$ is ample over $W$ and hence $-E$ is ample over $W$. However, $G$ is ample over $W$, which is a contradiction. 

Therefore, we obtain $m=3$ and hence $q$ is a quadric fibration. 
In this case, we have $(x,y)=(1,-2)$ since $x,y \in \Z$. 

(4) 
The assertion directly follows from solving the following equations:
\begin{align*}
(-K_{\wt{Q}})^{3}
=&(-K_{Q})^{3}-2(-K_{Q}.T)+2g(T)-2 \\
=&(-K_{X})^{3}-2(-K_{X}.C_{0})+2g(C)-2,   \\
 (-K_{\wt{Q}})^{2}.G
=&-K_{Q}.T+2-2g(T)
 \\
=&((-K_{X})^{3}-2(-K_{X}.C_{0})+2g(C)-2) \\
&-2(-K_{X}.C_{0}+2-2g(C)) + 5\deg \a, \text{ and } \\
(-K_{\wt{Q}})G^{2}
=&2g(T)-2  \\
=&((-K_{X})^{3}-2(-K_{X}.C_{0}) +2g(C)-2)\\
&+4(2g(C)-2) -4(-K_{X}.C_{0}+2-2g(C))+6\deg \a. 
\end{align*}

(5)
By (3), we obtain $-K_{Q} \sim 2E_{Q}-q^{\ast}\a$. 
For a general point $t \in C$, $E_{Q} \cap Q_{t}$ is a hyperplane section of $Q_{t}$. 
Since $E_{\wt{Q}} \cap \wt{Q}_{t}$ is a $(-1)$-curve in the $\wt{q}$-fiber $\wt{Q}_{t}$, $E_{Q} \cap Q_{t}$ passes through the three points $T \cap Q_{t}$. Thus $E_{Q}$ contains $T$.  
\end{proof}

\begin{defi}\label{defi-68}
For a sextic del Pezzo fibration $\vp \colon X \to C$ with a section $C_{0}$, 
the pair $(q \colon Q \to C,T)$ as in Proposition~\ref{prop-2ray68} is called the \emph{relative double projection} of the pair $(\vp \colon X \to C,C_{0})$. 
\end{defi}

\section{Preliminaries}

\subsection{The double covers associated to quadric fibrations}

In this section, we confirm some basic properties of a quadric fibration $q \colon Q \to C$. 
We refer to Subsection~\ref{NC}~(13) for the definition of a quadric fibration in this paper. 
The following lemma is well-known for experts and the author will give a proof of it for readers' convenience. 

\begin{lem}[{cf. \cite[Proposition~1.16]{ABB14} and \cite[Section~3.1]{LT16}}]\label{lem-89}
Let $q \colon Q \to C$ be a quadric fibration and $s$ a $q$-section. 
Let $f \colon \wt{Q}=\Bl_{s}Q \to Q$ be the blow-up of $Q$ along $s$. 
Then there exists a divisorial contraction $g \colon \wt{Q} \to P$ over $C$ such that $p \colon P \to C$ is a $\P^{2}$-bundle and $g$ is the blow-up along a smooth curve $B \subset P$. 
\[\xymatrix{
&\Bl_{s}Q=\wt{Q}=\Bl_{B}P \ar[rd]^{g} \ar[ld]_{f}&  \\
Q\ar[rd]_{q}&&P \ar[ld]^{p} \\
&C&
}\label{dia-89}
\]
Moreover, the morphism $q_{B}:=p|_{B} \colon B \to C$ is a finite morphism of degree $2$ with the following conditions:
\begin{enumerate}
\item The branched locus of $q_{B}$ coincides with the closed set 
\[\Sigma:=\{t \in C \mid Q_{t}=q^{-1}(t) \text{ is singular } \} \]
with the reduced induced closed subscheme structure. 
\item $\mc J(Q)$ and $\Jac(B)$ are isomorphic as complex tori. 
If $C=\P^{1}$, then these are isomorphic as principally polarized abelian varieties. 
\item It holds that $(-K_{Q})^{3}=40-(8g(B)+32g(C))$. 
\end{enumerate}
\end{lem}

\begin{proof}
Let $H_{Q}$ be a $q$-ample divisor $H_{Q}$ with $K_{Q}+2H_{Q} \sim_{C} 0$. 
For every point $t \in C$, 
$s_{t} \in Q_{t}$ is a smooth point of $Q_{t}$ since $s$ is a $q$-section. 
Then $|(f^{\ast}H_{Q}-\Exc(f))|_{\wt{Q}_{t}}|$ is base point free and $h^{0}(\mc O_{\wt{Q}_{t}}(f^{\ast}H_{Q}-\Exc(f)))=3$ for every point $t$.
Hence $-K_{\wt{Q}}$ is ample over $C$ and $|f^{\ast}H_{Q}-\Exc(f)|$ gives a morphism $g \colon \wt{Q} \to P$ over $C$, where $P$ is a $\P^{2}$-bundle over $C$. 
Since $-K_{\wt{Q}}$ is ample over $C$, $g$ is the contraction of an extremal ray. 
For each point $t \in C$, we can see the following. 
\begin{itemize}
\item[(a)] When $Q_{t}$ is smooth, $g|_{\wt{Q}_{t}}$ is the blow-up of $\P^{2}$ at reduced two points. 
\item[(b)] When $Q_{t}$ is the quadric cone, $g|_{\wt{Q}_{t}}$ is the blow-up of $\P^{2}$ at a non-reduced $0$-dimensional closed subscheme of length $2$. 
\end{itemize}
Then \cite[Theorem~(3.3)]{Mori82} implies that $g$ is the blow-up of $P$ along a smooth curve $B$. By (a) and (b), $p|_{B} \colon B \to C$ is a double covering. 

(1) This assertion follows from (a) and (b). 

(2) Let $E$ denote the $f$-exceptional divisor and $i \colon E \hra \wt{Q}$ be the closed immersion. Then the homomorphism
\[H^{3}(Q,\Z) \oplus H^{1}(s,\Z) \ni (\a,\b) \mapsto f^{\ast}\a + i_{!}(f|_{E})^{\ast}\b \in H^{3}(\wt{Q},\Z) \]
attains the isomorphism of the Hodge structures (cf. \cite[I, Theorem~7.31]{VoisinBook}). 
Let $E_{P}=g(E)$ and $k \colon E_{P} \hra P$ the inclusion. 
Then $E \to E_{P}$ is isomorphism and $E_{P} \subset P$ is a sub $\P^{1}$-bundle.  The Leray-Hirsch theorem implies that 
\[\begin{array}{ccccc}
H^{1}(C,\Z) & \longrightarrow &H^{1}(E_{P},\Z) & \longrightarrow & H^{3}(P,\Z)\\
\rotatebox{90}{$\in$} & & \rotatebox{90}{$\in$} & & \rotatebox{90}{$\in$} \\
\a & \longmapsto & (p|_{E_{P}})^{\ast}\a & \longmapsto & k_{!}(p|_{E_{P}})^{\ast}\a=p^{\ast}\a . [E_{P}]
\end{array}\]
attains isomorphisms of the Hodge structures.
Since a diagram 
\[\xymatrix{
H^{1}(s,\Z) \ar[r]^{(f|_{E})^{\ast}} \ar@{=}[d]& H^{1}(E,\Z) \ar[r]^{i_{!}} & H^{3}(\wt{Q},\Z) \\
H^{1}(C,\Z) \ar[r]_{(p|_{E_{P}})^{\ast}}&  H^{1}(E_{P},\Z) \ar[r]_{k_{!}} & H^{3}(P,\Z) \ar[u]^{g^{\ast}}
}\]
is commutative, the cokernel of $i_{!}(f|_{E})^{\ast} \colon H^{1}(s,\Z) \to H^{3}(\wt{Q},\Z)$ is isomorphic to that of $g^{\ast}k_{!}(p|_{E_{P}})^{\ast} \colon H^{1}(C,\Z) \to H^{3}(\wt{Q},\Z)$. 
Thus we obtain the isomorphism of Hodge structures $H^{1}(B,\Z) \simeq H^{3}(Q,\Z)$. 
If $C=\P^{1}$, then the last assertion follows immediately from \cite[Lemma~3.11]{CG72}.

(3) Set $\mc F:=q_{\ast}\mc O(H_{Q})$ and take a divisor $\a$ on $C$ such that $-K_{Q}=2H_{Q}-q^{\ast}\a$. 
Now $Q$ is a member of $|2\xi+\pi^{\ast}L|$ in $\P_{C}(\mc F)$, where $\pi \colon \P_{C}(\mc F) \to C$ is the projection, $\xi$ is a tautological divisor of $\pi$ and $L$ is a some divisor on $C$. 
By the adjunction formula, we have 
\begin{align}\label{eq-aL}
\mc O_{C}(L)=\mc O_{C}(\a-K_{C}) \otimes \det \mc F^{-1}. 
\end{align}
Take the section $u \in H^{0}(C,\Sym^{2}\mc F \otimes \mc O_{C}(L))$ that corresponds to the member $Q \in |2\xi+\pi^{\ast}L|$. 
Let $\bigcup_{\lambda} U_{\lambda}$ be a open covering of $C$ and take a trivialization $\mc F|_{U_{\lambda}}=\bigoplus_{i=1}^{4} \mc O_{U_{\lambda}} \cdot e_{\lambda}^{i}$ for all $\lambda$. 
Then there exists a symmetric matrix $(u_{\lambda}^{ij})$ such that $u_{\lambda}=\sum s_{\lambda}^{ij}e_{\lambda}^{i}e_{\lambda}^{j}$. 
By gluing the sections $\det (u_{\lambda}^{ij})_{i,j}$ on $U_{\lambda}$, we obtain the global section $\tilde{u}$ of $(\det \mc F)^{\otimes 2} \otimes \mc O(4L)$. 
Note that $\Sigma=\{t \in C \mid q^{-1}(t) \text{ is singular}\}$ is the set-theoretic zero set of $\tilde{u}$. 
Since $Q$ is non-singular, the zero scheme of $\tilde{u}$ is reduced. 
Thus we obtain $\mc O_{C}(\Sigma) = (\det \mc F)^{\otimes 2}\otimes \mc O(4L)$. 
Since the double covering $B \to C$ is branched along $\Sigma$, 
we obtain 
\begin{align}\label{eq-2adetF}
\omega_{B}=q_{B}^{\ast}(\omega_{C} \otimes \mc O_{C}(\det \mc F+2L))=q_{B}^{\ast}( \mc O(2\a) \otimes \omega_{C}^{-1} \otimes \det \mc F^{-1})
\end{align}
by the Hurwitz formula and (\ref{eq-aL}). 
On the other hand, it holds that $-K_{Q} \sim 2\xi|_{Q}-q^{\ast}(K_{C}+\det \mc F-\a)$ by the adjunction formula. 
Hence we have $(-K_{Q})^{3} = (2\xi-\pi^{\ast}(K_{C}+\det \mc F-\a))^{3}(2\xi+L) =24(2-2g(C))-(8 \deg \mc F-16 \deg \a) =40-(32g(C)+8g(B))$, which is the assertion. 
\end{proof}

\begin{defi}\label{defi-doublecov}
Let $q \colon Q \to C$ be a quadric fibration. 
Take a $q$-section $C_{0}$. 
Then we obtain the double covering $q_{B} \colon B \to C$ as in Lemma~\ref{lem-89}. 
Since Lemma~\ref{lem-89}~(1), this double covering is independent of the choice of $C_{0}$. 
We call this $q_{B} \colon B \to C$ \emph{the double covering associated to $q$}. 
\end{defi}

\begin{rem}\label{rem-quadfib}
(i) Our construction of the $\P^{2}$-bundle $P$ with the bisection $B$ is essentially the same as in \cite[Theorem~(2.7.3)]{DSouza88} and \cite[Proposition~1.16]{ABB14}. 

(ii) Let $\Hilb_{t+1}(Q/C) \to C$ denote the relative Hilbert scheme of lines and $\Hilb_{t+1}(Q/C) \to B' \to C$ the Stein factorization. 
Then $\Hilb_{t+1}(Q/C)$ and $B'$ are smooth by \cite[Corollary~2.3]{Kuznetsov14}. 
On the other hand, in the diagram (\ref{dia-89}), the $\P^{1}$-bundle $\Exc(g) \to B$ is a family of lines of $Q \to C$. 
By the universal property, there is a natural morphism $B \to B'$, which is bijective by Lemma~\ref{lem-89}~(1) and hence isomorphic by Zariski main theorem. 
Hence $\Exc(g) \to B$ coincides with the relative Hilbert scheme of lines of $Q \to C$. 
From this viewpoint, Lemma~\ref{lem-89}~(2) was also proved by \cite[Section~3.1]{LT16}. 

(iii) Assuming $\rho(Q)=2$ is essential for Lemma~\ref{lem-89}~(3). 
For instance, if we consider $Q=\P^{1} \times \P^{1} \times \P^{1}$ as a family of quadric surface over $\P^{1}$, then $(-K_{Q})^{3}=48$, 
which violates $(-K_{Q})^{3} \leq 40$. 
\end{rem}

\begin{lem}\label{lem-indep}
Let $\vp \colon X \to C$ be a sextic del Pezzo fibration and $C_{0}$ a $\vp$-section. 
Let $(q \colon Q \to C,T)$ be the relative double projection as in Definition~\ref{defi-68}. 
Let $\vp_{B} \colon B \to C$ be the double covering associated to $q$ and $\vp_{T}:=q|_{T}$. 

Then $\vp_{B}$ (resp. $\vp_{T}$) is isomorphic over the base $C$ to the normalization of the finite part of the Stein factorization of the relative Hilbert scheme of twisted cubics (resp. conics) $\Hilb_{3t+1}(X/C) \to C$ (resp. $\Hilb_{2t+1}(X/C) \to C$). 
In particular, $\vp_{B}$ and $\vp_{T}$ are independent of the choice of $\vp$-sections. 
\end{lem}
\begin{proof}
We fix a section $C_{0}$ and take the diagram (\ref{dia-2ray68}) as in Proposition~\ref{prop-2ray68}. 
Let $U \subset C$ be an open set such that $X_{t}$ is smooth and $-K_{\wt{X}_{t}}$ is ample for every $t \in U$. 
By Proposition~\ref{prop-2ray68}~(1), $U$ is not empty and the birational map $\chi \colon \wt{X} \dra \wt{Q}$ is isomorphic over $U$. 
Set $X_{U}:=\vp^{-1}(U)$, $Q_{U}:=q^{-1}(U)$, $\wt{X}_{U}:=\wt{\vp}^{-1}(U) \simeq \wt{q}^{-1}(U)$, $G_{U}:=\wt{q}^{-1}(U) \cap G$ and $T_{U}=\vp_{T}^{-1}(U)$. 
We set the morphisms as in the following diagram:
\begin{align}
\xymatrix{
&\ar[ld]\wt{X}_{U} \ar@{}[r]|{\supset} \ar[rd] \ar[dd]_{\wt{\vp}_{U}}&G_{U} \ar[rd]\\
X_{U}\ar[rd]_{\vp_{U}}&&Q_{U} \ar@{}[r]|{\supset} \ar[ld]^{q_{U}}& T_{U}  \\
&U&
}
\label{dia-2ray68U}
\end{align}
Note that $\wt{\vp}_{U},\vp_{U}$ and $q_{U}$ are isotrivial. 

First, we show the assertion for $T$. 
Composing the morphisms in the diagram (\ref{dia-2ray68U}), we have a morphism 
$e_{U} \colon G_{U} \to X_{U}$ over $U$. 
Then we can regard the $\P^{1}$-bundle $G_{U} \to T_{U}$ as a family of conics in the fibers of $X_{U} \to U$. 
By the universal property of Hilbert schemes, there is a natural morphism $T_{U} \to \Hilb_{2t+1}(X_{U}/U)$ over $U$. 
Since $\vp_{U}$ is isotrivial, the morphism $\Hilb_{2t+1}(X_{U}/U) \to U$ factors a triple cover $T'_{U}$ over $U$ as the Stein factorization. 
Hence $T_{U} \to \Hilb_{2t+1}(X_{U}/U)$ is a section of $\Hilb_{2t+1}(X_{U}/U) \to T'_{U}$, which implies that $T_{U} \simeq T'_{U}$. 
Hence, when $T' \to C$ denotes the finite part of Stein factorization of $\Hilb_{2t+1}(X/C) \to C$, the normalization of $T'$ is isomorphic to $T$ over $C$. 

Next, we show the assertion for $B$. 
As we have seen in Remark~\ref{rem-quadfib}~(ii), if we set $B_{U}=\vp_{B}^{-1}(U)$, then $\Hilb_{t+1}(Q_{U}/U) \to U$ factors through $B_{U}$ such that $\Hilb_{t+1}(Q_{U}/U) \to B_{U}$ is a $\P^{1}$-bundle and $B_{U} \to U$ is an \'etale double covering. 
Let $R_{U} \to \Hilb_{t+1}(Q_{U}/U)$ be the universal family of the relative Hilbert scheme of lines. 
Set $R'_{U}:=\wt{X}_{U} \times_{Q_{U}} R_{U}$. 
The flat family $R'_{U} \to \Hilb_{t+1}(Q_{U}/U)$ parametrizes twisted cubics on $X_{U}$ over $U$ by the evaluation $R'_{U} \to X_{U}$. 
Hence the universal property gives a morphism $\Hilb_{t+1}(Q_{U}/U) \to \Hilb_{3t+1}(X_{U}/U)$ over $U$. 
Let $B'_{U} \to U$ be the finite part of the Stein factorization of $\Hilb_{3t+1}(X_{U}/U) \to U$. 
Then we get a morphism $B_{U} \to B_{U}'$. 
Thanks to isotriviality, we can check that $B_{U}'$ is an \'etale double covering over $U$ and $B_{U} \to B_{U}'$ is bijective. Hence $B_{U} \to B_{U}'$ is an isomorphism. 
Thus $B$ is the normalization of the finite part of the Stein factorization $\Hilb_{3t+1}(X/C) \to C$. 
We complete the proof. 
\end{proof}

\begin{defi}\label{defi-BT}
Let $\vp \colon X \to C$ be a sextic del Pezzo fibration. 
We define $\vp_{B} \colon B \to C$ and $\vp_{T} \colon T \to C$ as in the settings of Lemma~\ref{lem-indep}. 
We call $\vp_{B}$ (resp. $\vp_{T}$) the \emph{associated double} (resp. \emph{triple}) \emph{covering} to $\vp$. 
\end{defi}

\subsection{Proof of Theorem~\ref{mainthm-inv}}
We use the same notation as in Proposition~\ref{prop-2ray68}. 

(1) As in Proposition~\ref{prop-2ray68}~(2), $\wt{X}$ is isomorphic to $\wt{Q}$ or its flop. 
By \cite[Corollary~4.12]{Kollar89}, $\mc J(\wt{X})$ is isomorphic to $\mc J(\wt{Q})$ as complex tori. 
Moreover, when $C \simeq \P^{1}$, the isomorphism $\mc J(\wt{X}) \simeq \mc J(\wt{Q})$ preserves the polarizations by \cite[Remark~4.13~(1)]{Kollar89}. 
Therefore, the assertion follows from Lemma~\ref{lem-89}~(2) and \cite[Lemma~3.11]{CG72}. 

(2) The assertion from a direct calculation using the formulae in Lemma~\ref{lem-89}~(3) and Proposition~\ref{prop-2ray68}~(4). 

We complete the proof of Theorem~\ref{mainthm-inv}. \qed

\subsection{Characterizations of some nef bundles on quadric surfaces}
In this subsection, we will obtain numerical characterizations of some nef vector bundles on an irreducible and reduced quadric surface. 
We will use these results for proving Theorem~\ref{mainthm-P22}. 

Until the end of this subsection, we work over the following setting:
\begin{itemize}
\item $Q$ denotes an irreducible and reduced quadric surface in $\P^{3}$. We refer the definitions of $\mc O_{Q}(1)$ and $\ms S_{Q}$ to Subsection~\ref{NC}~(11). 
\item When $Q$ is smooth, we set $\pi \colon \F:=\P_{\P^{1}}(\mc O(1) \oplus \mc O(1)) \to \P^{1}$ and let $\s \colon \F \to Q$ be the identity. 
\item When $Q$ is singular, we set $\pi \colon \F:=\P_{\P^{1}}(\mc O \oplus \mc O(2)) \to \P^{1}$ and let $\s \colon \F \to Q$ be the contraction of the $(-2)$-curve. 
\item Let $h$ (resp. $f$) be a tautological divisor (resp. a fiber) of the $\P^{1}$-bundle $\F \to \P^{1}$. 
Note that $\s^{\ast}\mc O_{Q}(1)=\mc O_{\F}(h)$ for each case. 
\end{itemize}

\begin{rem}
$\ms S_{Q}$ is a globally generated vector bundle of rank $2$ with $\det \ms S_{Q}=\mc O_{Q}(1)$ and $c_{2}(\ms S_{Q})=1$. When $Q$ is smooth, it is known that $\ms S_{Q} \simeq \mc O_{Q}(h-f) \oplus \mc O_{Q}(f)$ (cf. \cite{SW90,Ohno14}). 
\end{rem}
The aim is to give the following characterization of the bundle $\ms S_{Q} \oplus \mc O_{Q}(1)$. 
\begin{prop}\label{prop-Spi2}
Let $\mc F_{Q}$ be a nef vector bundle on $Q$ of rank $3$ such that $\det \mc F_{Q}=\mc O_{Q}(2)$, $c_{2}(\mc F)=3$ and $h^{0}(\mc F_{Q}^{\vee})=0$. 
Then $\mc F_{Q}$ is isomorphic to $\ms S_{Q} \oplus \mc O_{Q}(1)$. 
\end{prop}

First, we confirm the following lemma. 
\begin{lem}\label{lem-Spi0}
Let $\mc E$ be a rank $2$ nef vector bundle on $Q$ with $\det \mc E \simeq \mc O_{Q}(1)$. 
\begin{enumerate}
\item If $c_{2}(\mc E)=0$, then $\mc E \simeq \mc O_{Q}(1) \oplus \mc O_{Q}$. 
\item If $c_{2}(\mc E)=1$, then $\mc E \simeq \ms S_{Q}$. 
\end{enumerate}
\end{lem}
\begin{proof}
The assertion was already proved in \cite{SW90,Ohno14} if $Q$ is non-singular. 
We attain a proof even if $Q$ is singular. 

(1) Suppose that $c_{2}(\mc E)=0$. 
The Hirzebruch-Riemann-Roch theorem and the Leray spectral sequence implies $\chi(\mc E(-1))=\chi(\s^{\ast}\mc E(-h))=1$. 
Note that $h^{2}(Q,\mc E(-1))=h^{0}(\mc E^{\vee}(-1))=h^{0}(\mc E(-2)) \leq h^{0}(\mc E(-1))$ by the Serre duality. 
Thus we have $h^{0}(\mc E(-1))>0$ and hence an injection $\mc O(1) \to \mc E$. 
By the same arguments as in \cite[Proposition~5.2]{Ohno14}, we have $\mc E=\mc O \oplus \mc O_{Q}(1)$. 

(2) Suppose that $c_{2}(\mc E)=1$. 
Set $\wt{\mc E}:=\s^{\ast}\mc E$. 
Then it holds that $\det \wt{\mc E} = \mc O(h)$ and $c_{2}(\wt{\mc E})=1$. 

\begin{claim}\label{claim1-Spi1}
$h^{0}(\wt{\mc E}(-h))=0$.
\end{claim}

\begin{proof}
If $h^{0}(\wt{\mc E}(-h))>0$, then we have $\wt{\mc E} \simeq \mc O \oplus \mc O(h)$ by the same argument as in the proof of Lemma~\ref{lem-Spi0}~(1). 
However, it contradicts $c_{2}(\wt{\mc E})=1$.
\end{proof}

We have $\chi(\wt{\mc E}(-h+f))=1$ by the Hirzebruch-Riemann-Roch theorem and $h^{2}(\wt{\mc E}(-h+f))=h^{0}(\wt{\mc E}^{\vee}(-h-f))=h^{0}(\wt{\mc E}(-2h-f))$ by the Serre duality. 
Since $\wt{\mc E}$ is nef and $c_{1}(\wt{\mc E})=h$, we have $\wt{\mc E}|_{l}=\mc O_{\P^{1}} \oplus \mc O_{\P^{1}}(1)$ for every $l \in |f|$ and hence $h^{0}(\wt{\mc E}(-2h-f))=0$. 
Thus we have $h^{0}(\wt{\mc E}(-h+f))>0$. 
Hence there is an injection $\a \colon \mc O(h-f) \to \wt{\mc E}$. 
If there exists a fiber $l \in |f|$ such that $\a|_{l}=0$, then we obtain a non-zero map $\mc O(h) \to \wt{\mc E}$, which contradicts Claim~\ref{claim1-Spi1}. 
Therefore, for every $l \in |f|$, $\a|_{l}$ is non-zero map 
from $\mc O(h-f)|_{l} \simeq \mc O_{\P^{1}}(1)$ to $\wt{\mc E}|_{l} \simeq \mc O_{\P^{1}} \oplus \mc O_{\P^{1}}(1)$. 
Then $\Cok \a|_{l}=\mc O_{l}$ and hence $\Cok \a$ is the pull-back of a line bundle on $\P^{1}$. 
Since $\det \wt{\mc E}=\mc O(h)$, we obtain an exact sequence
\begin{align}
\label{eq1-Spi1}0 \to \mc O(h-f) \to \wt{\mc E} \to \mc O(f) \to 0. 
\end{align}
If $Q$ is smooth, then this sequence splits and we obtain $\wt{\mc E} \simeq \mc O(h-f) \oplus \mc O(f) \simeq \ms S_{Q}$. 
If $Q$ is singular, then this sequence does not splits since $\wt{\mc E}$ is nef and $\mc O(h-f)$ is not nef. 
The vector bundles fitting into the exact sequence (\ref{eq1-Spi1}) which does not split are unique up to isomorphism since 
$\Ext^{1}(\mc O(f),\mc O(h-f))=H^{1}(\F_{2},\mc O(h-2f))=\C$. 
By exactly the same argument, $\s^{\ast}\ms S_{Q}$ also fits into the exact sequence (\ref{eq1-Spi1}). 
Hence we obtain $\s^{\ast}\ms S_{Q} \simeq \s^{\ast}\mc E$ and hence $\ms S_{Q} \simeq \mc E$. 
\end{proof}

\begin{lem}\label{lem1-Spi2}
Let $\mc F$ be a rank $3$ nef vector bundle on $Q$ with $\det \mc F=\mc O_{\F}(2h)$ and $c_{2}(\mc F)=3$. 
Then the following assertions follow.
\begin{enumerate}
\item $h^{0}(\mc F(-2h+af))=0$ for any $a \in \Z$. 
\item $\mc F|_{l} \simeq \mc O \oplus \mc O(1)^{2}$ for a general member $l \in |f|$. 
\item $h^{0}(\mc F(-h-f))=0$. 
\item $h^{0}(\mc F(-h))>0$. 
\end{enumerate}
\end{lem}
\begin{proof}

(1) We may assume that $a \geq 0$. 
We will prove (1) by the induction on $a$. 
Assume that $a=0$. 
If $h^{0}(\mc F(-2h)) \neq 0$, it follows that $\mc F = \mc O(2h) \oplus \mc O^{2}$ from \cite[Proposition~5.2]{Ohno14}, which contradicts $c_{2}(\mc F)=3$. 
Thus we have $h^{0}(\mc F(-2h))=0$. 
Assume that $a>0$. 
If $h^{0}(\mc F(-2h+af)) \neq 0$, then there exists an injection 
$\iota \colon \mc O(2h-af) \to \mc F$. 
By the assumption of the induction $h^{0}(\mc F(-2h+(a-1)f))=0$, for every member $l \in |f|$, we have $\iota|_{l} \neq 0$. Therefore, we obtain an exact sequence 
$0 \to \mc O(2h-af)|_{l} \to \mc F|_{l} \to \Cok \iota|_{l} \to 0$. 
Since $\mc F|_{l}$ is a nef vector bundle with $c_{1}=2$, 
it holds that $\mc F|_{l}=\mc O(2) \oplus \mc O^{2}$ and $\Cok \iota|_{l}=\mc O_{l}^{2}$. 
It implies that the natural morphism $\pi^{\ast}\pi_{\ast}\Cok \iota \to \Cok \iota$ is isomorphic and thus we obtain an exact sequence 
$0 \to \mc O(2h-af) \to \mc F \to \mc O(a_{1}f) \oplus \mc O((a-a_{1})f) \to 0$ 
for some $a_{1} \in \Z$. 
Hence we have $c_{2}(\mc F)=2a$, which contradicts $c_{2}(\mc F)=3$. 

(2) Assume the contrary. 
Then we obtain $\mc F|_{l} \simeq \mc O^{2} \oplus \mc O(2)$ for any member $l \in |f|$ by the upper semicontinuity. 
Then $\pi_{\ast}\mc F(-2h)$ is a line bundle on $\P^{1}$ and we have a non-zero map $\pi^{\ast}\pi_{\ast}\mc F(-2h) \to \mc F(-2h)$. 
It implies that $h^{0}(\mc F(2h-af)) \neq 0$ for some $a \in \Z$, which contradicts (1). 

(3) If $h^{0}(\mc F(-h-f)) \neq 0$, then we have an injection $\a \colon \mc O(h+f) \to \mc F$. 
When $Q$ is smooth, we consider another ruling $h-f$. 
For every $l' \in |h-f|$, $\a|_{l'} \colon \mc O(h+f)|_{l'} \to \mc F|_{l'}$ is injective by (1). 
Since $\mc O(h+f)|_{l'} \simeq \mc O_{\P^{1}}(2)$ and $c_{1}(\mc F|_{l'})=2$, 
we have $\Cok \a|_{l'} \simeq \mc O_{l'}^{2}$ and hence $\Cok \a$ is a nef bundle. 
Then $\Cok \a \simeq \mc O \oplus \mc O(h-f)$ holds, which contradicts $c_{2}(\mc F)=3$. 
When $Q$ is singular, we take the restriction $\a|_{C_{0}}$ on the $(-2)$-curve $C_{0}$. 
By (1), $\a|_{C_{0}}$ is a non-zero map from $\mc O(h+f)|_{C_{0}} \simeq \mc O(1)$ to $\mc F|_{C_{0}} \simeq \mc O^{2}$, which is a contradiction. 
Therefore, we have $h^{0}(\mc F(-h-f))=0$.

(4) We have $\chi(\mc F(-h))=1$ by the Hirzebruch-Riemann-Roch theorem and $h^{2}(\mc F(-h))=h^{0}(\mc F^{\vee}(-h))$ by the Serre duality. 
Since $\mc F^{\vee}|_{l}=\mc O \oplus \mc O(-2)$ or $\mc O(-1)^{2}$ for any $l \in |f|$, $h^{0}(\mc F^{\vee}(-h))=0$ and hence $h^{0}(\mc F(-h)) >0$. 
\end{proof}

\begin{proof}[Proof of Proposition~\ref{prop-Spi2}]
Set $\mc F:=\s^{\ast}\mc F_{Q}$. 
Lemma~\ref{lem1-Spi2}~(4) attains an injection $\a \colon \mc O(h) \to \mc F$. 
Letting $\mc E=\Cok \a$, we have the following exact sequence: 
\begin{align}
\label{eq1-Spi2}0 \to \mc O(h) \mathop{\to}^{\a} \mc F \to \mc E \to 0.
\end{align}
By Lemma~\ref{lem1-Spi2}~(3), $\a|_{l}$ is non-zero for every $l \in |f|$. 
By Lemma~\ref{lem1-Spi2}~(2), the cokernel of $\a|_{l} \colon \mc O(h)|_{l} \to \mc F|_{l}$ is locally free for general members $l \in |f|$. 
Then $\mc E$ is locally free in codimension $1$ and hence torsion free by \cite[Lemma~5.4]{Ohno14}. 
Hence we have the following exact sequence: 
\begin{align}
\label{eq2-Spi2}0 \to \mc E \mathop{\to}^{\iota} \mc E^{\vee\vee} \to \mc T \to 0. 
\end{align}
Here $\mc T$ denotes the cokernel of the natural injection $\iota \colon \mc E \to \mc E^{\vee\vee}$. 
Since $\mc E$ is locally free in codimension $1$, $\Supp \mc T$ is 0-dimensional or empty.
Note that $\mc E^{\vee\vee}$ is locally free since $\F$ is a smooth surface. 
By \cite[Lemma~9.1]{Ohno16v7}, $\mc E^{\vee\vee}$ is nef. 
It is clear that $\det \mc E^{\vee\vee}=\det \mc E=\mc O(h)$. 
The Hirzebruch-Riemann-Roch theorem implies that $\chi(\mc F)=8$, $\chi(\mc O(h))=4$ and $\chi(\mc E^{\vee\vee})=5-c_{2}(\mc E^{\vee\vee})$. 
By the exact sequences (\ref{eq1-Spi2}) and (\ref{eq2-Spi2}), we have $\chi(\mc E)=4$ and $5-c_{2}(\mc E^{\vee\vee})=\chi(\mc E^{\vee\vee})=\chi(\mc E)+h^{0}(\mc T) \geq 4$. 
In particular, $c_{2}(\mc E^{\vee\vee})=0$ or $1$. 

If $c_{2}(\mc E^{\vee\vee})=0$, then $\mc E^{\vee\vee} \simeq \mc O(h) \oplus \mc O$ by Lemma~\ref{lem-Spi0}~(1). Hence we have a morphism $\mc F \to \mc O_{\F}$, which is surjective at generic point. 
This contradicts our assumption $h^{0}(\mc F_{Q}^{\vee})=0$. 
Hence we have $c_{2}(\mc E^{\vee\vee})=1$ and $\mc T=0$, which implies $\mc E \simeq \mc E^{\vee\vee} \simeq \ms S_{Q}$ by Lemma~\ref{lem-Spi0}~(2). 
Since $\Ext^{1}(\ms S_{Q},\mc O_{Q}(h))=H^{1}(\ms S_{Q}^{\vee}(h))=H^{1}(\ms S_{Q})=0$, 
the exact sequence (\ref{eq1-Spi2}) splits. 
We complete the proof of Proposition~\ref{prop-Spi2}. 
\end{proof}

\subsection{Certain vector bundles on families of surfaces with multi-sections}

We will use the following theorem for proving Theorem~\ref{mainthm-P22}. 

\begin{thm}\label{thm-univext}
Let $X,Y$ be smooth varieties with $\dim X=\dim Y+2$. 
Let $f \colon X \to Y$ be a flat projective morphism with $f_{\ast}\mc O_{X}=\mc O_{Y}$. 
Let $Z \subset X$ be a locally complete intersection closed subscheme of codimension $2$ and suppose that $f|_{Z} \colon Z \to Y$ is finite with $\deg(f|_{Z}) \geq 2$. 

If $R^{1}f_{\ast}\mc O_{X}$ is locally free and $H^{2}(Y,f_{\ast}\omega_{f} \otimes R^{1}f_{\ast}\mc I_{Z})=0$, 
then there exists a locally free sheaf $\mc F$ such that
\begin{enumerate}
\item $\mc F$ fits into the following exact sequence: 
\[0 \to f^{\ast}((R^{1}f_{\ast}\mc I_{Z})(-K_{Y})) \to \mc F \to \mc I_{Z}(-K_{X}) \to 0, \text{ and }\]
\item there are no surjections $\mc F|_{X_{y}} \to \mc O_{X_{y}}$ for every $y \in Y$. 
\end{enumerate}
\end{thm}

This theorem resembles the Hartshorne-Serre correspondence for a family of surfaces $f \colon X \to Y$ with a multi-section $Z$. 
However, even if $Y$ is point, the Hartshorne-Serre correspondence does not implies Theorem~\ref{thm-univext} since $H^{2}(X,\mc O(K_{X}))=\C \neq 0$. 
To prove Theorem~\ref{thm-univext}, we use relative Ext-sheaves. 
We will prove this theorem in Appendix~\ref{app-relext}.

\section{Extensions to moderate $(\P^{1})^{3}$-fibrations}\label{sec-P13}

We devote this section to the proof of Theorem~\ref{mainthm-P13}. 

\subsection{Double projection from a point on $(\P^{1})^{3}$}\label{subsec-doubleprojP13}
In Subsection~\ref{subsec-doubleprojdP6}, we discussed the double projection from a point on a smooth sextic del Pezzo surface $S$ (see the diagram (\ref{dia-68gen})). 
Since $S$ is always a hyperplane section of $(\P^{1})^{3}$, it is natural to consider the double projection from a point on $(\P^{1})^{3}$. 


Let $S$ be a sextic del Pezzo surface and $x$ a general point on $S$.  
Then we have $\Bl_{x}S=\Bl_{y_{1},y_{2},y_{3}}\Q^{2}$ as in the diagram (\ref{dia-68gen}). 
We consider an embedding $\Q^{2} \hra \P^{3}$ and the blow-up $\Bl_{y_{1},y_{2},y_{3}}\P^{3}$. 
Then there is a flop $\Bl_{y_{1},y_{2},y_{3}}\P^{3} \dra \Bl_{x}(\P^{1})^{3}$: 

\begin{align}\label{dia-P13gen}
\xymatrix{
&\Bl_{x}(\P^{1})^{3} \ar[ld] &\ar@{-->}[l] \Bl_{y_{1},y_{2},y_{3}}\P^{3}\ar[rd]& \\
(\P^{1})^{3}&&&\P^{3}.
}
\end{align}

This birational map $(\P^{1})^{3} \dra \P^{3}$ is the double projection from $x$ on $(\P^{1})^{3}$. 
We will see the detail for this birational transformation (\ref{dia-P13gen}) in the proof of Theorem~\ref{thm-constP13}. 
Under this birational map, the quadric $\Q^{2}$ is proper transformed into $S$ on $(\P^{1})^{3}$ as we will see in the proof of Theorem~\ref{mainthm-P13}. 

\subsection{Moderate $(\P^{1})^{3}$-fibrations}

In the next theorem, we make a Mori fiber space $\vp_{Y} \colon Y \to C$ with smooth total space $Y$ whose smooth fibers are $(\P^{1})^{3}$ by relativizing the inverse of the double projection $(\P^{1})^{3} \dra \P^{3}$ from a point on $(\P^{1})^{3}$. 
We call this $\vp_{Y}$ a \emph{moderate $(\P^{1})^{3}$-fibration} in this paper.

\begin{thm}\label{thm-constP13}
Let $f \colon \F \to C$ be a $\P^{3}$-bundle and $T \subset \F$ be a smooth curve such that $\deg (f|_{T})=3$ and $T_{t}:=(f|_{T})^{-1}(t)$ is non-colinear in $\F_{t}=\P^{3}$ for each $t \in C$. 
\begin{enumerate}
\item There exists a unique sub $\P^{2}$-bundle $\E \subset \F$ containing $T$. 
\item There exists the following diagram: 
\begin{align}
\xymatrix{
&\ar[ld]_{\mu_{\F}}\wt{\F}^{+}\ar[rd]^{\psi_{\F}^{+}}&& \ar@{-->}[ll]_{\Phi}\ar[ld]_{\psi_{\F}} \wt{\F}=\Bl_{T}\F\ar[rd]^{\s_{\F}} \ar[dd]_{\wt{f}} & \\
Y\ar[rd]_{\vp_{Y}}&&\ar[ld]\ol{\F}\ar[rd]&&\F\ar[ld]^{f} \\
&C\ar@{=}[rr]&&C,&
}\label{dia-P13}
\end{align}
where 
\begin{itemize}
\item $\s_{\F} \colon \wt{\F}:=\Bl_{T}\F \to \F$ is the blow-up of $\F$ along $T$ with exceptional divisor $G_{\F}:=\Exc(\s_{\F})$; 
\item $\Phi \colon \wt{\F} \dra \wt{\F}^{+}$ is a family of Atiyah flops over $C$;
\item $Y$ is smooth and $\vp_{Y} \colon Y \to C$ is a Mori fiber space; 
\item $\mu_{\F}$ is the blow-up along a $\vp_{Y}$-section $C_{0}$ and contracts the proper transform $\wt{\E}^{+}$ of $\E$.  
\end{itemize}
\item If we set $G_{Y} \subset Y$ be the proper transform of $G_{\F}$, then it holds that $-K_{Y} \sim 2G_{Y}-\vp_{Y}^{\ast}(K_{C}+\det f_{\ast}\mc O_{\F}(\E))$. 
\item Every smooth $\vp_{Y}$-fiber is isomorphic to $(\P^{1})^{3}$. 
\end{enumerate}
In this paper, we call this Mori fiber space $\vp_{Y} \colon Y \to C$ the \emph{moderate $(\P^{1})^{3}$-fibration} with respect to the pair $(f \colon \F \to C,T)$. 
\end{thm}

\begin{rem}
In the setting of Theorem~\ref{thm-constP13}, for every $t \in C$, 
the scheme $T_{t}=f^{-1}(t) \cap T$ is reduced, or a union of one reduced point and a non-reduced length 2 subscheme, or a curvilinear scheme of length $3$, i.e., $T_{t} \simeq \Spec \C[\e]/(\e^{3})$. 
\end{rem}

\begin{proof}[Proof of Theorem~\ref{thm-constP13}]
We proceed in 4 steps. 

\noindent
\textbf{Step 1.} 
First, we show (1). Let $\mc O_{\F}(1)$ be a tautological line bundle. 
Then $f_{\ast}(\mc O_{\F}(1) \otimes \mc I_{T})$ is a line bundle $\mc L$ on $C$. 
Thus we have $\C=H^{0}(f_{\ast}(\mc O_{\F}(1) \otimes \mc I_{T}) \otimes \mc L^{-1})=H^{0}(\mc O_{\F}(1) \otimes \mc I_{T} \otimes f^{\ast}\mc L^{-1})$. 
Let $\E \in |\mc O_{\F}(1) \otimes \mc I_{T} \otimes f^{\ast}\mc L^{-1}|$ be the unique member. 
By the definition of $\mc L$, $\E \to C$ is a $\P^{2}$-bundle and $\E_{t}$ is the linear span of $T_{t}$ for every $t \in C$, which proves (1). 

Replacing $\mc O_{\F}(1)$, we may assume that  $\mc O_{\F}(1)=\mc O_{\F}(\E)$. 
In this setting, we have $-K_{\F}=4\E-f^{\ast}(K_{C}+\det f_{\ast}\mc O_{\F}(\E))$. 

\noindent
\textbf{Step 2.} 
Let $\s_{\F} \colon \wt{\F} \to \F$ and $\s_{\E} \colon \wt{\E} \to \E$ be the blow-ups along $T$ and $G_{\F} \subset \wt{\F}$ and $G_{\E} \subset \wt{\E}$ the exceptional divisors of $\s_{\F}$ and $\s_{\E}$ respectively. 
Set 
\begin{align}
\label{eq-P13-L}
L_{\wt{\F}}:=\s_{\F}^{\ast}\mc O_{\F}(2)-G_{\F} \text{ and } L_{\wt{\E}}:=L_{\wt{\F}}|_{\wt{\E}}. 
\end{align}
Note that 
\begin{align}
\label{eq-P13-0}
-K_{\wt{\F}}=2L_{\wt{\F}}-(f \circ \s)^{\ast}(K_{C}+\det f_{\ast}\mc O_{\F}(1)) \sim_{C} 2L_{\wt{\F}}.
\end{align}

The following claim is the 2-ray game of $\wt{\E}$ over $C$. 
\begin{claim}\label{claim-P13E}
\begin{enumerate}
\item $\mc O(L_{\wt{\E}})$ is globally generated and big over $C$ and $(f|_{\E} \circ \s_{\E})_{\ast}\mc O(L_{\wt{\E}})$ is a vector bundle of rank $3$. 
\item Set $\ol{\E}=\P_{C}((f|_{\E} \circ \s_{\E})_{\ast}\mc O(L_{\wt{\E}}))$ and let $\psi_{\E} \colon \wt{\E} \to \ol{\E}$ 
denote the morphisms over $C$ defined by $|L_{\wt{\E}}|$. 
Then $\psi_{\E}$ is the blow-up of $\ol{\E}$ along a non-singular curve $\ol{T} \subset \ol{\E}$. Moreover, the composite morphism $\ol{T} \hra \ol{\E} \to C$ is a triple covering. 
\end{enumerate}
\end{claim}
\begin{proof}
(1) For every $t \in C$, there is a smooth conic $C \subset \E_{t}=\P^{2}$ containing $T_{t}$. 
By the exact sequence $0 \to \mc I_{C/\E_{t}} \to \mc I_{T_{t}/\E_{t}} \to \mc I_{T_{t}/C} \to 0$, we see that $\mc O_{\E_{t}}(2) \otimes \mc I_{T_{t}/\E_{t}}$ is globally generated and  
$h^{0}(\mc O_{\E_{t}}(2) \otimes \mc I_{T_{t}/\E_{t}})=3$ for every $t$, which proves (1). 

(2) For a general point $t \in C$, $\ol{\E}_{t} \gets \wt{\E_{t}} \to \E_{t}=\P^{2}$ is nothing but the Cremona involution. 
Thus $\psi_{\E}$ is a birational morphism onto the $\P^{2}$-bundle $\ol{\E} \to C$. 
Since $-K_{\wt{\E}} \sim_{C} L_{\E}+(\s|_{\E})^{\ast}\mc O_{\E}(1)$ is ample over $C$ and $\rho(\wt{\E})=3$, $\psi_{\E}$ is the contraction of an extremal ray. 
Then $\psi_{\E}$ is the blow-up along a non-singular curve $\ol{T}$ by \cite[Theorem~(3.3)]{Mori82}. 
For a general $t \in C$, $\wt{\E}_{t} \to \ol{\E}_{t}=\P^{2}$ is the blow-up at three points. Hence $\ol{T} \to C$ is a triple covering. 
\end{proof}

\noindent
\textbf{Step 3.} Next we play the 2-ray game of $\wt{\F}$ over $C$ by using Claim~\ref{claim-P13E}. 

\begin{claim}\label{claim-P13F}
\begin{enumerate}
\item $\mc O(L_{\wt{\F}})$ is globally generated and big over $C$. 
\item Let $\psi_{\F} \colon \wt{\F} \to \ol{\F}$ denote the Stein factorizations of the morphisms over $C$ defined by $|L_{\wt{\F}}|$. 
Then $\psi_{\F}|_{\wt{\E}}=\psi_{\E}$ and $\Exc(\psi_{\F})=\Exc(\psi_{\E})$. 
\item $\psi_{\F}$ is a family of Atiyah's flopping contraction over $C$. 
\end{enumerate}
\end{claim}

\begin{proof}
(1) For any $t \in C$, we have $(L_{\wt{\F}}|_{\wt{\F}_{t}})^{3}=5$ by a direct calculation. 
Since $\wt{\E} \sim \s^{\ast}\mc O_{\F}(1)-G_{\F}$, we have $\wt{\E}+\s^{\ast}\mc O_{\F}(1) \sim L_{\wt{\F}}$. Then we obtain an exact sequence $0 \to \s^{\ast}\mc O_{\F}(1) \to \mc O_{\wt{\F}}(L_{\wt{\F}}) \to \mc O_{\wt{\E}}(L_{\wt{\E}}) \to 0$. 
Since $R^{1}(f \circ \s)_{\ast}(\s^{\ast}\mc O_{\F}(1))=0$, $\mc O(L_{\wt{\F}})$ is globally generated over $C$ by Claim~\ref{claim-P13E}~(1). 

(2) Let $\g \subset \wt{\F}$ be an irreducible curve with $L_{\wt{\F}}.\g=0$. 
Then $L_{\wt{\F}} \sim_{C} \wt{\E}+\s^{\ast} \mc O_{\F}(1)$ and $\s^{\ast} \mc O_{\F}(1).\g>0$ since $L_{\wt{\F}}$ is ample over $C$. 
Thus we obtain $\wt{\E}.\g<0$, which implies $\g \subset \wt{\E}$. 
Since $L_{\wt{\E}}.\g=0$, $\g$ is a fiber of $\Exc(\psi_{\E}) \to \ol{T}$. 
Conversely, every curve $\g$ contracted by $\psi_{\E}$ is also contracted by $\psi_{\F}$. 
Then we have $\psi_{\F}|_{\wt{\E}}=\psi_{\E}$ by the rigidity lemma and hence $\Exc(\psi_{\F})=\Exc(\psi_{\E})$. 

(3) Let $l$ be a fiber of $\Exc(\psi_{\F}) \to \ol{T}$. 
Then we have $l \simeq \P^{1}$ and the exact sequence $0 \to \mc N_{l/\wt{\E}} \to \mc N_{l/\wt{\F}} \to \mc N_{\wt{\E}/\wt{\F}}|_{l} \to 0$, which implies that $\mc N_{l/\wt{F}}=\mc O_{\P^{1}}(-1)^{2} \oplus \mc O_{\P^{1}}$. We complete the proof. 
\end{proof}

\noindent
\textbf{Step 4.} 
Let $\psi^{+}_{\F} \colon \wt{\F}^{+} \to \ol{\F}$ be the flop of $\psi_{\F}$ and $\wt{\E}^{+}$ and $G^{+}_{\F}$ the proper transforms of $\wt{\E}$ and $G_{\F}$ on $\wt{\F}^{+}$ respectively. 

\begin{claim}\label{claim-P13conclu}
There exists a birational morphism 
\[\mu_{\F} \colon \wt{\F}^{+} \to Y \] 
over $C$ such that $Y$ is non-singular and $\mu_{\F}$ blows $\wt{\E}^{+}$ down to a $\vp_{Y}$-section $C_{0,Y}$, 
where $\vp_{Y} \colon Y \to C$ is the induced morphism. 
\end{claim}
\begin{proof}
By the construction of the flop $\wt{\F} \dra \wt{\F}^{+}$, the proper transform $\wt{\E}^{+} \subset \ol{\F}^{+}$ is a $\P^{2}$-bundle over $C$. 
By the equality (\ref{eq-P13-0}), we have $-K_{\ol{\F}_{t}}|_{\ol{\E}_{t}}=\mc O_{\P^{2}}(2)$ for every point $t$. 
Hence it holds that $-K_{\wt{\F}_{t}^{+}}|_{\ol{\E}^{+}_{t}} \simeq \mc O_{\P^{2}}(2)$ and hence $\mc N_{\wt{\E}^{+}_{t}/\wt{\F}^{+}_{t}} \simeq \mc O_{\P^{2}}(-1)$ for every $t \in C$. 
Thus we obtain the morphism $\mu_{\F}$. 
\end{proof}
Since $\mu_{\F}$ is the contraction of an extremal ray, we obtain $\rho(Y)=2$. Therefore, $\vp_{Y}$ is the contraction of a $K_{Y}$-negative ray. 
We complete the proof of Theorem~\ref{thm-constP13}~(2). 

Since $-K_{\wt{\F}^{+}} \sim 4\wt{\E}^{+}+2G_{\F}^{+}-(\vp_{Y} \circ \mu_{\F})^{\ast}(K_{C}+\det f_{\ast}\mc O_{\F}(1))$, 
if we set $G_{Y}=\mu_{\F\ast}G_{\F}^{+}$, then we have $-K_{Y} \sim 2G_{Y}-\vp_{Y}^{\ast}(K_{C}+\det f_{\ast}\mc O_{\F}(\E))$, which proves Theorem~\ref{thm-constP13}~(3). 

To confirm Theorem~\ref{thm-constP13}~(4), we take a point $t \in C$ such that $Y_{t}=\vp_{Y}^{-1}(t)$ is smooth. 
Then $Y_{t}$ is a smooth Fano 3-fold with index $2$, which is so-called a del Pezzo 3-fold. 
Since $(-K_{Y_{t}})^{3}=(-K_{\wt{\F}^{+}_{t}})^{3}+8=(-K_{\wt{\F}_{t}})^{3}+8=48$ and $\rho(Y_{t})=\rho(\wt{\F}^{+}_{t})-1=\rho(\wt{\F}_{t})-1=3$, 
we have $Y_{t} \simeq (\P^{1})^{3}$ by Fujita's classification of the del Pezzo manifolds \cite[Theorem~3.3.1]{Fanobook}. 
We complete the proof of Theorem~\ref{thm-constP13}. 
\end{proof}

\subsection{Proof of Theorem~\ref{mainthm-P13}}\label{subsec-proofmainthmP13}
Take a $\vp$-section $C_{0}$ and let $(q \colon Q \to C,T)$ be the relative double projection of $(\vp \colon X \to C,C_{0})$. 
Let $E_{Q} \subset Q$ be the proper transform of $\Exc(\mu \colon \Bl_{C_{0}}X \to X)$. 
By Proposition~\ref{prop-2ray68}~(5), there exists a divisor $\a$ on $C$ such that $-K_{Q} \sim 2E_{Q}-q^{\ast}\a$. 
Then consider the following projective bundle over $C$:
\[f \colon \F:=\P_{C}(q_{\ast}\mc O_{Q}(E_{Q})) \to C.\] 
Then $f$ is a $\P^{3}$-bundle and there is a natural closed embedding $Q \hra \F$ over $C$. 
When $\mc O_{\F}(1)$ denotes the tautological bundle of $f$, we have $\mc O_{\F}(1)|_{Q}=\mc O_{Q}(E_{Q})$. 
Since $-K_{Q} \sim (\mc O_{\F}(2)-f^{\ast}\a)|_{Q}$, we have the linear equivalence
\begin{align}\label{eq-classQ}
 Q \sim \mc O_{\F}(2)+f^{\ast}(\a-(K_{C}+\det (q_{\ast}\mc O_{Q}(E_{Q})))) 
\end{align}
by the adjunction formula. 
\begin{claim}\label{claim-noncoli}
There exists the unique member $\E \in |\mc O_{\F}(1)|$ such that 
\begin{enumerate}
\item $\E \cap Q=E_{Q}$ and $f|_{\E} \colon \E \to C$ is $\P^{2}$-bundle, and 
\item for every $t \in C$, $T_{t}=f^{-1}(t) \cap T$ is non-colinear $0$-dimensional subscheme of length $3$ in $\P^{3}=\F_{t}$ and spans $\E_{t}$. 
\end{enumerate}
In particular, the $\P^{3}$-bundle $f \colon \F \to C$ and $T$ satisfy the condition of the setting in Theorem~\ref{thm-constP13} and $\E$ is the sub $\P^{2}$-bundle as in Theorem~\ref{thm-constP13}~(1). 
\end{claim}
\begin{proof}
(1) Consider an exact sequence $0 \to \mc O_{\F}(-Q) \otimes \mc O_{\F}(1) \to \mc O_{\F}(1) \to \mc O_{Q}(E_{Q}) \to 0$. 
Since $\mc O_{\F}(Q) \sim_{C} \mc O_{\F}(2)$, we obtain $R^{i}f_{\ast}(\mc O_{\F}(-Q) \otimes \mc O_{\F}(1))=0$ for all $i \geq 0$. 
Then the restriction morphism $H^{0}(\F,\mc O_{\F}(1)) \to H^{0}(Q,\mc O_{Q}(E_{Q}))$ is isomorphism and hence there exists the unique member $\E \in |\mc O_{\F}(1)|$ such that $\E \cap Q=E_{Q}$. 
Since $E_{Q}$ is a prime divisor of $Q$, we obtain $\dim \E_{t}=2$ for all $t \in C$, which implies that $\E \to C$ is a $\P^{2}$-bundle. 

(2) Assume that there exists a line $l \subset \F_{t}$ such that $T_{t} \subset l$ for some $t$.  
Since $T_{t} \subset Q_{t}$ and $Q_{t}$ is a quadric surface, we obtain $l \subset Q_{t}$. 
Let $\s \colon \Bl_{T}Q \to Q$ denote the blow-up as in Proposition~\ref{prop-2ray68}. 
Then we have $-K_{\Bl_{T}Q}.\s^{-1}_{\ast}l <0$, which is a contradiction since $-K_{\Bl_{T}Q}$ is nef over $C$ from Proposition~\ref{prop-2ray68}~(2). 
Hence the linear span of $T_{t}$, say $\braket{T_{t}}$, is a 2-plane in $\F_{t}$ and thus we deduce that $\mb{E}_{t}=\braket{T_{t}}$ for every $t \in C$. 
\end{proof}
By Theorem~\ref{thm-constP13}, $\F \to C$ can be birationally transformed into a Mori fiber space $Y$ over $C$ with general fiber $(\P^{1})^{3}$ as in the diagram (\ref{dia-P13}). 
Note that $\wt{\F}=\Bl_{T}\F$ contains $\wt{Q}=\Bl_{T}Q$ in this setting. 

We use the same notation as in Theorem~\ref{thm-constP13} and its proof. 
The only remaining part is to show the following claim. 
\begin{claim}\label{claim-P13emb}
\begin{enumerate}
\item $\psi_{\F}|_{\wt{Q}}$ coincides with $\psi_{Q}$ in Proposition~\ref{prop-2ray68}~(2). 
\item The proper transform $\wt{Q}^{+} \subset \wt{\F}^{+}$ of $Q$ is isomorphic to $\wt{X}$. 
\item It holds that $\mu_{\F}|_{\wt{X}}=\mu_{X}$ and $C_{0,Y}=C_{0}$. 
\item $X$ is a member of $|G_{Y}+\vp_{Y}^{\ast}\b|$, where we set 
$\b:=\a-(K_{C}+\det (q_{\ast}\mc O_{Q}(E_{Q})))$. 
\item If we set 
$H_{Y}:=G_{Y}-\vp_{Y}^{\ast}\a$ and $\d:=\a+\b$, 
then we have $-K_{Y}=2H_{Y}+\vp_{Y}^{\ast}\d$, $X \in |H_{Y}+\vp_{Y}^{\ast}\d|$. Moreover, $\mc L:=\mc O_{C}(-\d)$ is isomorphic to $\Cok(\mc O_{C} \to \vp_{B\ast}\mc O_{B}) \otimes \mc O(-K_{C})$, where $\vp_{B} \colon B \to C$ denotes the associated double covering to $\vp$. 
\end{enumerate}
\end{claim}
\begin{proof}
(1) We have $Q \in |\mc O_{\F}(2)+f^{\ast}\b|$ by the equality (\ref{eq-classQ}) and hence $\wt{Q} \in |\mc O_{\F}(2)-G_{\F}+(f \circ \s)^{\ast}\b|=|L_{\wt{\F}}+(f \circ \s)^{\ast}\b|$ by the equality (\ref{eq-P13-L}). 
Recall that $\psi_{\F} \colon \wt{\F} \to \ol{\F}$ is the Stein factorization of the morphism given by $|L_{\wt{\F}}|$. 
For every $k>0$, it follows that $R^{1}\wt{f}_{\ast}\mc O_{\wt{\F}}((k-1)L_{\wt{\F}})=0$ and hence $\wt{f}_{\ast}\mc O_{\wt{\F}}(kL_{\wt{\F}}) \to (\wt{f}|_{\wt{Q}})_{\ast}\mc O_{\wt{Q}}(kL_{\wt{\F}})$ is surjective, which implies $\psi_{\F}|_{\wt{Q}}=\psi_{Q}$. 

(2) 
We obtain $\wt{Q} \sim_{\ol{\F}} 0$ as divisors on $\wt{\F}^{+}$ and hence $\ol{Q}:=\psi_{\F}(\wt{Q})$ is a Cartier divisor on $\ol{\F}$. 
Then we have $\wt{Q}^{+}=(\psi^{+}_{\F})^{-1}(\ol{Q}) \subset \wt{\F}^{+}$ and hence the dimension of the fibers of $\wt{Q}^{+} \to \ol{Q}$ is less than or equals to $1$. 
Since $\psi_{\F}|_{\wt{Q}}$ is the flopping contraction of $\wt{Q}$ over $C$, 
$\ol{Q} \cap \psi_{\F}(\Exc(\psi_{\F}))$ is a finite set or the empty set. 
(2) clearly holds when it is empty. 
Now we assume $\ol{Q} \cap \psi_{\F}(\Exc(\psi_{\F}))$ is not empty. 
Hence the morphism $\wt{Q}^{+} \to \ol{Q}$ is a small contraction. 
Thus $\wt{Q}^{+}$ is regular in codimension $1$ and hence normal since this is an effective divisor of a smooth variety $\wt{\F}^{+}$. 
Moreover, the birational map $\Psi|_{\wt{Q}} \colon \wt{Q} \dra \wt{Q}^{+}$ is isomorphic in codimension $1$, which implies that $E_{\wt{Q}}=\wt{\E}|_{\wt{Q}^{+}}$ coincides with the proper transform of $\wt{\E}^{+}|_{\wt{Q}}$. 
Since $-\wt{\E} \sim_{\ol{Y}} \s_{\F}^{\ast}\mc O_{\F}(1)$ is ample over $\ol{Y}$ and $\wt{\E}^{+}$ is ample over $\ol{Y}$, we conclude that $\wt{Q}^{+}$ is the flop of $\psi_{Q}$. 
Then we have $\wt{Q}^{+}=\wt{X}$ by the uniqueness of the flop. 

(3) It is enough to show that $\mu_{\F}^{\ast}G_{Y}|_{\wt{X}} \sim_{C}\mu_{X}^{\ast}(-K_{X})$. 
By (\ref{eq-classQ}),  $\wt{Q}$ is a member of $|\wt{\E}+G+(f \circ \s)^{\ast}\b|$ and hence $\wt{X}$ is a member of $|\wt{\E}^{+}+G^{+}+(\vp_{Y} \circ \mu_{\F})^{\ast}\b|$. 
Since $-K_{\wt{\F}^{+}} \sim_{C} 2(\wt{\E}^{+}+G_{\F}^{+})$, we have $-K_{\wt{X}} \sim_{C} (\wt{\E}^{+}+G_{\F}^{+})|_{\wt{X}}$. 
Recalling that we set $G_{Y}=\mu_{\F\ast}G^{+}$, we have $\mu_{\F}^{\ast}G_{Y}|_{\wt{X}} \sim_{C} \mu_{X}^{\ast}(-K_{X})$, which completes the proof of (3). 

(4) This assertion follows from the equality $\mu_{\F}^{\ast}X = \wt{\E}^{+}+\wt{X}$ as divisors. 

(5) By Theorem~\ref{thm-constP13}~(3), we have 
$-K_{Y}
=2G_{Y}-\vp_{Y}^{\ast}(K_{C}+\det q_{\ast}\mc O_{Q}(E_{Q}))
=2G_{Y}+\vp_{Y}^{\ast}(\b-\a)$. 
Thus we obtain 
$-K_{Y}=2H_{Y}+\vp_{Y}^{\ast}\d$ and $X \in |H_{Y}+\vp_{Y}^{\ast}\d|$. 
By the argument in the proof of Lemma~\ref{lem-89}~(3) and the equality (\ref{eq-2adetF}), we have $\omega_{B}=\vp_{B}^{\ast}\mc O_{C}(\d)$. 
By the duality of the finite flat morphism $\vp_{B} \colon B \to C$, we have 
$\mc O_{C}(\d) \otimes {\vp_{B}}_{\ast}\mc O_{B}={\vp_{B}}_{\ast}\omega_{B}=({\vp_{B}}_{\ast}\mc O_{B})^{\vee} \otimes \omega_{C}$. 
Thus $\mc O_{C}(K_{C}-\d)$ is the cokernel of the splitting injection $\mc O_{C} \to {\vp_{B}}_{\ast}\mc O_{B}$. 
\end{proof}
We complete the proof of Theorem~\ref{mainthm-P13}. \qed


\section{Extensions to moderate $(\P^{2})^{2}$-fibrations}\label{sec-P22}

We devote this section to prove Theorem~\ref{mainthm-P22}. 
A main idea is similar to that of Theorem~\ref{mainthm-P13}. 
\subsection{Double projection from a point on $(\P^{2})^{2}$}\label{subsec-doubleprojP22}
In Subsection~\ref{subsec-doubleprojP13}, we saw that the double projection from a point on a sextic del Pezzo surface $S \dra \Q^{2}$ is the restriction of the double projection $(\P^{1})^{3} \dra \P^{3}$ from the point as in the diagram (\ref{dia-P13gen}). 
As $S$ is also a codimension $2$ linear section of $(\P^{2})^{2}$, 
it is also natural to consider the double projection of $(\P^{2})^{2}$. 

Let $S$ be a sextic del Pezzo surface and take a general point $x \in S$. 
As in the diagram (\ref{dia-68gen}), we have $\Bl_{x}S \simeq \Bl_{y_{1},y_{2},y_{3}}\Q^{2}$ where $y_{1},y_{2},y_{3}$ are non-colinear three points. 
Applying Theorem~\ref{thm-univext} for $X=\Q^{2}$, $Y=\Spec \C$ and $Z=\{y_{1},y_{2},y_{3}\}$, we obtain a locally free sheaf $\mc F_{\Q^{2}}$ 
fitting into the following exact sequence 
\[0 \to \mc O_{\Q^{2}}^{2} \to \mc F_{\Q^{2}} \to \mc I_{Z}(-K_{\Q^{2}}) \to 0\]
such that $\mc F_{\Q^{2}}$ has no trivial bundles as quotients. 
Since $Z$ is a union of non-colinear three points, we can deduce that $\mc I_{Z}(-K_{\Q^{2}})$ is globally generated and so is $\mc F_{\Q^{2}}$ since $H^{1}(\mc O_{\Q^{2}})=0$. 
Thus we have $h^{0}(\mc F_{\Q^{2}}^{\vee})=0$ and hence 
$\mc F_{\Q^{2}} \simeq \ms S_{\Q^{2}} \oplus \mc O_{\Q^{2}}(1) \simeq \mc O_{(\P^{1})^{2}}(1,0) \oplus \mc O_{(\P^{1})^{2}}(0,1) \oplus \mc O_{(\P^{1})^{2}}(1,1)$
by Proposition~\ref{prop-Spi2}. 
Let us consider the projectivization $\pi \colon \P_{\Q^{2}}(\mc F_{\Q^{2}}) \to \Q^{2}$. 
Then we can see that there exists the following diagram:
\begin{align}\label{dia-genP22}
\xymatrix{
&\Bl_{x}(\P^{2})^{2} \ar[ld]_{\mu} &\P_{\Q^{2}}(\mc F_{\Q^{2}})\ar[rd]^{\pi} \ar@{-->}[l]_{\Psi}& \\
(\P^{2})^{2}&&&\Q^{2},
}
\end{align}
where $\Psi$ is a flop and $\mu$ is the blow-up of $(\P^{2})^{2}$ at a point $x$. 
The flopped locus is the union of two sections of $\pi$ which correspond to $\mc F_{\Q^{2}} \to \mc O_{\Q^{2}}(0,1)$ and $\mc F_{\Q^{2}} \to \mc O_{\Q^{2}}(1,0)$. 
The rational map $(\P^{2})^{2} \dra \Q^{2}$ is nothing but the double projection from $x$ on $(\P^{2})^{2}$ and its restriction on $S$ coincides with the double projection $S \dra \Q^{2}$. 
Our proof of Theorems~\ref{thm-constP22} and \ref{mainthm-P22} essentially includes the details of the above argument. 

\subsection{Moderate $(\P^{2})^{2}$-fibrations}

As in Subsection~\ref{subsec-doubleprojP22}, $(\P^{2})^{2}$ can be constructed from a quadric surface with non-colinear three points. 
We relativize this construction for making a Mori fiber space $\vp_{Z} \colon Z \to C$ with smooth total space $Z$ whose smooth fibers are $(\P^{2})^{2}$. 
We call this $\vp_{Z}$ a \emph{moderate $(\P^{2})^{2}$-fibration} in this paper.

\begin{thm}\label{thm-constP22}
Let $C$ be a smooth projective curve and $q \colon Q \to C$ a quadric fibration. 
Let $T \subset Q$ be a smooth irreducible closed subcurve such that $\deg(q|_{T})=3$ and $-K_{\Bl_{T}Q}$ is nef over $C$. 
\begin{enumerate}
\item There exists a locally free sheaf $\mc F$ and the following exact sequence
\begin{align}\label{eq-F}
0 \to q^{\ast}(R^{1}q_{\ast}\mc I_{T}(-K_{C})) \to \mc F \to \mc I_{T}(-K_{Q}) \to 0
\end{align}
such that 
$\mc F|_{Q_{t}} \simeq \ms S_{Q_{t}} \oplus \mc O_{Q_{t}}(1)$ 
holds for every $t \in C$. We refer the definitions of $\mc O_{Q_{t}}(1)$ and $\ms S_{Q_{t}}$ to Subsection~\ref{NC}~(11). 
\item There exists the following diagram:
\begin{align}
\label{dia-P22}
\xymatrix{
&\wt{Z}\ar[rd]_{\psi_{Z}} \ar[dd]_{r^{+}} \ar[ld]_{\mu_{Z}}&&\P_{Q}(\mc F) \ar[dd]^{r} \ar[ld]^{\psi_{\mc F}} \ar[rd]^{\pi_{\mc F}} \ar@{-->}[ll]_{\Psi}& \\
Z\ar[rd]_{\vp_{Z}}&&\ol{Z} \ar[rd] \ar[ld]&&Q \ar[ld]^{q} \\
&C\ar@{=}[rr]&&C,&
}
\end{align}
where 
\begin{itemize}
\item $\P_{Q}(\mc F) \dra \wt{Z}$ is a family of Atiyah flops;
\item $Z$ is smooth and $\vp_{Z} \colon Z \to C$ is a Mori fiber space; 
\item $\mu_{Z} \colon \wt{Z} \to Z$ is the blow-up along a $\vp_{Z}$-section $C_{0}$. 
\end{itemize}
\item Let $\xi_{\mc F}$ be a tautological divisor on $\P_{Q}(\mc F)$ and set $\xi_{Z}:={\mu_{Z}}_{\ast}\Psi_{\ast}\xi_{\mc F}$. 
Then $-K_{Z} \sim 3\xi_{Z}-\vp_{Z}^{\ast}\b$ and $\mu_{Z}^{\ast}\xi_{Z}=\Psi_{\ast}\xi_{\mc F}+\Exc(\mu_{Z})$ ,
where $\b=\det(R^{1}q_{\ast}\mc I_{T}(-K_{C}))$. 
\item Every smooth $\vp_{Z}$-fiber is isomorphic to $(\P^{2})^{2}$. 
\end{enumerate}
In this paper, we call the Mori fiber space $\vp_{Z} \colon Z \to C$ be as in Theorem~\ref{thm-constP22} the \emph{moderate $(\P^{2})^{2}$-fibration} with respect to the pair $(q \colon Q \to C,T)$. 
\end{thm}

\begin{proof}
We proceed in 4 steps.

\noindent\textbf{Step 1.} Let us prove (1). 
We apply Theorem~\ref{thm-univext} for $X=Q$, $Y=C$, $f=q$ and $Z=T$. 
Since $R^{1}q_{\ast}\mc O_{X}=0$ and $\dim C=1$, 
Theorem~\ref{thm-univext} gives the exact sequence (\ref{eq-F}) and the locally free sheaf $\mc F$. 
Moreover, there are no surjections $\mc F|_{Q_{t}} \to \mc O_{Q_{t}}$ for every $t$. 
Note that $\det \mc F|_{Q_{t}}=\mc O(-K_{Q_{t}})$ and $c_{2}(\mc F|_{Q_{t}})=3$ follow from (\ref{eq-F}). 
If $\mc F$ is $q$-nef, then every morphism $\mc F \to \mc O_{Q_{t}}$ is surjective and hence the zero map. 
Thus Proposition~\ref{prop-Spi2} shows (1). 
To show that $\mc F$ is $q$-nef, we consider the projectivization $\P_{Q}(\mc F)$. 
By the surjection $\mc F \epm \mc I_{T}(-K_{Q})$, $\Bl_{T}Q$ is embedded in $\P_{Q}(\mc F)$ over $Q$ as the zero scheme of the global section $s \in H^{0}(\P_{Q}(\mc F), \mc O_{\P_{Q}(\mc F)}(1) \otimes q^{\ast}(R^{1}q_{\ast}\mc I_{T}(-K_{C}))^{\vee})$ corresponding to the injection $q^{\ast}(R^{1}q_{\ast}\mc I_{T}(-K_{C})) \to \mc F$ (cf. Lemma~\ref{lem-irr}). 
Since $-K_{\Bl_{T}Q}=\mc O_{\P_{Q}(\mc F)}(1)|_{\wt{Q}}$ is nef over $C$ by assumption, so is $\mc O_{\P_{Q}(\mc F)}(1)$. 
We complete the proof of (1). 

\noindent\textbf{Step 2.} Next we confirm the following claim. 
\begin{claim}\label{claim-EQ}
There exists a unique effective divisor $E_{Q} \subset Q$ containing $T$ such that $2E_{Q}+K_{Q} \sim_{C} 0$ and $q_{\ast}\mc I_{T}(E_{Q})=\mc O_{C}$. 
\end{claim}
\begin{proof}
Take $f \colon \F \to C$ is a $\P^{3}$-bundle over $C$ such that $\F$ contains $Q$. 
Since $-K_{\Bl_{T}Q}$ is nef over $C$, the linear span of $T_{t}$ is a $2$-plane in $\F_{t}=\P^{3}$. 
By Theorem~\ref{thm-constP13}~(1), there exists a unique sub $\P^{2}$-bundle $\E$ contains $T$. 
Set $E_{Q}:=\E \cap Q$. 
Then we have an exact sequence $0 \to \mc O_{\F}(-Q+\E) \to \mc O_{\F}(\E) \otimes \mc I_{T/\F} \to \mc O(E_{Q}) \otimes \mc I_{T/Q} \to 0$. 
Since $R^{i}f_{\ast}\mc O_{\F}(-Q+E)=0$ for any $i$ and $f_{\ast}(\mc I_{T/\F} \otimes \mc O_{\F}(\E))=\mc O_{C}$ as in the proof of Theorem~\ref{thm-constP13}~(1), 
we have that $q_{\ast}\mc I_{T}(E_{Q})=\mc O_{C}$, which proves that $E_{Q}$ satisfies the conditions. 
For a general point $t \in C$, the fiber $(E_{Q})_{t}$ is the unique smooth conic passing through the three points $T_{t}$. 
Hence the uniqueness of $E_{Q}$ follows. 
\end{proof}
From now on, we fix a divisor $\a$ on $C$ such that $-K_{Q}=2E_{Q}-q^{\ast}\a$. 
\begin{claim}\label{claim-E'}
There exists a exact sequence 
\begin{align}
\label{eq-P22-2} 0 \to \mc O(E_{Q}-q^{\ast}\a) \to \mc F \to \mc E \to 0
\end{align}
where $\mc E$ is a locally free sheaf with $\mc E|_{Q_{t}} \simeq \ms S_{Q_{t}}$ for every $t \in C$. 
\end{claim}
\begin{proof}
Tensoring the exact sequence (\ref{eq-F}) with $\mc O(-E_{Q}+q^{\ast}\a)$, we obtain the following exact sequence: 
\[0 \to q^{\ast}((R^{1}q_{\ast}\mc I_{T})(\a-K_{C})) \otimes \mc O(-E_{Q}) \to \mc F(-E_{Q}+q^{\ast}\a) \to \mc I_{T}(E_{Q}) \to 0.\]
Claim~\ref{claim-EQ} implies that $q_{\ast}\mc I_{T}(E_{Q})=\mc O_{C}$. 
Since $R^{i}q_{\ast}\mc O_{Q}(-E_{Q})=0$ for any $i$, we have $q_{\ast}\mc F(-E_{Q} + q^{\ast}\a) \simeq q_{\ast}\mc I_{T}(E_{Q}) \simeq \mc O_{C}$. 
Hence we obtain an injection 
$\iota \colon \mc O(E_{Q}-q^{\ast}\a) \to \mc F$ 
with the locally free cokernel $\mc E:=\Cok \iota$. 
Then we have an exact sequence 
$0 \to \mc O_{Q_{t}}(1) \to \mc F|_{Q_{t}} \to \mc E|_{Q_{t}} \to 0$ 
for every $t \in C$. 
Since $\mc F|_{Q_{t}} \simeq \ms S_{Q_{t}} \oplus \mc O_{Q_{t}}(1)$ and $\Hom(\mc O_{Q_{t}}(1),\ms S_{Q_{t}})=0$, 
$\mc F|_{Q_{t}}$ contains $\mc O_{Q_{t}}(1)$ as a direct summand. 
Hence we have $\mc E|_{Q_{t}} \simeq \ms S_{Q_{t}}$ for every $t \in C$. 
\end{proof}

\noindent\textbf{Step 3.} Let $\P_{Q}(\mc E) \subset \P_{Q}(\mc F)$ be the natural inclusion from the exact sequence (\ref{eq-P22-2}). 
We define morphisms as in the following commutative diagram:
\[\xymatrix{
\P_{Q}(\mc F)\ar[d]^{\pi_{\mc F}}  \ar@(l,l)[dd]_{r_{\mc F}}& \ar@^{{(}->}[l] \ar[ld]^{\pi_{\mc E}}\P_{Q}(\mc E) \ar@(d,r)[ldd]^{r_{\mc E}}\\
Q \ar[d]_{q} \\
C. 
}\]
Let $\xi_{\mc F}$ be a tautological divisor of $\pi_{\mc F} \colon \P_{Q}(\mc F) \to Q$ and $\xi_{\mc E}:=\xi_{\mc F}|_{\P_{Q}(\mc E)}$. 
Then $\xi_{\mc E}$ is a tautological divisor of $\P_{Q}(\mc E)$ and $\P_{Q}(\mc E) \subset \P_{Q}(\mc F)$ is a member of $|\mc O(\xi_{\mc F}-\pi^{\ast}(E_{Q})) \otimes r_{\mc F}^{\ast}\mc O(\a)|$. 
Note that $\mc O_{\P_{Q}(\mc E)}(\xi_{\mc E})$ is $r_{\mc E}$-globally generated and $r_{\mc E\ast}\mc O_{\P_{Q}(\mc E)}(\xi_{\mc E})=q_{\ast}\mc E$ is a vector bundle of rank 4. 
Thus $\xi_{\mc E}$ gives a morphism 
$\psi_{\mc E} \colon \P_{Q}(\mc E) \to \P_{C}(q_{\ast}\mc E)$. 
Set $G=\Exc(\psi_{\mc E})$ and $S=\psi_{\mc E}(G)$. 
\begin{claim}\label{claim-P22E}
\begin{enumerate}
\item $S$ is smooth and $\psi_{\mc E} \colon \P_{Q}(\mc E) \to \P_{C}(q_{\ast}\mc E)$ is the blow-up of $\P_{C}(q_{\ast}\mc E)$ along $S$. 
\item The morphism $S \to C$ factors a non-singular curve $B$ such that $S \to B$ is a $\P^{1}$-bundle and $B \to C$ is a double covering. Moreover, $B \to C$ is the double covering associated to $q$. 
\end{enumerate}
\end{claim}
\begin{proof}

(1) Let $t \in C$ be a point and consider a morphism 
$\psi_{\mc E,t} \colon \P_{Q_{t}}(\ms S_{Q_{t}}) \to \P^{3}$. 
If $Q_{t}$ is smooth (resp. singular), then it is known that $\psi_{\mc E,t}$ is the blow-up along union of two disjoint lines in $\P^{3}$ \cite[Table~3, No~25.]{MM81} (resp. the blow-up along a double line which is contained in a smooth quadric surface). 
This fact implies that every fiber $l$ of $G \to S$ is isomorphic to $\P^{1}$ and satisfies $-K_{\P_{Q}(\mc E)}.l=1$. 
Then \cite[Theorem~2.3]{Ando85} implies that $S$ is non-singular and $\psi_{\mc E}$ is the blow-up along $S$. 

(2) For any $t\in C$, $S_{t} \subset \P^{3}$ is a union of two disjoint lines if  $Q_{t}$ is smooth and $(S_{t})_{\red}$ is a line if $Q_{t}$ is singular. 
Let $S \to B \to C$ be the Stein factorization of $S \to C$. 
Since $S$ is smooth, so is $B$. 
Then $S \to B$ is a $\P^{1}$-bundle and $B \to C$ is a double covering. 
The branched locus of this double cover $B \to C$ is $\{t \in C \mid Q_{t} \text{ is singular} \}$. Therefore, $B \to C$ is the double covering associated to $q \colon Q \to C$. 
\end{proof}

\noindent\textbf{Step 4.} 
Let $\psi_{\mc F} \colon \P_{Q}(\mc F) \to \ol{Z}$ be the Stein factorization of the morphism given by $|\xi_{\mc F}|$ over $C$. 
Since $-K_{\P_{Q}(\mc F)}\sim_{C} 3\xi_{\mc F}$, $\psi_{\mc F} \colon \P_{Q}(\mc F) \to \ol{Z}$ is a crepant contraction. 

\begin{claim}\label{claim-P22F}
\begin{enumerate}
\item $\Exc(\psi_{\mc E})=\Exc(\psi_{\mc F})$ holds and $\psi_{\mc F}$ is a 2-dimensional family of Atiyah's flopping contraction. In particular, if $\Psi \colon \P_{Q}(\mc F) \dra \wt{Z}$ denotes the flop, then $\wt{Z}$ is non-singular. 
\item Let $E_{\wt{Z}}$ be the proper transform of $\P_{Q}(\mc E) \subset \P_{Q}(\mc F)$ on $\wt{Z}$. Then there exists a birational morphism 
$\mu_{Z} \colon \wt{Z} \to Z$ 
over $C$ such that $Z$ is non-singular and $\mu_{Z}$ is the blow-up along a section $C_{0,Z}$ of the induced morphism $\vp_{Z} \colon Z \to C$. 
\end{enumerate}
\end{claim}
\begin{proof}
(1) Let $\g$ be a curve contracted by $\psi_{\mc F}$. 
Then we have $\xi_{\mc F}.\g=0$ and hence $\P_{Q}(\mc E).\g=-\pi_{\mc F}^{\ast}E_{Q}.\g<0$ since $E_{Q}$ is ample over $C$. 
Thus $\P_{Q}(\mc E)$ contains $\g$ and $\g$ is contracted by $\psi_{\mc E}$. 
Conversely, it is clear that every curve $\g$ contracted by $\psi_{\mc E}$ is also contracted by $\psi_{\mc F}$. 
Therefore, we have $\psi_{\mc F}|_{\P_{Q}(\mc E)}=\psi_{\mc E}$ by the rigidity lemma and $\Exc(\psi_{\mc F})=\Exc(\psi_{\mc E})$. 
Let $l$ be any fiber of $\Exc(\psi_{\mc F}) \to S$. 
Then we have $\mc N_{l/\P_{Q}(\mc E)}=\mc O(-1) \oplus \mc O^{2}$ and $1=-K_{\P_{Q}(\mc E)}.l=(\xi_{\mc F}+\pi_{\mc F}^{\ast}E_{Q}).l=\pi_{\mc F\ast}l.E_{Q}$. 
Hence $\P_{Q}(\mc E).l=-1$ in $\P_{Q}(\mc F)$. 
Considering the exact sequence $0 \to \mc N_{l/\P_{Q}(\mc E)} \to \mc N_{l/\P_{Q}(\mc F)} \to \mc N_{\P_{Q}(\mc E)/\P_{Q}(\mc F)}|_{l} \to 0$, 
we obtain $\mc N_{l/\P_{Q}(\mc F)} \simeq \mc O(-1)^{2} \oplus \mc O^{2}$. 
Thus $\psi_{\mc F}$ is a 2-dimensional family of Atiyah's flopping contraction. 

(2) We have $E_{\wt{Z}} \simeq \P_{C}(q_{\ast}\mc E)$ by the construction of this flop. 
Moreover, for each $t \in C$, if we take a line $l \subset E_{\wt{Z},t} \simeq \P^{3}$, then we have $-K_{\wt{Z}}.l=-K_{\wt{Z}_{t}}.l=3$. 
It implies that $-K_{\wt{Z}_{t}}|_{E_{\wt{Z}_{t}}}=\mc O_{\P^{3}}(3)$ and hence $\mc N_{E_{\wt{Z}_{t}}/\wt{Z}_{t}} \simeq \mc O_{\P^{3}}(-1)$ for every $t \in C$. 
Therefore, there exists a morphism $\mu_{Z} \colon \wt{Z} \to Z$ over $C$ such that $\mu_{Z}$ blows $E_{\wt{Z}}$ down to a section $C_{0,Z}$ of $\vp_{Z} \colon Z \to C$ and $Z$ is smooth. 
\end{proof}
Note that $\mu_{Z}$ is an extremal contraction and hence $\rho(Z)=2$. 

We show (3). 
Set $\b:=\det(R^{1}q_{\ast}\mc I_{T}(-K_{C}))$ as in Theorem~\ref{thm-constP22}~(3). 
Since $-K_{\P_{Q}(\mc F)}=3\xi_{\mc F}-r_{\mc F}^{\ast}\b$ by the exact sequence (\ref{eq-F}), 
we have $-K_{Z}=3\xi_{Z}-\vp_{Z}^{\ast}\b$. 
Moreover, since $\mu_{Z}$ is the blow-up along a $\vp_{Z}$-section, we have $\mu_{Z}^{\ast}K_{Z}=K_{\wt{Z}}+3E_{\wt{Z}}$. 
Thus we have $\mu_{Z}^{\ast}\xi_{Z}=\Psi_{\ast}\xi_{\mc F}+E_{\wt{Z}}$ since $\xi_{Z}={\mu_{Z}}_{\ast}\Psi_{\ast}\xi_{\mc F}$ by the definition. 
We complete the proof of (3). 

To prove (4), let $t \in C$ be a point such that $Z_{t}=\vp_{Z}^{-t}(t)$ is smooth. 
Then $Z_{t}$ is a so-called del Pezzo 4-fold. 
From the diagram (\ref{dia-P22}), we have $\xi_{Z}^{4} \cdot \vp_{Z}^{-1}(t)=\xi_{\mc F}^{4} \cdot r_{\mc F}^{-1}(t)+1=6$ by a direct calculation. 
Then we have $Z_{t} \simeq (\P^{2})^{2}$ by Fujita's classification of del Pezzo manifolds \cite[Theorem~3.3.1]{Fanobook}.

We complete the proof of Theorem~\ref{thm-constP22}. 
\end{proof}

\subsection{Proof of Theorem~\ref{mainthm-P22}}
Let $\vp \colon X \to C$ be a sextic del Pezzo fibration. 
Take a $\vp$-section $C_{0}$ and let $(q \colon Q \to C,T)$ be the relative double projection of $(\vp \colon X \to C,C_{0})$. 
Let $E_{Q} \subset Q$ be the proper transform of $\Exc(\Bl_{C_{0}}X \to X)$. 
Then $E_{Q}$ is nothing but the divisor that we obtain in Claim~\ref{claim-EQ}. 
By Proposition~\ref{prop-2ray68}~(2), we see that $\mc O(-K_{\Bl_{T}Q})$ is nef over $C$. 
Then Theorem~\ref{thm-constP22} gives the moderate $(\P^{2})^{2}$-fibration $\vp_{Z} \colon Z \to C$. 
In order to find an embedding from $X$ into $Z$ over $C$, it suffices to show that the proper transform of $Q$ on $Z$ coincides with $X$. 

Now let us use the same notation as in Theorem~\ref{thm-constP22} and its proof. 
Let $\mc G$ be as in the statement of Theorem~\ref{mainthm-P22}. 
Considering the exact sequence $0 \to \mc I_{T/Q} \to \mc O_{Q} \to \mc O_{T} \to 0$ and taking the cohomology of $q_{\ast}$, we obtain 
\begin{align}\label{eq-GisR1}
\mc G=R^{1}q_{\ast}\mc I_{T} \otimes \mc O(-K_{C}).
\end{align}
Note that $\wt{Q}=\Bl_{T}Q$ is the zero scheme of the global section of $H^{0}(\P_{Q}(\mc F),\mc O(\xi_{\mc F}) \otimes r_{\mc F}^{\ast}\mc G^{\vee})$ corresponding to the injection $q^{\ast}\mc G \to \mc F$ in the sequence (\ref{eq-F}) under the natural isomorphism $H^{0}(\P_{Q}(\mc F),\mc O(\xi_{\mc F}) \otimes r_{\mc F}^{\ast}\mc G^{\vee}) \simeq \Hom_{Q}(q^{\ast}\mc G,\mc F)$. 

Now Theorem~\ref{mainthm-P22} is a consequence of the following claim. 
\begin{claim}\label{claim-P22conclu}
\begin{enumerate}
\item $\psi_{\mc F}|_{\wt{Q}}$ coincides with $\psi_{Q}$ in Proposition~\ref{prop-2ray68}~(2).  \item If $\wt{Q}^{+} \subset \wt{Z}$ denotes the proper transform of $\wt{Q} \subset \P_{Q}(\mc F)$, then the birational map $\wt{Q} \dra \wt{Q}^{+}$ is the flop over $C$. In particular, $\wt{Q}^{+} \simeq \wt{X}$ holds. 
\item It holds that $\mu_{Z}|_{\wt{X}}=\mu_{X}$. In particular, there exists a closed embedding $i \colon X \hra Z$ such that $i(C_{0})=C_{0,Z}$. 
\item $X$ is the zero scheme of a global section of $\mc O_{Z}(\xi_{Z}) \otimes \vp_{Z}^{\ast}\mc G^{\vee}$. 
\item It holds that $\mc O(K_{Z}+3\xi_{Z}) \simeq \vp_{Z}^{\ast}\det \mc G$. 
\end{enumerate}
\end{claim}
\begin{proof}
(1) It suffices to show the restriction morphism ${r_{\mc F}}_{\ast}\mc O_{\P_{Q}(\mc F)}(k\xi_{\mc F}) \to (r_{\mc F}|_{\wt{Q}})_{\ast}\mc O_{\wt{Q}}(k\xi_{\mc F})$ is surjective for every $k>0$.
Since $\wt{Q}$ is the zero scheme of a global section of the rank $2$ vector bundle $\mc O_{\P_{Q}(\mc F)}(1) \otimes r_{\mc F}^{\ast}\mc G^{\vee}$, we have the exact sequence 
$0 \to  \mc O_{\P_{Q}(\mc F)}((k-2)\xi_{\mc F}) \otimes r_{\mc F}^{\ast}\det \mc G  \to \mc O_{\P_{Q}(\mc F)}((k-1)\xi_{\mc F}) \otimes r_{\mc F}^{\ast}\mc G \to \mc I_{\wt{Q}/\P_{Q}(\mc F)} (k\xi_{\mc F}) \to 0$. 
Thus we have $R^{1}{r_{\mc F}}_{\ast}\mc I_{\wt{Q}/\P_{Q}(\mc F)}(k\xi_{\mc F})=0$ for every $k>0$. 
Hence we are done. 

(2) By (1), we have $\psi_{\mc F}(\wt{Q})=\ol{X}$. 
Since $\psi_{\mc F}$ is defined over $C$ by $\xi_{\mc F}$, 
there exists a Cartier divisor $\xi_{\ol{Z}}$ on $\ol{Z}$ such that $\xi_{\ol{Z}}$ is ample over $C$ and $\psi_{\mc F}^{\ast}\xi_{\ol{Z}}=\xi_{\mc F}$. 
Since $\wt{Q}$ is the zero scheme of a global section of $\psi_{\mc F}^{\ast}(\mc O(\xi_{\ol{Z}}) \otimes \ol{r}^{\ast}\mc G^{\vee})$, 
$\ol{X}$ is that of $\mc O(\xi_{\ol{Z}}) \otimes \ol{r}^{\ast}\mc G^{\vee}$ in $\ol{Z}$ and hence 
$\wt{Q}^{+}$ is that of $\psi_{Z}^{\ast}\mc O(\xi_{\ol{Z}}) \otimes r^{+\ast}\mc G^{\vee}$ in $\wt{Z}$. 
Then (2) follows from similar arguments as in the proof of Claim~\ref{claim-P13emb}~(2).  

(3) By Theorem~\ref{thm-constP22}~(3), we have $\mu_{Z}^{\ast}\xi_{Z} \sim \xi_{\wt{Z}}+E_{\wt{Z}}$. 
Since $\Psi_{\ast}\xi_{\mc F}|_{\wt{X}} \sim_{C} -K_{\wt{X}}$, we obtain $\mu_{Z}^{\ast}\xi_{Z}|_{\wt{X}} \sim_{C} \mu_{X}^{\ast}(-K_{X})$, which proves the assertion. 

(4) Since $\wt{X}$ is the zero scheme of a global section of $\mc O_{\wt{Z}}(\Psi_{\ast}\xi_{\mc F}) \otimes r^{+\ast}\mc G^{\vee}$ and $\Psi_{\ast}\xi_{\mc F}=\mu_{Z}^{\ast}\xi_{Z}-\Exc(\mu_{Z})$ holds, 
we obtain a section $s \in H^{0}(Z,\mc I_{C_{0}/Z}(\xi_{Z}) \otimes r^{+\ast}\mc G^{\vee})$ such that the zero scheme of $s$ is $X$. 

(5) This assertion immediately follows from (\ref{eq-GisR1}) and Theorem~\ref{thm-constP22}~(3). 
\end{proof}

We complete the proof of Theorem~\ref{mainthm-P22}. 
\qed

\section{Singular fibers of sextic del Pezzo fibrations}\label{sec-singfib}

In this section, we classify singular fibers of sextic del Pezzo fibrations.

\subsection{Singular sextic del Pezzo surfaces} 
For a sextic del Pezzo fibration $\vp \colon X \to C$, every fiber is irreducible and Gorenstein. 
Such surfaces were studied by many peoples (e.g. \cite{HW81,CT88,Fujita86,Reid94,AF03}). 
As the first step to classify singular fibers, we review the classification of irreducible Gorenstein sextic del Pezzo surfaces.

\begin{thm}[\cite{HW81,CT88,Fujita86,Reid94,AF03}]\label{thm-GordP6}
Let $S$ be an irreducible Gorenstein del Pezzo surface with $(-K_{S})^{2}=6$. 

(1) Suppose that $S$ has only Du Val singularities. 
Let $\wt{S} \to S$ be the minimal resolution $\wt{S} \to S$. 
Then there is a birational morphism $\e \colon \wt{S} \to \P^{2}$ such that $\e$ is the blow-up of $\P^{2}$ at three (possibly infinitely near) points. 
If $\Sigma$ denotes the $0$-dimensional subscheme in $\P^{2}$ of length $3$ corresponding to the three (possibly infinitely near) points, 
then the isomorphism class of $S$ is determined by $\Sigma$ as in Table~2. 

\begin{table}[h]
\begin{tabular}{|c||c|c|c|c|c|}
\hline 
Type & $\Sigma$ & Is $\Sigma$ colinear? & $\#$ of lines& Singularity\\
\hline \hline 
(2,3) & reduced & non-colinear & $6$ & smooth  \\
\hline
(2,2) & $\Spec \C \sqcup \Spec \C[\e]/(\e^{2})$ & non-colinear & $4$ & $A_{1}$  \\
\hline
(2,1) & $\Spec \C[\e]/(\e^{3})$ & non-colinear & $2$ & $A_{2}$ \\
\hline
(1,3) & reduced & colinear & $3$ & $A_{1}$  \\
\hline
(1,2) & $\Spec \C \sqcup \Spec \C[\e]/(\e^{2})$ & colinear & $2$ & $A_{1}+A_{1}$  \\
\hline
(1,1) & $\Spec \C[\e]/(\e^{3})$ & colinear & $1$ & $A_{1}+A_{2}$  \\
\hline
\end{tabular}
\caption{Classification of Du Val sextic del Pezzo surfaces}
\end{table}

We refer to \cite{HW81,CT88,Fujita86} for more precise details. 

(2) Suppose that $S$ is not Du Val but rational. 
Let $\nu \colon \ol{S} \to S$ be the the normalization. 
Then $\ol{S}$ is a Hirzebruch surface and the complete linear system $|\nu^{\ast}\omega_{S}^{-1}|$ gives an embedding $\ol{S} \hra \P^{7}$ and $S \subset \P^{6}$ is the image of the projection from a point away from $\ol{S}$. 
Let $\mc C \subset \mc O_{\ol{S}}$ be the conductor of $\nu$ and $E=\Spec \mc O_{\ol{S}}/\mc C$. 
Then $(\ol{S},\nu^{\ast}\omega_{S}^{-1},E,E \to \nu(E))$ is one of the two cases in Table~3. 

\begin{table}[h]
\begin{tabular}{|c||c|c|c|c|}
\hline 
Type & $\ol{S}$ & $\nu^{\ast}\omega_{S}^{-1}$ & $E$ & $E \to \nu(E)$\\
\hline \hline 
(n2) & $\F_{2}$ & $h+2f$ & $E=C_{0}$  & double cover \\
\hline
(n4) & $\F_{4}$ & $h+f$ & $E=E_{1}+E_{2}$ where & $\nu(E)=\nu(E_{1})=\nu(E_{2})$ and  \\
&&&$E_{1}=C_{0}$ and $E_{2} \sim f$&$E_{i} \to \nu(E)$ is isomorphic \\
\hline
\end{tabular}
\caption{Classification of non-normal rational sextic del Pezzo surfaces}
\end{table}

For notation of $\F_{n}$, $h$, $f$ and $C_{0}$, we refer to Subsection~\ref{NC}~(5). 
We refer to \cite{Reid94,AF03} for more precise details. 

(3) When $S$ is neither Du Val nor rational, $S$ is the cone over a curve $C \subset \P^{5}$ of degree $6$ and arithmetic genus $1$. 
\end{thm}

\begin{rem}\label{rem-nocone}
Let $S$ be the cone over a curve $C \subset \P^{6}$ of degree $6$ and arithmetic genus $1$. 
Then we have $\dim T_{v}S=6$ where $v$ is the vertex. 
Hence for any sextic del Pezzo fibration $\vp \colon X \to C$, $X$ does not contains such an $S$ since we assume that the total space $X$ is smooth. 
In other words, every $\vp$-fiber is of type ($i$,$j$), (n2) or (n4) as in Theorem~\ref{thm-GordP6}. 
\end{rem}

\subsection{The aim and main results of Section~\ref{sec-singfib}}

The main aim of this section is to prove Theorem~\ref{mainthm-singfib} and Corollary~\ref{maincor-inv}.  
For this purpose, we will show the following three theorems, which are the main results of this section.

\begin{thm}[{\cite[(4.6)]{Fujita90}}]\label{thm-singfibP22}
Let $q \colon Q \to C$ and $T$ be as in the setting of Theorem~\ref{thm-constP22}. 
Let $\vp_{Z} \colon Z \to C$ be the moderate $(\P^{2})^{2}$-fibration that is obtained by Theorem~\ref{thm-constP22}. 
Then the following assertions hold for any $t \in C$. 
\begin{enumerate}
\item $Q_{t}$ is smooth if and only if $Z_{t} \simeq (\P^{2})^{2}$. 
\item $Q_{t}$ is singular if and only if $Z_{t} \simeq \P^{2,2}$. 
\end{enumerate}
For the definition of $\P^{2,2}$, we refer to Definition \ref{defi-22}. 
\end{thm}

\begin{thm}\label{thm-singfibP13}
Let $f \colon \F \to C$ and $T \subset \F$ be as in the setting of Theorem~\ref{thm-constP13}. 
Let $\vp_{Y} \colon Y \to C$ be the moderate $(\P^{1})^{3}$-fibration which is obtained by Theorem~\ref{thm-constP13}. 
Then the following assertions hold for any $t \in C$. 
\begin{enumerate}
\item $Y_{t} \simeq (\P^{1})^{3}$ if and only if $\#(T_{t})_{\red}=3$. 
\item $Y_{t} \simeq \P^{1} \times \Q^{2}_{0}$ if and only if $\#(T_{t})_{\red}=2$, where $\Q^{2}_{0}$ denotes the $2$-dimensional singular quadric cone. 
\item $Y_{t} \simeq \P^{1,1,1}$ if and only if $\#(T_{t})_{\red}=1$. 
\end{enumerate}
For the definition of $\P^{1,1,1}$, we refer to Definition \ref{defi-111}. 
\end{thm}

\begin{thm}\label{thm-singfibdP6}
Let $\vp \colon X \to C$ be a sextic del Pezzo fibration and $C_{0}$ be an arbitrary $\vp$-section. 
Let $(q \colon Q \to C,T)$ be the relative double projection of $(\vp \colon X \to C,C_{0})$.  
Then the following assertions hold for any $t \in C$. 
\begin{enumerate}
\item For $j \in \{1,2,3\}$, $X_{t}$ is of type $(2,j)$ if and only if $Q_{t}$ is smooth and $\#(T_{t})_{\red}=j$. 
\item For $j \in \{1,2,3\}$, $X_{t}$ is of type $(1,j)$ if and only if $Q_{t}$ is singular, $\#(T_{t})_{\red}=j$ and $\Sing Q_{t} \cap T_{t}=\emp$. 
\item $X_{t}$ is of type (n2) if and only if $Q_{t}$ is singular, $\# (T_{t})_{\red}=2$ and the double point of $T_{t}$ is supported at the vertex of $Q_{t}$. 
\item $X_{t}$ is of type (n4) if and only if $Q_{t}$ is singular, $\# (T_{t})_{\red}=1$ and $\Sing Q_{t}=(T_{t})_{\red}$. 
\end{enumerate}
\end{thm}

\begin{proof}[Proof of Theorem~\ref{mainthm-singfib} assuming Theorems~\ref{thm-singfibP22}, \ref{thm-singfibP13} and \ref{thm-singfibdP6}] 
Fix a $\vp$-section $C_{0}$ and let $(q \colon Q \to C,T)$ be the relative double projection.  Then Theorem~\ref{mainthm-singfib} immediately follows from Theorems~\ref{thm-singfibP22}, \ref{thm-singfibP13} and \ref{thm-singfibdP6} since for any $t \in C$, $\#(B_{t})_{\red}=2$ (resp. $1$) if and only if $Q_{t}$ is smooth (resp. the quadric cone) by definition. 
\end{proof}


Corollary~\ref{maincor-inv} follows directly from Theorems~\ref{mainthm-inv}~and~\ref{mainthm-singfib}. 

\begin{proof}[Proof of Corollary~\ref{maincor-inv}]
The items (1) and (2) follow immediately from Theorem~\ref{mainthm-inv}~(3). 
To show (3), we compute $(-K_{X/C})^{3}=(-K_{X}+\vp^{\ast}K_{C})^{3}$. 
By Theorem~\ref{mainthm-inv}~(3), we have 
$(-K_{X}+\vp^{\ast}K_{C})^{3}
=(-K_{X})^{3}+3(-K_{X})^{2}(\vp^{\ast}K_{C})
=24g(C)-(6g(B)+4g(T)+14)$. 
By the Hurwitz formula, when $R_{B}$ and $R_{T}$ denote the ramification divisor of $\vp_{B}$ and $\vp_{T}$ respectively, we have 
$\deg R_{B}=2g(B)+2-4g(C)$ and $\deg R_{T}=2g(T)+4-6g(C)$. 
Thus we have $(-K_{X/C})^{3}=-(3\deg R_{B}+2\deg R_{T}) \leq 0$. 
Hence $(-K_{X/C})^{3}=0$ if and only if $\deg R_{B}=\deg R_{T}=0$, which is equivalent to that $\vp_{B}$ and $\vp_{T}$ are \'etale. 
By Theorem~\ref{mainthm-singfib}, it is equivalent to that $\vp$ is smooth, which completes the proof. 
\end{proof}

\subsection{Definition of $\P^{2,2}$ and proof of Theorem~\ref{thm-singfibP22}}
Now our aim is reduced to prove Theorems~\ref{thm-singfibP22}, \ref{thm-singfibP13} and \ref{thm-singfibdP6}. 
In this subsection, we define $\P^{2,2}$ and prove Theorem~\ref{thm-singfibP22}. 

\begin{defi}[\cite{Mukai-lect}]\label{defi-22}
We define 
\[\wt{\P^{2,2}}:=\P_{\P^{2}}(\mc O_{\P^{2}}(2) \oplus \Omega_{\P^{2}}(2)).\]
Then the tautological divisor $\xi$ is free and the linear system $|\xi|$ gives a morphism $\wt{\P^{2,2}} \to \P^{8}$ since $h^{0}(\P^{2},\mc O_{\P^{2}}(2) \oplus \Omega_{\P^{2}}(2))=9$. 
We define $\P^{2,2}$ as the image of the morphism and $\psi \colon \wt{\P^{2,2}} \to \P^{2,2}$. 
\end{defi}

\begin{rem}
The notation $\P^{2,2}$ was introduced by Mukai in his talk \cite{Mukai-lect}. 
\end{rem}

\begin{lem}[{\cite{Fujita86}, \cite{Mukai-lect}}]\label{lem-P22}
$\P^{2,2}$ is the del Pezzo variety $V$ of type (vu) in \cite{Fujita86}. 
Moreover, the following assertions hold. 
\begin{enumerate}
\item $\wt{\P^{2,2}}$ is a weak Fano $4$-fold and $\P^{2,2}$ is its anti-canonical model. The morphism $\psi \colon \wt{\P^{2,2}} \to \P^{2,2}$ is a crepant divisorial contraction and its exceptional divisor $\Exc(\psi)$ is the unique member of $|\xi-\pi^{\ast}\mc O_{\P^{2}}(2)|$. 
\item The embedding $\P^{2,2} \hra \P^{8}$ factors through the weighted projective space $\P(1^{3},2^{3}) \subset \P^{8}$.  In $\P(1^{3},2^{3})$, $\P^{2,2}$ is a cubic hypersurface. 
Moreover, $\Sing (\P^{2,2})=\Sing(\P(1^{3},2^{3}))$ and the singularity is a family of Du Val $A_{1}$-singularities. 
\end{enumerate}
\end{lem}

\begin{proof}
We recall the definition of the del Pezzo variety $V$ of type (vu) in \cite{Fujita86}. 
The projective space bundle $\P_{\P^{2}}(\mc O^{3} \oplus \mc O(2))$ is the blow-up of $\P(1^{3},2^{3})$ along the singular locus $\Pi$, which is a $2$-plane on $\P(1^{3},2^{3}) \subset \P^{8}$. 
Let $p \colon \P_{\P^{2}}(\mc O^{3} \oplus \mc O(2)) \to \P^{2}$ be the projection, $\Xi$ a tautological divisor and $f \colon \P_{\P^{2}}(\mc O^{3} \oplus \mc O(2)) \to \P(1^{3},2^{3})$ the blow-up. 
Then $\Exc(f)$ is the unique member of $|\mc O(\Xi) \otimes p^{\ast}\mc O(-2)|$ and isomorphic to $\P^{2} \times \P^{2}$. 
Consider a member $\wt{V} \in |\mc O(\Xi) \otimes p^{\ast}\mc O(1)|$ such that $E \cap \wt{V}$ is isomorphic to $\P(T_{\P^{2}})$. 
The del Pezzo variety of $V$ of type (vu) is defined to be the image $f(\wt{V})$ (see \cite[p.155]{Fujita86}). 
Then $V$ is a cubic hypersurface of $\P(1^{3},2^{3})$ containing $\Pi$ and $f|_{\wt{V}} \colon \wt{V} \to V$ is a divisorial contraction whose exceptional divisor is $E \cap \wt{V}=\P(T_{\P^{2}})$. 
Since $H^{i}(p^{\ast}\mc O(-1))=0$ for any $i$, 
$f|_{\wt{V}}$ is given by $|\Xi|_{\wt{V}}|$. 

Now it suffices to show that $\wt{V} \simeq \wt{\P^{2,2}}$ and $\Xi|_{\wt{V}}=\xi$. 
Such a member $\wt{V}$ corresponds to an injection $v=(s_{1},s_{2},s_{3},t) \colon \mc O(-1) \to \mc O^{3} \oplus \mc O(2)$ such that $s_{1},s_{2},s_{3} \in H^{0}(\P^{2},\mc O(1))$ forms a basis of $H^{0}(\P^{2},\mc O(1))$. 
Since $(s_{1}=s_{2}=s_{3}=0)=\emp$, $\Cok v$ is a locally free sheaf of rank $3$ and 
$\wt{V}$ is isomorphic to $\P_{\P^{2}}(\Cok v)$. 
Now $\Xi|_{\wt{V}}$ is a tautological divisor of $\P_{\P^{2}}(\Cok v)$. 
Letting $s=(s_{1},s_{2},s_{3}) \colon \mc O(-1) \to \mc O^{3}$, we have $\Cok s \simeq \Omega_{\P^{2}}(2)$ and the exact sequence $0 \to \mc O_{\P^{2}}(2) \to \Cok v \to \Omega_{\P^{2}}(2) \to 0$, which splits since $\Ext^{1}(\Omega_{\P^{2}}(2),\mc O_{\P^{2}}(2))=H^{1}(\P^{2},T_{\P^{2}})=0$. 
Therefore, $\wt{V}$ is isomorphic to $\wt{\P^{2,2}}$. 
Since $-K_{\wt{\P^{2,2}}}=3\xi$, the morphism $\psi$ is a crepant divisorial contraction. 
The remaining items follow from the above construction. 
\end{proof}

Then Theorem~\ref{thm-singfibP22} is just a corollary of a result of Fujita {\cite[(4.6)]{Fujita90}}. 

\begin{proof}[Proof of Theorem~\ref{thm-singfibP22}]
Set $\mc F$ as in Theorem~\ref{thm-constP22}~(1). 
By Theorem~\ref{thm-constP22}~(2), the blow-up of $Z_{t}$ at a smooth point is the flop of $\P_{Q_{t}}(\mc F_{Q_{t}})$. 
Moreover, we see that $Q_{t}$ is smooth if and only if $Z_{t}$ is smooth since the map $\Psi$ in the diagram (\ref{dia-P22}) is a family of Atiyah flop over $C$. 
Hence it suffices to show that $Z_{t} \simeq \P^{2,2}$ if $Z_{t}$ is singular, which is known by \cite[(4.6)]{Fujita90}. 
\end{proof}

Let us consider codimension $2$ linear sections of $(\P^{2})^{2}$ or $\P^{2,2}$. 
The following lemma is useful to classify linear sections of such kind of varieties. 

\begin{lem}\label{lem-irr}
Let $X$ be a smooth variety and let $\mc E$ and $\mc F$ be locally free sheaves on $X$ with $\rk \mc F=\rk \mc E-1$. 
Let $s \colon \mc F \to \mc E$ be an injection. 
Let $\pi \colon \P(\mc E) \to S$ be the projectivization and $\wt{s} \in H^{0}(\P_{X}(\mc E), \mc O_{\P(\mc E)}(1) \otimes \pi^{\ast}\mc F^{\vee})$ the section corresponding to $s$. 
Set $Z=\{p \in X \mid \rk s(p) < \rk \mc F\}$ and $Y=(\wt{s}=0)$. 
\begin{enumerate}
\item The following conditions are equivalent:
\begin{enumerate}
\item $\codim Z \geq 2$. 
\item $\Cok s$ is torsion free. 
\end{enumerate}
Moreover, if one of the above conditions holds, then $\Cok s \simeq \mc L \otimes \mc I_{Z}$ where $\mc L:=(\det \mc E) \otimes (\det \mc F)^{\vee}$. 
\item If $Y$ is irreducible, then $\Cok s$ is torsion free and $Y=\Bl_{Z}X$ holds. 
\end{enumerate}
\end{lem}
\begin{proof}
Set $U:=X \setminus Z$ and $\ol{Y}:=\ol{\pi^{-1}(U) \cap Y}$. 
For any $p \in X$,  we have $\pi^{-1}(p) \cap Y \simeq \P^{\rk \mc F-\rk s(p)}$. 
Hence $\ol{Y} \to X$ and $Y \to X$ are isomorphic over $U$ and $\ol{Y}$ is an irreducible component of $Y$. 
Note that $U$ coincides with the set $\{p\in X\mid (\Cok s)_{p} \text{ is free } \mc O_{X,p} \text{-module}\}$. 

(1) The implication (a) $\ra$ (b) holds by \cite[Lemma~5.4]{Ohno14}. 
Conversely, if $\Cok s$ is torsion free, then it is well-known that $\codim Z \geq 2$ since $X$ is smooth. 
Moreover, the natural map $\Cok s \to (\Cok s)^{\vee\vee}$ is injective and its cokernel $\mc T$ is a torsion sheaf with $\Supp \mc T=Z$. 
Set $\mc L= \det \mc E \otimes \det \mc F^{\vee} \simeq (\Cok s)^{\vee\vee}$. 
Then we have an exact sequence $0 \to \Cok s \otimes \mc L^{-1} \to \mc O \to \mc T \otimes \mc L \to 0$. 
Thus $\Cok s \otimes \mc L^{-1}$ is an ideal sheaf of $Z$ since the scheme theoretic support 
$\Supp (\mc T \otimes \mc L)$ is $Z$. 

(2) Suppose that $Y$ is irreducible, i.e., $Y=\ol{Y}$. 
If $Z$ contains a prime divisor $D$ on $X$, then we have $\dim (\pi|_{Y})^{-1}(D) \geq \dim D+1=\dim X$ and hence $Y$ is not irreducible, which is a contradiction. 
Thus $\codim_{X}Z \geq 2$ holds and hence $\Cok s$ is torsion free by (1). 
The surjection $\mc E \epm \mc L \otimes \mc I_{Z}$ gives an embedding $\Bl_{Z}X \hra \P_{X}(\mc E)$. 
Since $\Bl_{Z}X$ coincides with $Y$ over $U$ and $Y$ and $\Bl_{Z}X$ are irreducible, $Y$ coincides with $\Bl_{Z}X$ as sets. 
Since $Y$ is the zero scheme of $\wt{s}$ and $\codim_{\P_{X}(\mc E)}Y=\rk \mc F$, $Y$ is reduced. 
Thus we have $Y=\Bl_{Z}X$ as schemes. 
\end{proof}

The following lemma will be used for Theorem~\ref{thm-singfibdP6}. 

\begin{lem}\label{lem-linsecP22}
Let $S$ be a Du Val sextic del Pezzo surface. 
If $S$ is a codimension $2$ linear section of $(\P^{2})^{2}$ (resp. $\P^{2,2}$) with respect to the embedding into $\P^{8}$, then $S$ is of type (2,$j$) (resp. (1,$j$)). 
\end{lem}
\begin{proof}
Suppose that $S$ is a codimension $2$ linear section of $\P^{2} \times \P^{2}$ (resp. $\P^{2,2}$). 
Letting $\mc F=\mc O(1)^{3}$ (resp. $\mc F=\Omega_{\P^{2}}(2) \oplus \mc O(2)$), 
$S$ (resp. $\psi^{-1}(S)$) is the complete intersection of two members of $|\mc O_{\P(\mc F)}(1)|$. 
Let $s_{1},s_{2} \colon \mc O \to \mc F$ be the corresponding two morphisms and $s=\!^{t}(s_{1},s_{2}) \colon \mc O^{2} \to \mc F$. 
Then Lemma~\ref{lem-irr} gives the exact sequence $0 \to \mc O^{2} \mathop{\to}^{s}  \mc F \to \mc I_{\Sigma/\P^{2}}(3) \to 0$. 
Note that $S=\Bl_{\Sigma}\P^{2}$ (resp. $\psi^{-1}(S)=\Bl_{\Sigma}\P^{2}$) and $\Sigma$ is a $0$-dimensional subscheme of length $c_{2}(\mc F)=3$. 
Tensoring the above exact sequence with $\mc O(-2)$, we get $h^{0}(\mc I_{\Sigma}(1))=0$ (resp. $1$), which implies that $S$ is of type (2,$j$) (resp. (1,$j$)) by Theorem~\ref{thm-GordP6}. 
\end{proof}

\subsection{Definition of $\P^{1,1,1}$}

To define $\P^{1,1,1}$, we need the following lemma. 

\begin{lem}\label{lem-wtP111}
Let $\e \in \Ext^{1}_{\F_{2}}(\mc O(f),\mc O(2h-f))=H^{1}(\F_{2},\mc O_{\F_{2}}(2h-2f))=\C$ be a non-zero element. 
This element gives the following non-trivial extension 
\begin{align}\label{ex-111}
0 \to \mc O(2h-f) \to \mc E \to \mc O(f) \to 0.
\end{align}
Then $\mc E$ is the cokernel of an injection $\mc O_{\F_{2}}(-h+f) \to \mc O^{2} \oplus \mc O(h+f)$. In particular, $\mc E$ is globally generated. 
\end{lem}
\begin{proof}
Let $s_{1},s_{2} \in H^{0}(\mc O_{\F_{2}}(h-f))$ be sections such that $(s_{i}=0)=C_{0}+l_{i}$, where $l_{i} \in |f|$ and $l_{1} \cap l_{2}=\emp$. 
Let $t \in H^{0}(\mc O_{\F_{2}}(2h))$ be a general section such that $(t=0) \cap C_{0}=\emp$. 
Then the cokernel of the map 
$v:=(s_{1},s_{2},t) \colon \mc O_{\F_{2}}(-h+f) \to \mc O^{2} \oplus \mc O(h+f)$ 
is locally free. 
Let us confirm that $\Cok v=\mc E$. 
Set $s=(s_{1},s_{2}) \colon \mc O_{\F_{2}}(-h+f) \to \mc O_{\F_{2}}^{2}$. 
Then we have a surjection $\Cok v \epm \Cok s$. 
Since $(s_{1}=0)$ and $(s_{2}=0)$ contain $C_{0}$, 
the morphism $s \colon \mc O(-h+f) \to \mc O^{2}$ factors through the map $\mc O(-f) \to \mc O^{2}$ which is given by $(l_{1},l_{2})$.  
Hence we have a surjection $\Cok s \to \mc O(f)$ and hence a surjection $\Cok v \to \mc O(f)$ and an exact sequence $0 \to \mc O(2h-f) \to \Cok v \to \mc O(f) \to 0$, which does not split since $\mc O(2h-f)$ is not globally generated but $\Cok v$ is. 
Since $\Ext^{1}(\mc O(f),\mc O(2h-f))=\C$, we have $\Cok v \simeq \mc E$. 
\end{proof}


\begin{defi}[\cite{Fujita86}]\label{defi-111}
Let $\mc E$ be the bundle on $\F_{2}$ fitting into the exact sequence (\ref{ex-111}) that does not split. 
We define $\wt{\P^{1,1,1}}:=\P_{\F_{2}}(\mc E)$ and set $\P^{1,1,1} \subset \P^{7}$ be the image of the morphism defined by $|\mc O_{\P_{\F_{2}}(\mc E)}(1)|$. 
\end{defi}
\begin{lem}[\cite{Fujita86}]\label{lem-P111}
$\P^{1,1,1}$ is the del Pezzo variety of type (si31i) in \cite{Fujita86}. 
Moreover, the following assertions hold. 
\begin{enumerate}
\item $\wt{\P^{1,1,1}}$ is a weak Fano $3$-fold and $\P^{1,1,1}$ is its anti-canonical model. 
The morphism $\psi \colon \wt{\P^{1,1,1}} \to \P^{1,1,1}$ is a crepant divisorial contraction. 
If $\pi \colon \P_{\F_{2}}(\mc E) \to \F_{2}$ denotes the projection, 
then the exceptional divisors of $\P_{\F_{2}}(\mc E) \to \P^{1,1,1}$ is the union of the unique members $E_{1} \in |\mc O_{\P(\mc E)}(1) \otimes \pi^{\ast}\mc O(-2h+f)|$ and $E_{2}:=\pi^{\ast}C_{0}$. 
$E_{1}$ is a $\pi$-section and hence isomorphic to $\F_{2}$ and $E_{2} \simeq \P^{1} \times \P^{1}$. 
$\psi(E_{1})$ and $\psi(E_{2})$ are the same line on $\P^{1,1,1}$ and $\psi|_{E_{i}} \colon E_{i} \to \P^{1}$ is a $\P^{1}$-bundle. 
The intersection $E_{1} \cap E_{2}$ is a section of $\psi|_{E_{i}}$ for each $i$. 
\item $\P^{1,1,1}$ is a Gorenstein canonical del Pezzo $3$-fold of Picard rank $1$. 
The singular locus of $\P^{1,1,1}$ is a line and $\P^{1,1,1}$ has a family of Du Val $A_{2}$-singularities along the line. 
\end{enumerate}
\end{lem}
\begin{proof}
Let $\Phi \colon \P_{\F_{2}}(\mc O^{2} \oplus \mc O(h+f)) \to \P^{7}$ be the morphism given by the complete linear system of the tautological line bundle. 
By Lemma~\ref{lem-wtP111}, $\wt{\P^{1,1,1}}$ is a sub $\P^{1}$-bundle of $\P_{\F_{2}}(\mc O^{2} \oplus \mc O(h+f))$. 
In \cite[p.170]{Fujita86}, the del Pezzo variety of type (si31i) is defined as the image of $\wt{\P^{1,1,1}}$ under the morphism $\Phi$. 
Since $H^{i}(\F_{2},\mc O(-h+f))=0$ holds for any $i$, $\Phi|_{\wt{\P^{1,1,1}}}$ is also given by $|\mc O_{\P(\mc E)}(1)|$. 
Thus $\P^{1,1,1}$ is the del Pezzo variety of type (si31i). 

Let us show (1). Since $\mc O(-K_{\P_{\F_{2}}(\mc E)})=\mc O_{\P_{\F_{2}}(\mc E)}(2)$, 
$\psi \colon \wt{\P^{1,1,1}} \to \P^{1,1,1}$ is a crepant birational morphism. 
$E_{1}+E_{2}$ is a member of $|\mc O_{\P(\mc E)}(1) \otimes \pi^{\ast}\mc O(-h-f)|=|\Exc(\Phi)|_{\wt{\P^{1,1,1}}}|$ by the definitions. 
Since $h^{0}(\mc E(-h-f))=1$ by (\ref{ex-111}), 
we have $E_{1}+E_{2}=\Exc(\Phi)|_{\wt{\P^{1,1,1}}}$. 
It is clear that $E_{1}$ is a $\pi$-section. 
As divisors on $\Exc(\Phi)=\F_{2} \times \P^{1}$, 
$E_{1}+E_{2}$ is a member of $|\pr_{1}^{\ast}\mc O(h-f) \otimes \pr_{2}^{\ast}\mc O_{\P^{1}}(1)|$. 
Since $E_{2} \in |\pr_{1}^{\ast}\mc O(h-2f)|$, we have $E_{1} \in |\pr_{1}^{\ast}\mc O(f) \otimes \pr_{2}^{\ast}\mc O_{\P^{1}}(1)|$, which implies that $E_{2}=C_{0} \times \P^{1} \simeq (\P^{1})^{2}$ and $E_{1} \cap E_{2}$ is a section of $\pr_{2}$. 
It also follows that $\pr_{2}|_{E_{i}} \colon E_{i} \to \P^{1}$ gives a $\P^{1}$-bundle structure for each $i$. In particular, we have $E_{1}+E_{2}=\Exc(\psi)$. 
(2) immediately follows from (1). 
\end{proof}

\subsection{Reducing Theorems~\ref{thm-singfibP13}\ and \ref{thm-singfibdP6}\ to Propositions~\ref{prop-singfibP13}\ and \ref{prop-hypsecP13}}

In order to prove Theorems~\ref{thm-singfibP13}~and~\ref{thm-singfibdP6}, we reduce these theorems to explicit descriptions of birational transforms of the blow-up of $\P^{3}$ along a $0$-dimensional subscheme of length $3$. 
First, we reduce Theorem~\ref{thm-singfibP13} to the following proposition. 
\begin{prop}\label{prop-singfibP13}
Let $\Sigma \subset \P^{3}$ be a $0$-dimensional closed subscheme of length $3$. 
Suppose that the minimal linear subspace $\braket{\Sigma}$ containing $\Sigma$ as a subscheme is a hyperplane on $\P^{3}$. 
Let $\s_{V} \colon V=\Bl_{\Sigma}\P^{3} \to \P^{3}$ be the blow-up and $G$ the Cartier divisor such that $\mc O_{V}(-G)=\s_{V}^{-1}\mc I_{\Sigma/\P^{3}} \cdot \mc O_{V}$. 

Then $V$ is a canonical Gorenstein weak Fano $3$-fold and the contraction morphism $\psi_{V} \colon V \to \ol{V}:=\Proj \bigoplus_{n \geq 0} H^{0}(V,\mc O(-nK_{V}))$ is a flopping contraction. Let $\chi_{V} \colon V \dra V^{+}$ be the flop and $P_{V^{+}}$ be the proper transformation of the $2$-plane $\braket{\Sigma}$ on $V^{+}$. 
Then there exists the blow-down 
$\mu_{V} \colon V^{+} \to \ol{W}$ 
which contracts $P_{V^{+}}$ to a smooth point on $\ol{W}$. 
Moreover, $\ol{W}$ is one of the following. 
\begin{enumerate}
\item If $\Sigma$ is reduced, then $\ol{W} \simeq (\P^{1})^{3}$. 
\item If $\Sigma$ is isomorphic to $\Spec \C[\e]/(\e^{2}) \sqcup \Spec \C$, 
then $\ol{W} \simeq \P^{1} \times \Q^{2}_{0}$, where $\Q^{2}_{0}$ is the cone over a conic. 
\item If $\Sigma$ is isomorphic to $\Spec \C[\e]/(\e^{3})$, 
then $\ol{W} \simeq \P^{1,1,1}$. 
\end{enumerate}
\end{prop}

\begin{proof}[Proof of Theorem~\ref{thm-singfibP13} assuming Proposition~\ref{prop-singfibP13}.]
By Theorem~\ref{thm-constP13}, the blow-up of $\P^{3}=\F_{t}$ at the $0$-dimensional scheme $\Sigma:=T_{t}$ of length $3$ has a flop $\Phi_{t} \colon \Bl_{\Sigma}\P^{3} \dra \wt{Y}_{t}$. 
Moreover, the morphism $\wt{Y}_{t} \to Y_{t}$ contracts the proper transform of the $2$-plane $\braket{\Sigma}$ on $\wt{Y}_{t}$ to a smooth point on $Y_{t}$. 
Then Theorem~\ref{thm-singfibP13} is a direct corollary of Proposition~\ref{prop-singfibP13}. 
\end{proof}

\begin{prop}\label{prop-hypsecP13}
As in Proposition~\ref{prop-singfibP13}, consider the following diagram. 
\begin{align}\label{dia-P13sp}
\xymatrix{
&\Bl_{x}\ol{W}=V^{+} \ar[ld]_{\mu_{V}}&V\ar[rd]^{\s_{V}} \ar@{-->}[l]_{\qquad \chi_{V}}& \\
\ol{W}&&&\P^{3}.
}
\end{align}
Let $X$ be a prime hyperplane section of $\ol{W}$ containing $x$ as a smooth point. 
Then the proper transform of $X$ on $\P^{3}$ is an irreducible and reduced quadric surface $Q$ containing $\Sigma$. 
Moreover, the following assertions hold. 
\begin{enumerate}
\item If $\# \Sigma=3$, then $X$ is of type (2,3) or (1,3). 
Moreover, $Q$ is smooth at every point of $\Sigma$. 
\item If $\# \Sigma_{\red}=2$, then $X$ is of type (2,2), (1,2) or (n2). 
In this case, if $p_{1}$ denote the reduced point of $\Sigma$ and $p_{2}$ the double point of $\Sigma$, then $Q$ is smooth at $p_{1}$. 
Moreover, $X$ is of type (2,2) or (1,2) (resp. (n2)) if and only if $Q$ is smooth (resp. singular) at $p_{2}$. 
\item If $\# \Sigma_{\red}=1$, then $X$ is of type (2,1), (1,1) or (n4). 
In this case, $X$ is of type (2,1) or (1,1) (resp. (n4)) if and only if $Q$ is smooth (resp. singular) at $\Sigma_{\red}$. 
\end{enumerate}
\end{prop}

\begin{proof}[Proof of Theorem~\ref{thm-singfibdP6} assuming Proposition~\ref{prop-hypsecP13}]
Let $\vp \colon X \to C$ be a sextic del Pezzo fibration and $C_{0}$ a $\vp$-section. 
Let $(q \colon Q \to C,T)$ be the relative double projection of $(\vp \colon X \to C,C_{0})$. 
Let $f \colon \F \to C$ be the $\P^{3}$-bundle containing $Q$ as in the proof of Theorem~\ref{mainthm-P13} (see Subsection~\ref{subsec-proofmainthmP13}). 
In Theorems \ref{mainthm-P13} and \ref{mainthm-P22}, 
we constructed a moderate $(\P^{1})^{3}$-fibration $\vp_{Y} \colon Y \to C$ and a moderate $(\P^{2})^{2}$-fibration $\vp_{Z} \colon Z \to C$ such that $Y$ and $Z$ contain $X$ as a relatively linear section over $C$. 
Let us fix $t \in C$. 
Then $X_{t}$ is a linear section of $Y_{t}$ and $Z_{t}$ and $x:=C_{0,t}$ is a smooth point of $X_{t}$ and hence that of $Y_{t}$ and $Z_{t}$. 

\begin{claim}\label{claim-ij}
If $X_{t}$ is Du Val, then $X_{t}$ is of type ($\#(B_{t})_{\red}$, $\#(T_{t})_{\red}$). 
\end{claim}
\begin{proof}
Take $i \in \{1,2\}$ and $j \in \{1,2,3\}$ such that $X_{t}$ is of type $(i,j)$. 
First, we show $i=\#(B_{t})_{\red}$. 
By Theorem~\ref{thm-singfibP22}, $\#(B_{t})_{\red}=2$ or $1$ if and only if $Z_{t}=(\P^{2})^{2}$ or $\P^{2,2}$ respectively. 
Since $X_{t}$ is a codimension $2$ linear section of $Z_{t}$, Lemma~\ref{lem-linsecP22} implies $\#(B_{t})_{\red}=i$. 
Next, we show $\#(T_{t})_{\red}=j$. 
Letting $\Sigma:=T_{t} \subset \F_{t}=\P^{3}$, $Y_{t}$ is nothing but $\ol{W}$ in Proposition~\ref{prop-singfibP13}. 
Then $\#(T_{t})_{\red}=j$ follows from Proposition~\ref{prop-hypsecP13}. 
\end{proof}
Let us prove Theorem~\ref{thm-singfibdP6}. 

(1) Assume $X_{t}$ is of type (2,$j$). 
By the definition of $B$, $\#(B_{t})_{\red}=2$ (resp. $1$) if and only if $Q_{t}$ is smooth (resp. singular). 
Claim~\ref{claim-ij} shows that $Q_{t}$ is smooth and $\#(T_{t})_{\red}=j$. 
Conversely, let us assume that $Q_{t}$ is smooth and $\# (T_{t})_{\red}=j$. 
Then $X_{t}$ is normal by Proposition~\ref{prop-hypsecP13}. 
By Claim~\ref{claim-ij}, $X_{t}$ is of type (2,$j$). 

(2) Assume $X_{t}$ is of type (1,$j$).  
Then Claim~\ref{claim-ij} shows that $Q_{t}$ is singular and $\#(T_{t})_{\red}=j$. 
Moreover, it follows that $\Sing Q_{t} \cap (T_{t})_{\red} = \emp$ from Proposition~\ref{prop-hypsecP13}. 
Conversely, let us assume that $Q_{t}$ is singular, $\# (T_{t})_{\red}=j$ and $\Sing Q_{t} \cap (T_{t})_{\red}=\emp$. 
Then Proposition~\ref{prop-hypsecP13} and Claim~\ref{claim-ij} shows that $X_{t}$ is of type (1,$j$). 

(3) and (4) follow directly from Proposition~\ref{prop-hypsecP13}~(2) and (3). 
\end{proof}

\subsection{Proofs of Propositions~\ref{prop-singfibP13}\ and \ref{prop-hypsecP13}}
To complete the aim, we devote the remaining part of this section for proving Propsositions~\ref{prop-singfibP13} and \ref{prop-hypsecP13}. 
We proceed in the following order: 
Proposition~\ref{prop-singfibP13}~(1); 
Proposition~\ref{prop-hypsecP13}~(1); 
Proposition~\ref{prop-singfibP13}~(2); 
Proposition~\ref{prop-hypsecP13}~(2); 
Proposition~\ref{prop-singfibP13}~(3); 
Proposition~\ref{prop-hypsecP13}~(3). 

\subsubsection{Proof of Propositions~\ref{prop-singfibP13}~(1) and \ref{prop-hypsecP13}~(1)}
Suppose that $\Sigma$ is reduced. 
Note that such $\Sigma$ is unique up to projection equivalences on $\P^{3}$. 
Thus, as we saw in the proof of Theorem~\ref{thm-constP13}, we conclude Proposition~\ref{prop-singfibP13}~(1) by taking the proper transform of general fibers in the diagram~(\ref{dia-P13}). 
Let us show Proposition~\ref{prop-hypsecP13}~(1). 
Set $\mc E=\mc O_{(\P^{1})^{2}}(1,1)^{2}$. 
Then $X$ is a member of $|\mc O_{\P(\mc E)}(1)|$. 
Lemma~\ref{lem-irr} gives an exact sequence 
\[0 \to \mc O \to \mc O(1,1)^{2} \to \mc I_{Z}(2,2) \to 0\]
on $(\P^{1})^{2}$ and $\Bl_{Z}(\P^{1})^{2}=X$, where $Z$ is a $0$-dimensional subscheme with $\lgth(\mc O_{Z})=c_{2}(\mc E)=2$. 
By seeing the above sequence, no lines in the quadric $(\P^{1})^{2} \subset \P^{3}$ contain $Z$. 
By counting the number of the lines on $X$, we can conclude that $X$ is of type (2,3) or (2,2). 
Let $H$ be a hyperplane section of $\P^{3}$. 
Under the birational map in the diagram (\ref{dia-P13sp}), 
we have 
$(\chi_{V})^{-1}_{\ast}(\mu_{V})^{-1}_{\ast}X \sim (\chi_{V})^{-1}_{\ast}(\mu_{V}^{\ast}X-\Exc(\mu_{V})) \sim (\s_{V}^{\ast}(3H)-2\Exc(\s_{V}))-(\s_{V}^{\ast}(H)-\Exc(\s_{V})) \sim \s_{V}^{\ast}(2H)-\Exc(\s)$. 
Thus $Q$ is a quadric on $\P^{3}$ and smooth at all points of $\Sigma$. 
\qed

\subsubsection{Proof of Proposition~\ref{prop-singfibP13}~(2)}
We proceed in 5 steps. 

\noindent\textbf{Step 1. } In this step, we set our basic preliminaries and notation. 

Let $p_{1} \in \Sigma$ be the unique reduced point of $\Sigma$ and let $p_{2}=(\Sigma \setminus p_{1})_{\red}$. 
Let $f \colon \Bl_{p_{1},p_{2}}\P^{3} \to \P^{3}$ be the blow-up and $E_{1}'$ and $E_{2}'$ the $f$-exceptional divisors which dominates $p_{1}$ and $p_{2}$ respectively. 
Let $p_{3} \in E_{2}'$ be the point corresponding to the tangent vector of $\Sigma$ at $p_{2}$. 
Set $g \colon M:=\Bl_{p_{3}}\Bl_{p_{1},p_{2}}\P^{3} \to \Bl_{p_{1},p_{2}}\P^{3}$ and $\s:=f \circ g$. 
Set $H:=\s^{\ast}\mc O_{\P^{3}}(1)$, $E_{3}:=\Exc(g)$, $E_{i}:=g^{-1}_{\ast}E_{i}'$ for $i=1,2$. 
Let $P$ be the proper transform of the linear hull of $\Sigma$. 
Since $K_{M} \sim \s^{\ast}K_{\P^{3}}+2E_{1}+2E_{2}+4E_{3}$, we have 
\begin{align}\label{eq-KP-P11,1}
P \in |H-(E_{1}+E_{2}+2E_{3})| \text{ and } -K_{M} \sim 2(H+P). 
\end{align}
Since $\s^{-1}\mc I_{\Sigma} \cdot \mc O_{M}=\mc O(-E_{1}-E_{2}-2E_{3})$, there is a natural morphism 
\[\tau \colon M \to V:=\Bl_{\Sigma}\P^{3}\]
by the universal property of the blow-up. 
Note that $\tau(E_{2})=C \simeq \P^{1}$ and $\tau$ is a crepant divisorial contraction. 
Thus $\Bl_{\Sigma}\P^{3}$ has a family of Du Val $A_{1}$-singularities along its singular locus $C$. 

Finally, let $l_{12}$ (resp. $l_{23}$) be the proper transform of the linear hull of $\Sigma_{\red}=\{p_{1},p_{2}\}$ (resp. $\Sigma \setminus \{p_{1}\}$) and $e_{i}=E_{i} \cap P$ for $i=1,2,3$. 
As divisors on $P$, we have 
\begin{align}\label{eq-le-P11,1}
l_{12} \sim H|_{P}-(e_{1}+e_{2}+e_{3}) \text{ and } l_{23} \sim H|_{P}-(e_{2}+2e_{3}). 
\end{align}

\noindent
\textbf{Step 2. } In this step, we show that $V$ is weak Fano (and so is $M$) and the morphism $\psi_{V}$ is small. We also construct a birational link from $M$ explicitly. 

First, we confirm that $-K_{V}$ is nef, big and divisible by $2$, which is equivalent to the following claim. 
\begin{claim}\label{claim-MwF}
$M$ is weak Fano and $-K_{M}$ is divisible by $2$ in $\Pic(M)$. 
\end{claim}
\begin{proof}
Since $-K_{M} \sim 2(H+P)$ by (\ref{eq-KP-P11,1}) and $(-K_{M})^{3}=40$, it suffices to see that $(H+P)|_{P}$ is nef. 
By (\ref{eq-KP-P11,1}) and (\ref{eq-le-P11,1}), we have $(H+P)|_{P} \sim l_{12}+l_{23}+e_{2}+e_{3}$, which is nef. 
\end{proof}
Next, we show that the morphism from $V$ onto the anti-canonical model $\ol{V}$ is flopping. 
From the proof of Claim~\ref{claim-MwF}, it follows that a curve $\g \subset P$ satisfying $-K_{M}.\g=0$ is $l_{12}$, $l_{23}$ or $e_{2}$. 
Set $H_{V}:=\tau_{\ast}H$ and $P_{V}:=\tau_{\ast}P$. 
Then we get $-K_{V} \sim 2(H_{V}+P_{V})$, which is ample over $\P^{3}$ from our construction. 
Therefore, if a curve $\g \subset V$ satisfies $-K_{V}.\g=0$, then we get $\g \subset P_{V}$. 
Since the morphism $\tau|_{P} \colon P \to P_{V}$ is the contraction of the $(-2)$-curve $e_{2}$ on $P$, it holds that $\g=\tau(l_{12})$ or $\tau(l_{23})$. 
Hence $V \to \ol{V}$ is small. 

Let us see the flop of $l_{12}$ and $l_{23}$ on $M$ explicitly. 
Note that $P \subset M$ contains $l_{12},l_{23}$ as $(-1)$-curves and $e_{2}$ as a $(-2)$-curve. 
Hence we can see that $\mc N_{l_{12}/M},\mc N_{l_{23}/M} \simeq \mc O(-1)^{2}$. 
Let $\chi_{1} \colon M \dra M^{+}$ be the flop of $l_{12}$ and $l_{23}$. 
Set $e_{2}^{+}:=\chi_{1\ast}(e_{2})$ and $P^{+}:=\chi_{1\ast}(P)$. 
Then $e_{2}^{+}$ is a $(-1)$-curve of $P^{+}$ and hence $\mc N_{e_{2}^{+}/M^{+}} \simeq \mc O(-1)^{2}$. 
Let $\chi_{2} \colon M^{+}\dra M^{++}$ be the flop of $e_{2}^{+}$ and set $P^{++}:=\chi_{2\ast}(P^{+})$. 
Then $P^{++}$ is isomorphic to $\P^{2}$. 
By the construction of these flops, we have 
\begin{align}\label{E2Pemp}
E_{2}^{++} \cap P^{++}=\emp.
\end{align}
If $e_{3}^{++} \subset M^{++}$ denotes the proper transform, then $e_{3}^{++} \subset P^{++}$ is a line of $P^{++} \simeq \P^{2}$ and $-K_{M^{++}}.e_{3}^{++}=-K_{M}.e_{3}=2$. 
Thus $-K_{M}|_{P^{++}} \simeq \mc O_{\P^{2}}(2)$ and hence $\mc N_{P^{++}/M^{++}} \simeq \mc O(-1)$. 
Hence there exists $\mu \colon M^{++} \to W$ such that $\mu(P^{++})=p$ is a smooth point and $\mu$ is the blow-up at $p$: 
\[\xymatrix{
&\ar[ld]_{\mu}M^{++}&\ar@{-->}[l]_{\chi_{2}}M^{+}&M\ar[rd]^{\s} \ar@{-->}[l]_{\chi_{1}}& \\
W&&&&\P^{3}.
}\]
Let $\chi=\chi_{2} \circ \chi_{1}$ be the composition. 

\noindent
\textbf{Step 3. } We consider a general effective member $D \in |H-(E_{2}+2E_{3})|$ in $M$ and let $D^{++} \subset M^{++}$ be the proper transform. 
\begin{claim}\label{claim3-P1,11}
The following assertions follow.
\begin{enumerate}
\item $\dim |D|=1$ and $\Bs|D|=l_{23}$. 
\item $D^{++} \simeq \F_{2}$ and $D^{++} \cap E_{2}^{++}$ is the $(-2)$-curve on $D^{++}$. 
\item $|D^{++}|$ is base point free and $D^{++} \cap P^{++}=\emp$. 
\end{enumerate}
\end{claim}
\begin{proof}
(1) Let $l_{23}' \subset \P^{3}$ be the linear span of the tangent vector of $\Sigma$ at $p_{2}$. 
Then $D$ is the proper transform of a $2$-plane in $\P^{3}$ containing $l_{23}'$. 
Since $l_{23}$ is the proper transform of $l_{23}'$ on $M$, the assertion (1) holds. 

(2) Note that $D$ is the blow-up of $\P^{2}$ at an infinitely near point of order $2$ supported on $p_{2}$. 
Therefore, $E_{2} \cap D$ is the $(-2)$-curve on $D$. 
Since $\chi|_{E_{2}}$ and $\chi|_{D}$ are isomorphic along $E_{2} \cap D$, 
$E_{2}^{++} \cap D^{++}$ is a $(-2)$-curve on $D$ also. 
Since $\chi|_{D}$ is a morphism and blows down the $(-1)$-curve $l_{23}$ on $D$, we have $D^{++} \simeq \F_{2}$, which proves (2). 

(3) Recall that $P$ is the blow-up of $\P^{2}$ at a union of reduced point $p_{1}$ and an infinitely near point of order $2$ supported at $p_{2}$. 
Note that $D$ meets $P$ transversally along $l_{23}$. 
Since $M \dra M^{+}$ is the Atiyah flop of $l_{12}$ and $l_{23}$, it holds that $D^{+} \cap P^{+}=\emp$ and $\Bs |D^{+}|=\emp$. 
Since $D^{+} \cap e_{2}=\emp$, we obtain $D^{++} \cap P^{++}=\emp$ and $\Bs|D^{++}|=\emp$. 
Hence we are done. 
\end{proof}

\noindent
\textbf{Step 4. } 
In this step, we show that $W \simeq \F_{2} \times \P^{1}$. 
By Claim~\ref{claim-MwF}, $W$ is also weak Fano and 
\begin{align}\label{eq-index2-P1,11}
-K_{W} \sim 2\mu_{\ast}H.
\end{align}
Set $D_{W}:=\t_{\ast}D^{++}$. 
By Claim~\ref{claim3-P1,11}, the complete linear system $|D_{W}|$ gives a morphism $\a \colon W \to \P^{1}$ whose fibers are connected. 
By Claim~\ref{claim3-P1,11}~(2), a general fiber of $\a$ is isomorphic to $\F_{2}$ and hence  there exists an extremal ray $R \subset \ol{\NE}(W/\P^{1})$ such that $-K_{W}.R<0$. 
Let $\b \colon W \to S$ be the contraction of $R$ and $\pi \colon S \to \P^{1}$ the induced morphism. 
Since a general $\a$-fiber $D_{W}$ is isomorphic to $\F_{2}$, 
$\b|_{D_{W}} \colon D_{W} \to \b(D_{W})$ is the $\P^{1}$-bundle structure of $\F_{2}$, which implies $\dim S=2$. 
Then $S$ is non-singular and a general fiber of $S \to \P^{1}$ is a smooth rational curve. 
Since $\rho(S)=2$, $S \to \P^{1}$ is a $\P^{1}$-bundle. 
If $\D$ denotes the discriminant of the standard conic bundle $\b \colon W \to S$, then $\pi|_{\D} \colon \D \to \P^{1}$ is not dominant. 
Hence $\D$ is supported on a disjoint union of smooth rational curves if $\D$ is not empty. 
As the composition of flops $M \dra M^{+} \dra M^{++}$ preserves their Euler numbers with respect to analytic topologies, we have $\Eu(W)=8$, which implies $\D=\emp$, that is, $\b \colon W \to S$ is a $\P^{1}$-bundle: 
\[\xymatrix{
W \ar[r]^{\b} \ar@(u,u)[rr]_{\a}& S \ar[r]^{\pi} & \P^{1}.
}\]
In particular, $\a$ is smooth as a morphism and has $\F_{2}$ as a general fiber. 
Thus every $\a$-fiber is $\F_{2}$ since $W$ is weak Fano. 

\begin{claim}\label{claim5-P1,11}
Set $E_{i,W}=\t_{\ast}(E_{i}^{++})$ for $i \in \{1,2,3\}$. 
Then $E_{2,W}$ is a section of $\b$ and isomorphic to $(\P^{1})^{2}$. 
In particular, we have $S \simeq (\P^{1})^{2}$. 
Moreover, it holds that $-K_{W}|_{E_{2,W}}=(\a|_{E_{2,W}})^{\ast}\mc O_{\P^{1}}(2)$. 
\end{claim}
\begin{proof}
Recall that $E_{2} \simeq \F_{1}$ as varieties and $\tau|_{E_{2}} \colon E_{2} \to C \simeq \P^{1}$ gives the ruling. 
Then $E_{2} \dra E_{2}^{++}$ is the elementary transformation of the $\P^{1}$-bundle $\tau|_{E_{2}} \colon E_{2} \to C$ with center $l_{12} \cap E_{2}$. 
Since the $(-1)$-curve on $E_{2}$ does not pass through $l_{12} \cap E_{2}$, we have $E_{W} \simeq E_{2}^{++} \simeq (\P^{1})^{2}$ by (\ref{E2Pemp}). 
By Claim~\ref{claim3-P1,11}, for a $\a$-fiber $D_{W} \simeq \F_{2}$, the intersection $D_{W} \cap E_{2,W}$ is the $(-2)$-curve on $D_{W}$. 
Thus $\b|_{E_{2,W}} \colon E_{2,W} \to S$ is birational and hence is isomorphic. Hence $E_{2,W}$ is a section of $\b$. 

Let $\g$ be a fiber of $\a|_{E_{2,W}}$ and $D_{W}$ the fiber of $\a$ containing $\g$. 
Then $\g$ is the $(-2)$-curve on $D_{W}$ and hence $-K_{W}.\g=0$. 
Hence $-K_{W}|_{E_{2,W}}$ is the pull-back of a line bundle on $\P^{1}$ under the morphism $\a|_{E_{2,W}}$. 
To see the last statement, set $e_{23}:=E_{2} \cap E_{3}$ and let $e_{23}^{++}$ be the proper transform on $M^{++}$. 
Since $E_{2}^{++}$ contains $e_{23}^{++}$, we can define the proper transform $e_{23,W}$ on $W$, which is the different ruling of $E_{2,W} \simeq (\P^{1})^{2}$ from $\g$. 
Then we have $-K_{W}.e_{23,W}=-K_{M^{++}}.e_{23}^{++}=-K_{M}.e_{23}=-K_{M}E_{2}E_{3}=2$ by a direct calculation. Hence we are done. 
\end{proof}

Let $\mc E$ be a rank $2$ bundle on $S \simeq (\P^{1})^{2}$ such that $\P(\mc E) \simeq W$. 
By (\ref{eq-index2-P1,11}) we may assume that $\det \mc E=-K_{S}$. 
Then $\mc O(-K_{W})=\mc O_{\P(\mc E)}(2)$. 
By Claim~\ref{claim5-P1,11}, $E_{2,W}$ is the section corresponding to a surjection $\mc E \to \pi^{\ast}\mc O(1)$ on $S$. 
Therefore, we have $\mc E \simeq \pi^{\ast}\mc O(1) \oplus \mc O(-K_{S}) \otimes \pi^{\ast}\mc O(-1)$, which implies $W \simeq \F_{2} \times \P^{1}$. 

\noindent
\textbf{Step 5. } 
Set $\ol{W}=\Q^{2}_{0} \times \P^{1}$ and $\psi_{W} \colon W \to \ol{W}$ be the crepant contraction. 
Since $P^{++} \cap E_{1}^{++}=\emp$ by Claim~\ref{claim3-P1,11}, the center of the blow-up $M^{++} \to W$, say $w \in W$, is not contained in $E_{1,W}$. 
Set $\ol{w}:=\psi_{W}(w)$ and let $V^{+}:=\Bl_{\ol{w}}\ol{W}$. 
Note that $\ol{w}$ is a smooth point of $\ol{W}$. 
Then we have a commutative diagram 
\[\xymatrix{
&\ar[ld]_{\mu}M^{++}\ar[d]&\ar@{-->}[l]_{\chi_{2}}M^{+}&\ar@{-->}[l]_{\chi_{1}}M\ar[rdd]^{\sigma} \ar[d]_{\tau}& \\
W\ar[d]_{\psi_{W}}&\ar[ld]^{\mu_{V}}V^{+}&&\ar@{-->}[ll]_{\chi_{V}}V=\Bl_{\Sigma}\P^{3}\ar[rd]_{\sigma_{V}}& \\
\ol{W}&&&&\P^{3}.
}\]
The rational map $\chi_{V}$ is isomorphic in codimension $1$ and not isomorphism. 
Since the anti-canonical models of $M,M^{+},M^{++}$ are the same, those of $V,V^{+}$ are also the same. Hence $\chi_{V}$ is a flop, which completes the proof. 
\qed

\subsubsection{Proof of Proposition~\ref{prop-hypsecP13}~(2)}

We use the same notation as in Proposition~\ref{prop-singfibP13}~(2) and its proof. 
Set $\mc E=\mc O_{\F_{2}}(h)^{2}$. 
Then $W=\P_{\F_{2}}(\mc E)$ and $\psi_{W}$ is the blow-up of $\ol{W}=\Q^{2}_{0} \times \P^{1}$ along its singular locus. 
Let $p \colon W \to \F_{2}$ be the projection, $\xi$ the tautological divisor of $\P_{\F_{2}}(\mc E)$, $L=\pi^{\ast}h$ and $F=\pi^{\ast}f$. 
Then the unique member $E \in |L-2F|$ is $\Exc(\psi_{W})$. 
From the construction, we have the following linear equivalences of divisors by a direct calculation. 
\begin{align}\label{rel-P11,1}
\left\{ \begin{array}{ll}
\chi^{-1}_{\ast}\mu^{\ast}F\sim D \sim (H-(E_{2}+2E_{3})) \\
\chi^{-1}_{\ast}\mu^{\ast}(L-2F) \sim E_{2} \\
\chi^{-1}_{\ast}P^{++}\sim P  \sim H-(E_{1}+E_{2}+2E_{3}) \text{ and } \\
\chi^{-1}_{\ast}\mu^{\ast}\xi 
\sim \frac{1}{2}\left( -K_{M}+2P\right) \sim 3H-(2E_{1}+2E_{2}+4E_{3}). 
\end{array} \right.
\end{align}

Let $\wt{X} \subset W$ be the proper transformation of $X \subset \ol{W}$. 
Noting that $\mc O_{\P(\mc E)}(\xi)=f^{\ast}\mc O_{\ol{W}}(1)$, we can take an integer $a \geq 0$ such that $\wt{X} \in |\xi-aE|=|\xi-a(L-2F)|$. 
It holds that $a \in \{0,1\}$ since $|\xi-aE| \neq \emp$. 

\textit{Case 1; $a=1$. } 

In this case, $\wt{X} \in |\xi-E|=|\xi-(L-2F)|$. Since $X$ contains $\Sing(\ol{W})$, $X$ is non-normal. 
Let $s \colon \mc O(h-2f) \to \mc E$ be the section which corresponds to $\wt{X}$. 
Since $\wt{X}$ is irreducible, Lemma~\ref{lem-irr} gives the exact sequence 
$0 \to \mc O(h-2f) \to \mc O(h)^{2} \to \mc I_{Z}(h+2f)  \to 0$, 
where $Z \subset \F_{2}$ is $0$-dimensional or empty. 
By computing the Chern classes, we obtain $\lgth(\mc O_{Z})=0$, i.e., $Z=\emp$. 
Therefore, $\wt{X}$ is a section of $\pi \colon \P_{\F_{2}}(\mc E) \to \F_{2}$. 
The morphism $\wt{X} \simeq \F_{2} \to X$ is nothing but the normalization of $X$ and hence $X$ is of type (n2). 
Moreover, we have 
\begin{align*}
\chi^{-1}_{\ast}\wt{X} &\sim \chi^{-1}_{\ast}\mu^{\ast}(\xi-(L-2F))-P 
\sim 2H-(E_{1}+2E_{2}+2E_{3}). 
\end{align*}
by (\ref{rel-P11,1}). 
Hence $Q$ is a quadric cone in $\P^{3}$ such that $Q$ contains $p_{1}$ as a smooth point and $p_{2}$ as the vertex. 

\textit{Case 2; $a=0$. }

In this case, $\wt{X}$ is a prime member of $|\mc O_{\P(\mc E)}(1)|$. 
Lemma~\ref{lem-irr} gives the following exact sequence on $\F_{2}$: 
\begin{align}\label{ex-sing(n,2)}
0 \to \mc O \to \mc O(h)^{2} \to \mc I_{Z}(2h) \to 0. 
\end{align}
Now $\Bl_{Z}\F_{2}=\wt{X}$ and $Z$ is a $0$-dimensional subscheme of length $c_{2}(\mc E)=2$. 
Since $-K_{\wt{X}}$ is nef and big, we have $Z \cap C_{0}=\emp$.
Tensoring the sequence (\ref{ex-sing(n,2)}) with $\mc O(-2h+f)$, 
we have $h^{0}(\mc I_{Z}(f))=0$, which means that every fiber of $\F_{2} \to \P^{1}$ does not contain $Z$. 

If $Z$ is a union of two points $\{q_{1},q_{2}\}$, 
then by taking an elementary transformation of $\F_{2}$ with center $q_{1}$, 
we can see that $\wt{X}$ is the blow-up of $\P^{2}$ at a non-colinear union of a reduced point and an infinitely near point of order $2$. 
Hence $X$ is of type (2,2). 

On the other hand, if $Z$ is non-reduced, 
then by taking an elementary transformation of $\F_{2}$ with center $q_{1}:=Z_{\red}$, 
we can see that the minimal resolution of $\wt{X}$ is the blow-up of $\P^{2}$ at a colinear union of a reduced point and an infinitely near point of order $2$. 
Hence $X$ is of type (1,2). 

For each cases, we obtain 
\begin{align*}
\chi^{-1}_{\ast}\wt{X} &\sim \chi^{-1}_{\ast}\mu^{\ast}\xi-P
\sim 2H-(E_{1}+E_{2}+2E_{3}).
\end{align*}
by (\ref{rel-P11,1}). 
Hence $Q$ is an irreducible and reduced quadric surface in $\P^{3}$ such that $Q$ contains $p_{1},p_{2}$ as smooth points. 
We complete the proof. 
\qed

\subsubsection{Proof of Proposition~\ref{prop-singfibP13}~(3)}
We proceed in the same way as in the proof of Proposition~\ref{prop-singfibP13}~(2). 

\noindent
\textbf{Step 1.} 
We set our basic preliminaries and fix notation as follows. 

Let $\Sigma_{\red}=\{p_{1}\}$ and $f_{1} \colon \Bl_{p_{1}}\P^{3} \to \P^{3}$ the blow-up at $p_{1}$. 
Let $\Sigma'$ be the $0$-dimensional length $2$ subscheme $\Sigma'$ corresponding to $\Sigma$. 
We set $\Sigma'_{\red}=\{p_{2}\}$. Then $p_{2}$ is a point of $\Exc(f_{1})$. 
Let $f_{2} \colon \Bl_{p_{2}}\Bl_{p_{1}}\P^{3} \to \Bl_{p_{1}}\P^{3}$ be the blow-up at $p_{2}$. 
Then we have a point $p_{3} \in \Exc(f_{2})$ corresponding to the tangent vector of $\Sigma'$ at $p_{2}$. 
Since $\Sigma$ is not colinear, $\Sigma'$ is not contained in $\Exc(f_{1})$ as schemes and hence it holds that $p_{3} \not\in (f_{2})^{-1}_{\ast}\Exc(f_{1})$. 
Let $f_{3} \colon M:=\Bl_{p_{3}}\Bl_{p_{2}}\Bl_{p_{1}}\P^{3} \to \Bl_{p_{2}}\Bl_{p_{1}}\P^{3}$ denote the blow-up at $p_{3}$ and $\s=f_{1} \circ f_{2} \circ f_{3}$. 
Let $H=\s^{\ast}\mc O(1)$ and $E_{i} \subset M$ the proper transformation of $\Exc(f_{i})$. 
Note that $E_{1}$ and $E_{2}$ are isomorphic to $\F_{1}$. 
Let $P \subset M$ be the proper transform of the linear hull of $\Sigma$. 
Since $K_{M}=\s^{\ast}K_{\P^{3}}+2E_{1}+4E_{2}+6E_{3}$, we have 
\begin{align}
\label{eq-P111-PK}
P \in |H-(E_{1}+2E_{2}+3E_{3})| \text{ and } -K_{M} \sim 2H+2P. 
\end{align}

Since $\s^{-1}\mc I_{\Sigma} \cdot \mc O_{M}=\mc O(-E_{1}-2E_{2}-3E_{3})$, we have a natural morphism 
\[\tau_{M} \colon M \to V:=\Bl_{\Sigma}\P^{3}\]
by the universal property. 
Then $\tau_{M}$ contracts $E_{1}$ and $E_{2}$ to one smooth rational curve $C$. 
For any $i \in \{1,2\}$, $\tau_{M}|_{E_{i}} \colon E_{i} \to C$ is a $\P^{1}$-bundle and $E_{1} \cap E_{2}$ is a section of $\tau_{M}|_{E_{i}}$. 
Moreover, if $f_{i}$ denotes a fiber of $\tau_{M}|_{E_{i}}$, then we have $\mc N_{f_{i}/M} \simeq \mc O_{\P^{1}} \oplus \mc O_{\P^{1}}(-2)$. 
Hence $V$ has a family of Du Val $A_{2}$-singularities along $C$. 
Finally, let $l$ be the proper transform of the linear hull of the tangent vector of $\Sigma$ and 
$e_{i}=E_{i} \cap P$ for $i \in \{1,2,3\}$. 
Then we have 
\begin{align}\label{eq-P111-le}
l \sim H|_{P}-(e_{1}+2e_{2}+e_{3}) \text{ on } P.
\end{align}

\noindent
\textbf{Step 2.} 
Let us show that $V$ is weak Fano and $-K_{V}$ is divisible by $2$, which is equivalent to the following claim. 
\begin{claim}\label{claim-MwF-P111}
$M$ is weak Fano and $-K_{M}$ is divisible by $2$ in $\Pic(M)$. 
\end{claim}
\begin{proof}
Since $-K_{M}=2(H+P)$ by (\ref{eq-P111-PK}) and $(-K_{M})^{3}=40$, it suffices to show that $(H+P)|_{P}$ is nef. 
By (\ref{eq-P111-PK}) and (\ref{eq-P111-le}), we have 
$(H+P)|_{P} \sim 2l+e_{1}+2e_{2}+e_{3}$. 
It is easy to see that the above divisor is nef. 
\end{proof}

Now let us confirm that $\psi_{V}$ is small. 
Let $\g \subset V$ be a curve such that $-K_{V}.\g=0$. 
Letting $H_{V}=\t_{\ast}H$ and $P_{V}:=\t_{\ast}P$, we have $-K_{V}=2(H_{V}+P_{V})$. 
Since $-K_{V}$ is ample over $\P^{3}$, we have $\g \subset P_{V}$. 
Let $\wt{\g} \subset P$ be the proper transform. 
Then $-K_{M}.\wt{\g}=0$ and it follows that $\wt{\g}=l$, $e_{1}$ or $e_{2}$ from the proof of Claim~\ref{claim-MwF-P111}. 
Since the morphism $\tau|_{P} \colon P \to P_{V}$ is the contraction of $e_{1}$ and $e_{2}$, we have $\g=\tau(l)$. 
Hence $\psi_{V}$ is small. 

Let us construct a sequence of flops and a divisorial contraction from $M$ as follows. 
Now $P$ is the blow-up of $\P^{2}$ at an infinitely near point of order $3$ corresponding to $\Sigma$. 
Since $l$ is a $(-1)$-curve in $P$, we have $\mc N_{l/M}=\mc O_{\P^{1}}(-1)^{2}$ and there exists the Atiyah flop $\chi_{1} \colon M \dra M^{+}$ of $l$. 
Let $l^{+}$ be the flopped curve and $e_{i}^{+},P^{+}$ be the proper transforms of $e_{i},P$ on $M^{+}$ respectively. 
Since $e_{2}^{+}$ is a $(-1)$-curve in $P^{+}$, we have $\mc N_{e_{2}^{+}/M^{+}}=\mc O_{\P^{1}}(-1)^{2}$ and there exists the Atiyah flop $\chi_{2} \colon M^{+} \dra M^{++}$ of $e_{2}^{+}$. 
Let $e_{2}^{++}$ be the flopped curve and $l^{++},e_{1}^{++},e_{3}^{++},P^{++}$ be the proper transforms of $l^{+},e_{1}^{+},e_{3}^{+},P^{+}$ on $M^{++}$ respectively. 
Since $e_{1}^{++}$ is a $(-1)$-curve in $P^{++}$,we have the Atiyah flop $\chi_{3} \colon M^{++} \dra M^{+++}$ of $e_{1}^{++}$. 
If $P^{+++} \subset M^{+++}$ denotes the proper transform of $P$, then we have $P^{+++} \simeq \P^{2}$. 
Let $e_{3}^{+++}$ be the proper transform of $e_{3}$. 
Then we have $-K_{M^{+++}}.e_{3}^{+++}=-K_{M}.e_{3}=2$. 
Therefore, $\mc N_{P^{+++}/M^{+++}}=\mc O_{\P^{2}}(-1)$ and hence there exists $\mu \colon M^{+++} \to W$ such that $\mu(P^{+++})=w$ is a smooth point of $W$ and $\mu$ is the blow-up of $W$ at $w$. 

\[\xymatrix{
&\ar[ld]_{\mu}M^{+++} & \ar@{-->}[l]_{\chi_{3}} M^{++}& \ar@{-->}[l]_{\chi_{2}} M^{+}& \ar@{-->}[l]_{\chi_{1}}M\ar[rd]^{\s}& \\
W&&&&&\P^{3}.
}\]

Set $\chi=\chi_{3} \circ \chi_{2} \circ \chi_{1}$. 

\noindent
\textbf{Step 3. }
Take a general member $D \in |H-(E_{1}+2E_{2}+2E_{3})|$ on $M$. 
Let $\ol{l} \subset \P^{3}$ be the linear span of the tangent vector of $\Sigma$ at $p_{1}$. 
In this setting, it holds that $\s^{-1}_{\ast}\ol{l}=l$. 
Then $D$ is the proper transform of a $2$-plane in $\P^{3}$ containing $\ol{l}$. 
Let $D^{+},D^{++},D^{+++}$ be the proper transforms of $D$ on $M^{+},M^{++},M^{+++}$ respectively. 
\begin{claim}\label{claim1-P111}
\begin{enumerate}
\item $\Bs |D|=l$ and $\dim |D|=1$. 
\item $|D^{+++}|$ is base point free and $D^{+++} \cap P^{+++}=\emp$. 
\item $D^{+++} \simeq \F_{2}$ and $E_{1}^{+++} \cap D^{+++}$ is the $(-2)$-curve of $D^{+++}$. 
\item $E_{1}^{+++} \cap P^{+++}=\emp$ and $E_{2}^{+++} \cap P^{+++}=\emp$. 
\end{enumerate}
\end{claim}
\begin{proof}
(1) This assertion follows directly from that $|D|$ is the proper transform of the pencil of hyperplanes in $\P^{3}$ containing $\ol{l}$. 

(2) Since $P$ is the blow-up of $\P^{2}$ at the infinitely near $3$ points with respect to $\Sigma$, $P$ meets $D$ transversally along $l$. 
By the construction of the flop $\chi_{1} \colon M \dra M^{+}$, we have that $D^{+} \cap P^{+}=\emp$ and $\Bs |D^{+}|=\emp$ by (1). 
Since we take the member $D$ generically, $D^{+}$ does not meet $e_{2}^{+}$ and $e_{3}^{+}$. 
Thus we obtain $D^{+++} \cap P^{+++}=\emp$ and $\Bs |D^{+++}|=\emp$, which proves (2). 

(3) Note that $D$ is the blow-up of $\P^{2}$ at an infinitely near point of order $2$. 
Then $D \dra D^{+}$ is the contraction of $l$, which is a $(-1)$-curve in $D$. 
Thus we have $D^{+} \simeq \F_{2}$. 
Moreover, we can see that $D^{+} \dra D^{++} \dra D^{+++}$ are isomorphic by the construction of the flops and hence $D^{+++} \simeq \F_{2}$. 
Note that $E_{1} \cap D$ is $(-2)$-curve on $D$ and $\chi$ is isomorphic along $E_{1} \cap D$, which show the second assertion. 

(4) It follows from that $E_{i}$ meets $P$ along $e_{i}$ transversally for $i=1,2$. 
\end{proof}

\noindent
\textbf{Step 4. }
In this step, we will show that $W \simeq \wt{\P}^{1,1,1}$. 
By Claim~\ref{claim-MwF-P111}, $W$ is a weak Fano $3$-fold of Picard rank $3$ and $-K_{W}$ is divisible by $2$. 
Let $D_{W}=\mu_{\ast}D^{+++}$. 
By Claim~\ref{claim1-P111}~(1) and (2), 
the complete linear system $|D_{W}|$ gives the morphism $\a \colon W \to \P^{1}$. 
Using Claims~\ref{claim-MwF-P111} and \ref{claim1-P111}, we can see that $\a$ decomposes into two $\P^{1}$-bundles $\b \colon W \to S$ and $\pi \colon S \to \P^{1}$ by the same argument in Step 4 of the proof of Proposition~\ref{prop-singfibP13}~(2): 

\begin{claim}\label{claim-singP13-2-2}
Let $E_{1,W}$ be the proper transform of $E_{1}$ on $W$. 
Then it holds that $E_{1,W} \simeq \F_{2}$ and $-K_{W}|_{E_{1,W}} \sim (\a|_{E_{1,W}})^{\ast}\mc O_{\P^{1}}(2)$. 
Moreover, $E_{1,W}$ is a $\b$-section and hence $S \simeq \F_{2}$. 
\end{claim}
\begin{proof}
We recall that $E_{1} \simeq \F_{1}$ and $-K_{M}|_{E_{1}}=2f$, where $f$ is a fiber of the $\P^{1}$-bundle $\F_{1} \to \P^{1}$. 
Claim~\ref{claim1-P111}~(3) and (4) gives that $\F_{2} \simeq E_{1}^{+++} \simeq E_{1,W}$ and $-K_{W}|_{E_{1,W}}=2f$. 
Moreover, for an $\a$-fiber $D_{W}$, the intersection $D_{W} \cap E_{1,W}$ is the $(-2)$-curve of the fiber $D_{W} \simeq \F_{2}$. 
Hence $\b|_{E_{1,W}} \colon E_{1,W} \to S$ is birational and hence isomorphic since $S$ is a Hirzebruch surface. 
\end{proof}
Let $\mc E$ be a rank $2$ bundle $\mc E$ on $S \simeq \F_{2}$ such that $\P_{S}(\mc E)=W$. 
Since $-K_{W}$ is divisible by $2$, we may assume that $\det \mc E =\mc O(-K_{\F_{2}})$. 
In this setting, $\mc E$ is nef since $W$ is weak Fano. 
As we have seen in Claim~\ref{claim-singP13-2-2}, $\b \colon W=\P_{S}(\mc E) \to S \simeq \F_{2}$ has a section $E_{1,W}$ satisfying $\mc O_{\P(\mc E)}|_{E_{1,W}}=\mc O_{\F_{2}}(f)$. 
Hence the nef bundle $\mc E$ fits into the exact sequence (\ref{ex-111}). 
Then we have $W=\wt{\P}^{1,1,1}$ by the definition of $\wt{\P}^{1,1,1}$. 

\noindent
\textbf{Step 5.}
Let $\psi_{W} \colon W \to \ol{W}$ be the morphism to the anti-canonical model $\ol{W}=\P^{1,1,1}$ of $W=\wt{\P}^{1,1,1}$. 
For $i \in \{1,2,3\}$, if $E_{i,W}$ denotes the proper transform of $E_{i}$, then 
$\Exc(\psi_{W})=E_{1,W} \cup E_{2,W}$. 
Moreover, $\psi(E_{1,W})=\psi(E_{2,W})$ and $\psi(\Exc(\psi_{W}))$ is one smooth rational curve $\ol{C} \subset \ol{W}$. 
Set $w=\mu(\Exc(\mu))$. Then Claim~\ref{claim1-P111}~(3) implies that $\ol{w}:=\psi_{W}(w)$ is also a smooth point of $\ol{W}$. 
Setting $V^{+}:=\Bl_{\ol{w}}\ol{W}$, we have a birational map $\chi_{V} \colon V \dra V^{+}$ by composing birational maps in the following commutative diagram:
\[\xymatrix{
&\ar[ld]_{\mu}M^{+++} \ar[d]& \ar@{-->}[l]_{\chi_{3}} M^{++}& \ar@{-->}[l]_{\chi_{2}} M^{+}& \ar@{-->}[l]_{\chi_{1}}M\ar[d]_{\tau_{M}}\ar[rdd]^{\s}& \\
W\ar[d]_{\psi_{W}}&V^{+}\ar[ld]^{\mu_{V}}&&&\ar@{-->}[lll]_{\chi_{V}}V=\Bl_{\Sigma}\P^{3}\ar[rd]_{\s_{V}}& \\
\ol{W}&&&&&\P^{3}.
}\]
Then $\chi_{V}$ is an isomorphism in codimension $1$ and not isomorphic since they have different extremal contractions. 
Since $M^{+++} \to V^{+}$ and $M \to V$ are crepant contractions, the anti-canonical model of $V$ coincides with that of $V^{+}$. 
Hence $\chi_{V}$ is the flop. 
We complete a proof of Proposition~\ref{prop-singfibP13}~(3). 
\qed

\subsubsection{Proof of Proposition~\ref{prop-hypsecP13}~(3)}
Let $\b \colon \P_{\F_{2}}(\mc E)=\wt{\P}^{1,1,1}=W \to \F_{2}=S$ be the projection, $L,F$ the pull-back of $h,f$ respectively and $\xi$ the tautological divisor of $\b$. 
Then $E_{1} \in |\xi-(2L-F)|$, $E_{2} \in |L-2F|$ and $\Exc(\psi_{W})=E_{1,W} \cup E_{2,W}$ hold. 
From the construction, we have the following linear equivalence by a direct calculation: 
\begin{align}\label{rel-P111}
\left\{ \begin{array}{l}
\chi^{-1}_{\ast}\t^{\ast}F=H-(E_{1}+2E_{2}+2E_{3}),\\
\chi^{-1}_{\ast}\t^{\ast}(\xi-(2L-F))=E_{1},\quad \chi^{-1}_{\ast}\t^{\ast}(L-2F)=E_{2} ,  \\
\chi^{-1}_{\ast}\Exc(\t)=P=H-(E_{1}+2E_{2}+3E_{3}) \text{ and } \\
\chi^{-1}_{\ast}\t^{\ast}\xi=H+2P=3H-(2E_{1}+4E_{2}+6E_{3}). 
\end{array} \right.
\end{align}

If we set $\wt{X}:=\psi_{W\ast}^{-1}X$, then there exists $a,b \in \Z_{\geq 0}$ such that 
\[\wt{X} \sim \xi-(aE_{1}+bE_{2}) = (1-a)\xi+(2a-b)L+(2b-a)F.\]
Since $\wt{X}$ is a prime divisor, we have $a \leq 1$. 
If $a=1$, then $X$ is non-normal. 
Moreover, $\wt{X} \sim (2-b)L+(2b-1)F$ implies $b=1$ since $\wt{X}$ is a prime divisor. 
If $a=0$, then $\wt{X} \in \xi-b(L-2F)$ and hence $b=0$ or $1$. 
Thus we have $(a,b) \in \{(1,1),(0,1),(0,0)\}$. 
\begin{claim}
It does not occur that $(a,b)=(0,1)$. 
\end{claim}
\begin{proof}
In this case, $\wt{X} \sim \xi-(L-2F)$. 
Let $s \colon \mc O(h-2f) \to \mc E$ be the section corresponding $\wt{X}$. 
Then Lemma~\ref{lem-irr} gives an exact sequence 
\[0 \to \mc O(h-2f) \to \mc E \to \mc I_{Z}(h+2f) \to 0.\]
Since $c_{2}(\mc E)=2$, $Z$ must be empty. 
Then we have $\mc E \simeq \mc O(h)^{2}$ since $\mc E$ is nef, which is a contradiction to (\ref{ex-111}). 
\end{proof}

\textit{Case 1; $a=b=1$.} 

Since $\wt{X} \in |L+F|$, $C:=\b(\wt{X})$ is an irreducible member of $|h+f|$ and hence is smooth. 
The sequence (\ref{ex-111}) gives $\mc E|_{C}=\mc O(1) \oplus \mc O(5)$. 
Therefore, we have $\wt{X} \simeq \P(\mc E|_{C})=\F_{4}$. 
$\wt{X} \to X$ is nothing but the normalization of $X$. Hence $X$ is of type (n4). 
By (\ref{rel-P111}), we obtain 
\begin{align*}
\chi^{-1}_{\ast}\wt{X} &\sim \chi^{-1}_{\ast}(\t^{\ast}(L+F)-P^{+++}) \sim 2H-(2E_{1}+3E_{2}+3E_{3}). 
\end{align*}
Hence $Q$ is singular and contains $p_{1}$ as its vertex. 

\textit{Case 2; $a=b=0$.} 

Let $s \colon \mc O \to \mc E$ be the corresponding section to $\wt{X} \in |\xi|$. 
Applying Lemma~\ref{lem-irr}, we obtain an exact sequence 
\begin{align}\label{ex2-sing(n,1)}
0 \to \mc O \mathop{\to}^{s} \mc E \to \mc I_{Z}(2h) \to 0.
\end{align}
Since $c_{2}(\mc E)=2$, we obtain $\lgth (\mc O_{Z})=2$. 
Noting that $-K_{\wt{X}}$ is nef and big, we have $Z \cap C_{0}=\emp$. 
By tensoring the sequence (\ref{ex2-sing(n,1)}) with $\mc O(-2h+f)$, 
we have $h^{0}(\mc I_{Z}(f))=1$ from the exact sequence (\ref{ex-111}). 
Therefore, $Z$ is contained in a fiber $l$ of $p \colon \F_{2} \to \P^{1}$. 

When $Z$ consists of two points $q_{1},q_{2}$, 
taking the elementary transformation of $\F_{2}$ with center $q_{1}$, 
we can see that $\wt{X}$ is the blow-up of $\P^{2}$ at non-colinear curvilinear infinitely near $3$ points.
Hence $X$ is type of (2,1). 

If $Z$ is not reduced, then by taking an elementary transformation of $\F_{2}$ with center $q_{1}:=Z_{\red}$, the minimal resolution of $\wt{X}$ is the blow-up of $\P^{2}$ at colinear curvilinear infinitely near $3$ points. 
Hence $X$ is type of (1,1). 

For each case, we obtain 
\begin{align*}
\chi^{-1}_{\ast}\wt{X} &\sim \chi^{-1}_{\ast}(\t^{\ast}\xi-P^{+++}) \sim 2H-(E_{1}+2E_{2}+3E_{3}). 
\end{align*}
by (\ref{rel-P111}). 
Hence $Q$ contains $p_{1}$ as a smooth point. 

We complete the proof of Propositions~\ref{prop-singfibP13} and \ref{prop-hypsecP13}. 
\qed

\section{Proof of Theorem~\ref{mainthm-ex}}

Let us show Theorem~\ref{mainthm-ex} by presenting explicit examples of sextic del Pezzo fibrations which contains singular fibers of type $(2,j)$ for $j=1,2,3$ (see Example~\ref{ex-(2,j)}), $(1,j)$ for $j=1,2,3$ (see Example~\ref{ex-(1,j)}), (n2) and (n4) (see Example~\ref{ex-nonnormal}). 

\noindent\textbf{Step 1.}
We fix the following submanifolds in $\P^{4}$:
\[\ol{T} \subset \Q^{2} \subset \Q^{3} \subset \P^{4},\]
where $\Q^{3}$ is a smooth quadric $3$-fold, $\Q^{2} \subset \Q^{3}$ is a smooth hyperplane section and $\ol{T} \subset \Q^{2}$ is a twisted cubic curve. 

For a smooth conic $C \subset \Q^{3}$ with $\ol{T} \cap C=\emp$, 
let $\t \colon Q:=\Bl_{C}\Q^{3} \to \Q^{3}$ be the blow-up along $C$ and $T:=\t^{-1}_{\ast}\ol{T}$. 
Then there exists a quadric fibration $q \colon Q \to \P^{1}$ given by the linear system of hyperplane sections of $\Q^{3}$ containing $C$. 
Note that $q|_{T} \colon T \to \P^{1}$ is a triple covering. 

\noindent\textbf{Step 2.}
\begin{claim}\label{claim-22}
There exists a sextic del Pezzo fibration $\vp \colon X \to \P^{1}$ with a section $s$ such that $(q \colon Q \to \P^{1},T)$ is the relative double projection of $(\vp \colon X \to \P^{1},s)$. 
Moreover, $X$ is a weak Fano 3-fold with $(-K_{X})^{3}=22$ and $-K_{X}.s=0$. 
\end{claim}
\begin{proof}
Let $\s \colon \wt{Q}:=\Bl_{T}Q \to Q$ be the blow-up. 
Set $H=\s^{\ast}\t^{\ast}\mc O_{\Q^{3}}(1)$, $E=\Exc(\s)$ and $G=\s^{-1}_{\ast}\Exc(\t)$. 
Note that $-K_{\wt{Q}}=3H-E-G$ is free since $H-G$ and $2H-E$ are free. 
Using \cite[Proposition~3.5]{Fukuoka17}, which is an inverse transformation of Proposition~\ref{prop-2ray68}, we obtain a sextic del Pezzo fibration $\vp \colon X \to \P^{1}$ and a $\vp$-section $s$ satisfying $(-K_{X})^{3}=22$ and $-K_{X}.s=0$. 
Since $\wt{Q}$ is weak Fano, $X$ is also weak Fano by [ibid, Proposition~3.5]. 
\end{proof}

\begin{rem}\label{rem-quad40}
Let $(B,T)$ be the coverings associated to this sextic del Pezzo fibration. 
Since $(-K_{X})^{3}=22$, Theorem~\ref{mainthm-inv} implies that $B \simeq T \simeq \P^{1}$. 
Then $B \to \P^{1}$ is ramified over exactly two points and hence $Q \to \P^{1}$ has exactly two singular fibers by Lemma~\ref{lem-89}~(1). 
\end{rem}

We prove Theorem~\ref{mainthm-ex} by showing that for any condition in Theorem~\ref{thm-singfibdP6}, there exists a suitable smooth conic $C$ and a point $t \in \P^{1}$ such that the pair $(Q_{t},T_{t})$ satisfies the condition. 

\noindent\textbf{Step 3.} For a point $p \in \Q^{3}$, let $\T_{p}\Q^{3}$ denote the projective tangent space of $\Q^{3}$ at $p$. 
For two points $v_{1},v_{2} \in \Q^{3}$, we set 
\[C(v_{1},v_{2}):=\Q^{3} \cap \T_{v_{1}}\Q^{3} \cap \T_{v_{2}}\Q^{3}.\]
\begin{claim}\label{claim-tang}
Let $C \subset \Q^{3}$ be a smooth conic and $\braket{C} \subset \P^{4}$ be the linear span of $C$. 
Then there exists exactly two points $v_{1}(C),v_{2}(C)$ of $\Q^{3}$ such that 
$C(v_{1}(C),v_{2}(C))=C$. 

Moreover, there exists the following one-to-one correspondence:
\[\begin{array}{ccc}
\left\{ \left. \{v_{1},v_{2}\}  \subset \Q^{3} \right| 
C(v_{1},v_{2})\text{ is smooth } 
 \right\} & \longleftrightarrow & \{\text{ smooth conics in } \Q^{3} \}\\
\rotatebox{90}{$\in$} & & \rotatebox{90}{$\in$} \\
\{v_{1},v_{2}\} & \longmapsto & C(v_{1},v_{2}) \\
\{v_{1}(C),v_{2}(C)\}&\longmapsfrom& C.
\end{array}\]
\end{claim}
\begin{proof}
$q \colon Q \to \P^{1}$ has exactly two singular fibers $Q_{1,C}$ and $Q_{2,C}$ as we saw in Remark~\ref{rem-quad40}. 
Then $Q_{i}:=\t(Q_{i,C}) \subset \Q^{3}$ are singular quadric cones. 
Set $v_{i}(C)=\s(\Sing Q_{i})$.
Then a hyperplane $H$ in $\P^{4}$ containing $\braket{C}$ tangents $\Q^{3}$ if and only if $H \cap \Q^{3}=Q_{i}$ and $H=\T_{v_{i}(C)}\Q^{3}$ for some $i \in \{1,2\}$. 
Therefore, we have $\T_{v_{1}(C)}\Q^{3} \cap \T_{v_{2}(C)}\Q^{3}=\braket{C}$. 

Let us confirm the one-to-one correspondence. 
Take two points $v_{1},v_{2} \in \Q^{3}$ such that $C(v_{1},v_{2})$ is smooth conic. 
Then $\{v_{1}(C(v_{1},v_{2})),v_{2}(C(v_{1},v_{2}))\}$ is the set of vertices of the cones $\{\T_{v_{1}}\Q^{3} \cap \Q^{3},\T_{v_{2}}\Q^{3} \cap \Q^{3}\}$, which is nothing but $\{v_{1},v_{2}\}$. Hence we are done. 
\end{proof}

\noindent\textbf{Step 4.} We finish the proof by presenting suitable examples as follows. 

\begin{ex}[Singular fiber of type $(2,j)$ for $j=1,2,3$]\label{ex-(2,j)}
Fix $j \in \{1,2,3\}$. 
We can take a smooth conic $C \subset \Q^{2}$ such that $\# (C \cap \ol{T})_{\red}=j$. 
Let $Q_{1}$ be a smooth hyperplane section of $\Q^{3}$ such that $C=Q_{1} \cap \Q^{2}$. 
Then we have $Q_{1} \cap \ol{T}=C \cap \ol{T}$. 
Let $Q_{2} \subset \Q^{3}$ be a general hyperplane section such that $Q_{2} \cap Q_{1} \cap \ol{T} =\emp$. 
The pencil $\L=\{\lambda Q_{1} + \mu Q_{2} \mid [\lambda:\mu]\}$ induces a quadric fibration 
$q \colon \Bl_{Q_{1} \cap Q_{2}}\Q^{3} \to \P^{1}$. 
Let $f \colon \Bl_{Q_{1} \cap Q_{2}}\Q^{3} \to \Q^{3}$ be the blow-up. 
If we set $\wt{Q_{i}}:=f^{-1}_{\ast}Q_{i}$, then $\wt{Q_{i}}$ is a $q$-fiber. 
Then we obtain $T \cap \wt{Q_{1}} \simeq \ol{T} \cap C$ since $Q_{1} \cap Q_{2} \cap \ol{T}=\emp$. 
Set $t=q(\wt{Q_{1}})$. 
Seeing the fiber $(Q_{t},\ol{T}_{t})=(\wt{Q_{1}},T \cap \wt{Q_{1}})$ and using Theorem~\ref{thm-singfibdP6}, we obtain an example of a sextic del Pezzo fibration having fibers of type $(2,j)$.
\end{ex}

\begin{ex}[Singular fiber of type $(1,j)$ for $j=1,2,3$]\label{ex-(1,j)}
Fix $j \in \{1,2,3\}$. 
We can take a smooth conic $C \subset \Q^{2}$ such that $\# (C \cap \ol{T})_{\red}=j$. 
By Claim~\ref{claim-tang}, we can take a point $v_{1},v' \in \Q^{3}$ such that $C=C(v_{1},v')$. Then we have $\T_{v_{1}}\Q^{3} \cap \Q^{2}=C$. 

Let $v_{2} \in \Q^{3}$ be a point of $\Q^{3}$ such that $\braket{v_{1},v_{2}} \not\subset \Q^{3}$ and $\T_{v_{2}}\Q^{3} \cap C \cap \ol{T}=\emp$. 
Then 
$C(v_{1},v_{2})$ is smooth and 
%
%
$C(v_{1},v_{2}) \cap \ol{T}=\emp$. 
%
Set $Q_{i}=\T_{v_{i}}\Q^{3} \cap \Q^{3}$ and $\wt{Q_{i}} \subset \Bl_{C(v_{1},v_{2})}\Q^{3}$ be the proper transform of $Q_{i}$ for $i=1,2$. 
Then by Claim~\ref{claim-tang}, $\wt{Q_{1}},\wt{Q_{2}}$ are the singular fibers of the quadric fibration $\Bl_{C(v_{1},v_{2})}\Q^{3} \to \P^{1}$. 
Recall that $Q_{1} \cap \Q^{2}=C$ is a smooth conic. 
Hence $Q_{1} \cap \Q^{2}$ does not contain the vertex of $Q_{1}$ and hence we have $C \cap \ol{T} \simeq \wt{Q_{1}} \cap T$. 
Therefore, $T$ does not pass through the vertex of $Q_{1}$. 
Set $t=q(\wt{Q_{1}})$. 
Seeing the fiber $(Q_{t},\ol{T}_{t})=(\wt{Q_{1}},T \cap \wt{Q_{1}})$ and using Theorem~\ref{thm-singfibdP6}, we obtain an example of a sextic del Pezzo fibration having fibers of type $(1,j)$.
\end{ex}

\begin{ex}[Singular fibers of type (n2) and (n4)]\label{ex-nonnormal} 

We fix an isomorphism $\Q^{2} \simeq \P^{1}_{a} \times \P^{1}_{b}$ and let $l_{a}$ and $l_{b}$ be the two rulings. 
We may assume that $\ol{T} \in |l_{a}+2l_{b}|$. 
Let $g_{a} \colon \ol{T} \to \P^{1}_{a}$ be the restriction of the first projection to $\ol{T}$. 
Take a point $p_{1} \in \ol{T}$. 
Then we have $\T_{p_{1}}\Q^{2} =l_{a}+l_{b}$ and hence 
$\T_{p_{1}}\Q^{3} \cap \ol{T}=\T_{p_{1}}\Q^{3} \cap \Q^{2} \cap \ol{T}=(l_{a}+l_{b}) \cap \ol{T}=p_{1}+g_{a}^{-1}(g_{a}(p_{1}))$ 
as effective Cartier divisors on $\ol{T}$. 

Let $p_{2} \in \Q^{3}$ be a general point such that $C(p_{1},p_{2})$ is smooth conic. 
Let $Q_{i}:=\T_{p_{i}}\Q^{3} \cap \Q^{3}$, $q \colon Q=\Bl_{C}\Q^{3} \to \P^{1}$ the quadric fibration and $\wt{Q_{i}}$ the proper transform of $Q_{i}$. 
Then $\wt{Q_{1}}$ is a singular $q$-fiber with the vertex $p_{1}$ and $\wt{Q_{1}} \cap T \simeq \ol{T} \cap (l_{a}+l_{b})=p_{1}+g_{a}^{-1}(g_{a}(p_{1}))$. 

If $p_{1}$ is a unramified (resp. ramified) point of $\ol{T} \to \P^{1}_{a}$, 
then we obtain an example of a sextic del Pezzo fibration having fibers of type (n2) (resp. (n4)) by seeing the fiber $(\wt{Q_{1}},T \cap \wt{Q_{1}})$ and using Theorem~\ref{thm-singfibdP6}. 
\end{ex}

We complete the proof of Theorem~\ref{mainthm-ex}. \qed

\section{Related result for Du Val families of sextic del Pezzo surfaces}

Since a paper \cite{AHTVA16} came out, families of sextic del Pezzo have been studied more actively from various interests, e.g., cubic 4-folds and their rationality, Hodge theory and derived categories. 
We devote this section to mention some related results about Du Val families of sextic del Pezzo surfaces, which are defined as follows. 
\begin{defi}[{\cite[Definition~5.1]{Kuznetsov17}}]
A morphism $f \colon \ms X \to S$ between schemes over $\C$ is said to be a \emph{Du Val family of sextic del Pezzo surfaces} if $f$ is a flat projective morphism and every geometric fiber of $f$ is a sextic del Pezzo surface with only Du Val singularities. 
\end{defi}

\begin{rem}
Note that we do not impose the condition $\rho(\ms X/S)=1$ on a Du Val family of sextic del Pezzo surfaces. 
A sextic del Pezzo fibration in our sense is not necessarily a Du Val family of sextic del pezzo surfaces since it may have non-normal fiber by Theorem~\ref{mainthm-ex}. 
\end{rem}

\begin{dsc}
Addington, Hassett, Tschinkel and V{\'a}rilly-Alvarado \cite{AHTVA16} treated Du Val family of sextic del Pezzo surfaces $f \colon \ms X \to S$ such that $\ms X$, $S$ are smooth and which is good in the sense of \cite[Definition~11]{AHTVA16}. 
In this case, they constructed a double cover $Y \to S$, a triple cover $Z \to S$, a Brauer class on $Y$, and that on $Z$ corresponding to $f$. 
In the paper \cite{AHTVA16}, they mainly studied about the case when $\ms X$ is the blow-up of a cubic $4$-fold along a certain elliptic scroll of degree $6$. 
In this case, the corresponding double covering $Y \to \P^{2}$ is branched along a sextic curve and $Y$ is a smooth K3 surface. 
They expect that this degree 2 smooth K3 surface with the Brauer class is obtained by using Hodge theory \cite[Proposition~10 and Remark~18]{AHTVA16}. 
\end{dsc}

\begin{dsc}
After that, for any Du Val family of sextic del Pezzo surfaces $f \colon \ms X \to S$, Kuznetsov \cite{Kuznetsov17} constructed two finite flat morphisms $\ms Z_{2} \to S$ and $\ms Z_{3} \to S$ of degrees $3$ and $2$ respectively and a semiorthogonal decomposition 
$\mathbf{D}(\ms X)=\braket{\mathbf{D}(S),\mathbf{D}(\ms Z_{2},\beta_{\ms Z_{2}}), \mathbf{D}(\ms Z_{3},\beta_{\ms Z_{3}})}$, 
where $\beta_{\ms Z_{2}}$ (resp. $\beta_{\ms Z_{3}}$) is a Brauer class on $\ms Z_{2}$ (resp. a Brauer class on $\ms Z_{3}$) \cite[Theorem~5.2]{Kuznetsov17}. 
These coverings with Brauer classes are generalization of that of \cite{AHTVA16}. 
\end{dsc}
From the viewpoint of the following lemma, our associated coverings $(\vp_{B},\vp_{T})$ associated to a sextic del Pezzo fibration $\vp$ are generalization of coverings $(\ms Z_{3} \to C,\ms Z_{2} \to C)$. 

\begin{prop}\label{prop-Kuz}
Let $\vp \colon X \to C$ be a sextic del Pezzo fibration such that every $\vp$-fiber is normal. 
Then the two finite flat morphisms $\ms Z_{2} \to C$ and $\ms Z_{3} \to C$ due to \cite{Kuznetsov17} are $\vp_{T} \colon T \to C$ and $\vp_{B} \colon B \to C$ in Definition~\ref{defi-BT} respectively. 
\end{prop}
\begin{proof}
We use the same notation as in Lemma~\ref{lem-indep}. 
By \cite[Propositions~5.12 and 5.14]{Kuznetsov17}, $\Hilb_{dt+1}(X/C) \to C$ factors through $\ms Z_{d}$ as the Stein factorization and $\ms Z_{d}$ is non-singular. 
Thus Lemma~\ref{lem-indep} implies this assertion. 
\end{proof}

\begin{dsc}\label{work-AABF}
Let $f \colon \ms X \to S$ be a Du Val family of sextic del Pezzo surfaces and suppose that $\ms X$ and $S$ are smooth. 
Recently, assuming some conditions on the discriminant of $f$, Addington, Auel, Bernardara and Faenzi constructed a $(\P^{2})^{2}$-fibration $f_{\ms Y} \colon \ms Y \to S$ with smooth $\ms Y$ such that $\ms Y$ contains $\ms X$ as a relative linear section. 
More precisely, they supposed that $f$ is good in the sense of \cite[Definition~11]{AHTVA16} and $\{t \in S \mid \ms X_{t} \text{ is of type } (1,j) \text{ for some } j \} $ is a smooth divisor on $S$. 

This statement is close to Theorem~\ref{mainthm-P22}. 
However, the result does not implies Theorem~\ref{mainthm-P22} and not conversely. 
\end{dsc}

It should be emphasized that our method are totally different from that of them. 
Our results have mutually complementary relationships with the above results. 

\appendix
\section{Relative universal extension of sheaves}\label{app-relext}

This appendix is devoted to prove Theorem~\ref{thm-univext}. 

\subsection{Relative Ext-sheaf}

First, we recall the notion of relative Ext sheaves and organize some basic properties. 
Let $f \colon X \to Y$ be a proper morphism between noetherian schemes $X$ and $Y$. 
Let $\mc F$ and $\mc G$ be coherent sheaves on $X$ and $\mc E$ a coherent sheaf on $Y$. 
\begin{defi}[cf. \cite{Kleiman80,BPS80}]
We denote the $i$-th cohomology of right derived functor of $f_{\ast}\mc Hom(\mc F,-)$ as $\mc Ext^{i}_{f}(\mc F,-)$. 
We call this sheaf $\mc Ext^{i}_{f}(\mc F,\mc G)$ the \emph{relative Ext sheaf}. 
\end{defi}

\begin{rem}
(1) $\mc Ext^{i}_{f}(\mc O_{X},\mc G)=R^{i}f_{\ast}\mc G$ by the definition. 

(2) Composing natural canonical morphisms 
\[f_{\ast}\mc Hom(\mc F,\mc G) \otimes \mc E \to f_{\ast}\left(\mc Hom(\mc F,\mc G) \otimes f^{\ast}\mc E\right) \to f_{\ast}\mc Hom(\mc F,\mc G \otimes f^{\ast}\mc E),\]
we obtain a natural morphism $\mc Ext^{i}_{f}(\mc F,\mc G \otimes f^{\ast}\mc E) \to \mc Ext^{i}_{f}(\mc F,\mc G) \otimes \mc E$. 
By the projection formula, this is an isomorphism if $\mc E$ is locally free. 
\end{rem}

\begin{rem}\label{rem-relext}
Consider the following three spectral sequences:
\begin{align}
\label{eq1-relext}
&R^{i}f_{\ast}\mc Ext^{j}(\mc F,\mc G \otimes f^{\ast}\mc E) \ra \mc Ext^{i+j}_{f}(\mc F,\mc G\otimes f^{\ast}\mc E), \\
\label{eq2-relext}
&H^{i}(Y,\mc Ext^{j}_{f}(\mc F,\mc G\otimes f^{\ast}\mc E)) \ra \Ext^{i+j}(\mc F,\mc G\otimes f^{\ast}\mc E) \text{ and } \\
\label{eq3-relext}
&H^{i}(X,\mc Ext^{j}(\mc F,\mc G\otimes f^{\ast}\mc E)) \ra \Ext^{i+j}(\mc F,\mc G\otimes f^{\ast}\mc E). 
\end{align}
These spectral sequence give three exact sequences as follows: 
{\small 
\begin{align}
\label{eq4-relext}
0& \to R^{1}f_{\ast}\mc Hom(\mc F,\mc G\otimes f^{\ast}\mc E) \to \mc Ext^{1}_{f}(\mc F,\mc G\otimes f^{\ast}\mc E) \mathop{\to}^{\a} f_{\ast}\mc Ext^{1}(\mc F,\mc G\otimes f^{\ast}\mc E) \\
\notag
&\quad \to R^{2}f_{\ast}\mc Hom(\mc F,\mc G\otimes f^{\ast}\mc E) \to \mc Ext^{2}_{f}(\mc F,\mc G \otimes f^{\ast}\mc E),  \\
\label{eq5-relext}
0& \to H^{1}(Y,f_{\ast}\mc Hom(\mc F,\mc G\otimes f^{\ast}\mc E)) \to \Ext^{1}(\mc F,\mc G\otimes f^{\ast}\mc E) \mathop{\to}^{\b'} H^{0}(Y,\mc Ext^{1}_{f}(\mc F,\mc G\otimes f^{\ast}\mc E)) \\
\notag
&\quad \to H^{2}(Y,f_{\ast}\mc Hom(\mc F,\mc G\otimes f^{\ast}\mc E)) \to \Ext^{2}(\mc F,\mc G \otimes f^{\ast}\mc E)
 \text{ and } \\
\label{eq6-relext}
0& \to H^{1}(X,\mc Hom(\mc F,\mc G\otimes f^{\ast}\mc E)) \to \Ext^{1}(\mc F,\mc G\otimes f^{\ast}\mc E) \mathop{\to}^{\g'} H^{0}(X,\mc Ext^{1}(\mc F,\mc G\otimes f^{\ast}\mc E)) \\
\notag
&\quad \to H^{2}(X,\mc Hom(\mc F,\mc G\otimes f^{\ast}\mc E)) \to \Ext^{2}(\mc F,\mc G \otimes f^{\ast}\mc E).
\end{align}
}
Then we obtain the following commutative diagram: 
\[
\xymatrixrowsep{5mm}
\xymatrixcolsep{5mm}
\xymatrix{
\Ext^{1}(\mc F,\mc G \otimes f^{\ast}\mc E) \ar[r]^{\b' \quad} \ar@(ur,lu)[rr]^{\g'} &H^{0}(Y,\mc Ext^{1}_{f}(\mc F,\mc G \otimes f^{\ast}\mc E)) \ar[d]\ar[r]^{ H^{0}(\a) \ } & H^{0}(Y,f_{\ast}\mc Ext^{1}(\mc F,\mc G \otimes f^{\ast}\mc E)) \ar[d]&\\
&H^{0}(Y,\mc Ext^{1}_{f}(\mc F,\mc G) \otimes \mc E)  \ar[d] \ar[r]& H^{0}(Y,f_{\ast}\mc Ext^{1}(\mc F,\mc G) \otimes \mc E) \ar[d]&\\
&\Hom(\mc E^{\vee},\mc Ext^{1}_{f}(\mc F,\mc G)) \ar[r]_{\a \circ -} & \Hom(\mc E^{\vee},f_{\ast}\mc Ext^{1}(\mc F,\mc G)) \ar[d]^{\simeq}_{\mathrm{adjoint}}& \\
&&\Hom(f^{\ast}\mc E^{\vee},\mc Ext^{1}(\mc F,\mc G))  
}\]

Note that all vertical arrows are isomorphic when $\mc E$ is locally free. 

Let $\wt{\b},\b,\g$ denote the composition of morphisms in the above diagram as follows:
\begin{align}
\label{eqwtb-relext}
&\wt{\b} \colon \Ext^{1}(\mc F,\mc G \otimes f^{\ast}\mc E) \to H^{0}(Y,\mc Ext^{1}_{f}(\mc F,\mc G) \otimes \mc E) \\
\label{eqb-relext}
&\b \colon \Ext^{1}(\mc F,\mc G \otimes f^{\ast}\mc E) \to \Hom(\mc E^{\vee},\mc Ext^{1}_{f}(\mc F,\mc G)) \\
\label{eqg-relext}
&\g \colon \Ext^{1}(\mc F,\mc G \otimes f^{\ast}\mc E) \to \Hom(f^{\ast}\mc E^{\vee},\mc Ext^{1}(\mc F,\mc G))
\end{align}
For an element $t \in \Ext^{1}(\mc F,\mc G \otimes f^{\ast}\mc E)$, 
it is easy to verify that the composite morphism 
\begin{align}
\label{eqea-relext}
f^{\ast}\mc E^{\vee} \mathop{\to}^{f^{\ast}\b(t)} f^{\ast}\mc Ext^{1}_{f}(\mc F,\mc G) \mathop{\to}^{f^{\ast}\a} f^{\ast}f_{\ast}\mc Ext^{1}(\mc F,\mc G) \mathop{\to}^{\e} \mc Ext^{1}(\mc F,\mc G)
\end{align}
is nothing but $\g(t)$, where $\e \colon f^{\ast}f_{\ast}\mc Ext^{1}(\mc F,\mc G) \to \mc Ext^{1}(\mc F,\mc G)$ is the natural morphism. 
\end{rem}

\begin{lem}\label{lem-locfr}
Assume that $X$ is regular, $\mc E$ is locally free, $\mc G$ is locally free and $\mc F$ satisfies $\hd(\mc F_{x}) \leq 1$ for every $x \in X$. 
Let $t \in \Ext^{1}(\mc F,\mc G \otimes f^{\ast}\mc E)$ be an arbitrary element  
and $\mc H_{t}$ the extension corresponding to $t$: 
\begin{align}\label{eq-locfr}
0 \to \mc G \otimes f^{\ast}\mc E \to \mc H_{t} \to \mc F \to 0.
\end{align}
If $\g(t)$ is surjective, then $\mc H_{t}$ is locally free. 
\end{lem}
\begin{proof}
$\mc H_{t}$ is locally free if and only if $\mc Ext^{i}(\mc H_{t},\mc G) =0$ for any $i \geq 1$ since $\mc G$ is locally free and $X$ is regular. 
By taking $\mc Hom(-,\mc G )$ of the sequence (\ref{eq-locfr}), we obtain an exact sequence
\begin{align*}
0 &\to \mc Hom(\mc F,\mc G) \to \mc Hom(\mc H_{t},\mc G) \to \mc Hom(\mc G \otimes f^{\ast}\mc E,\mc G) \\
&\mathop{\to}^{\d} \mc Ext^{1}(\mc F,\mc G) \to \mc Ext^{1}(\mc H_{t},\mc G) \to 0. 
\end{align*}
Since $\hd(\mc F_{x}) \leq 1$ for all $x \in X$, we have $\mc Ext^{i}(\mc H_{t},\mc G)=0$ for any $i \geq 2$. 
Therefore, $\mc H_{t}$ is locally free if and only if $\d$ is surjective. 
Let $\t \colon f^{\ast}\mc E^{\vee} \to \mc Hom(\mc G \otimes f^{\ast}\mc E,\mc G)$ be the natural map which comes from the morphism $\mc O \ni 1 \mapsto \id \in \mc Hom(\mc G,\mc G)$. 
By the definition of $\g(t)$, we obtain $\d \circ \t=\g(t)$. 
If $\g(t)$ is surjective, then so is $\d$, which completes the proof. 
\end{proof}

\subsection{Relative universal extensions of sheaves}
Next we introduce a notion of $f$-universal extension. 
Let $f \colon X \to Y$ be a proper morphism between noetherian schemes $X,Y$ and $\mc F,\mc G$ coherent sheaves on $X$. 
\begin{defi}\label{defi-univext}
Put $\mc E=\mc Ext^{1}_{f}(\mc F,\mc G)^{\vee}$. 
Consider the composition 
\begin{align}\label{eqt-relext}
\t \colon \Ext^{1}(\mc F,\mc G \otimes f^{\ast}\mc Ext^{1}_{f}(\mc F,\mc G)^{\vee}) &\mathop{\to}^{\wt{\b}} H^{0}(Y,\mc Ext^{1}_{f}(\mc F,\mc G) \otimes \mc Ext^{1}_{f}(\mc F,\mc G)^{\vee}) \\
\notag&\mathop{\to}^{\theta} H^{0}(Y,\mc O_{Y}),
\end{align}
where $\wt{\b}$ is the morphism defined in (\ref{eqwtb-relext}) and $\theta$ is the natural morphism. 
We say that an element $t \in \Ext^{1}(\mc F,\mc G \otimes f^{\ast}\mc Ext^{1}_{f}(\mc F,\mc G)^{\vee})$ is \emph{universal} if $\t(t)=1$. 
If $t$ is universal, then we say that the corresponding extension 
\[0 \to \mc G \otimes f^{\ast}\mc Ext^{1}_{f}(\mc F,\mc G)^{\vee} \to \mc H_{t} \to \mc F \to 0.\] 
is an \emph{$f$-universal extension} of $\mc F$ by $\mc G$. 
If $Y=\Spec k$ for some field $k$, we just say that $\mc H_{t}$ is a \emph{universal extension} of $\mc F$ by $\mc G$. 
\end{defi}

The following lemma is a criterion for the existence of a locally free $f$-universal extension of $\mc F$ by $\mc G$. 
\begin{lem}\label{lem-univlocfr}
Suppose the following conditions:
\begin{enumerate}
\item $X$ is regular, $\mc G$ is locally free and $\hd(\mc F_{x}) \leq 1$ for any $x \in X$. 
\item $\e \circ f^{\ast}\a \colon f^{\ast}\mc Ext^{1}_{f}(\mc F,\mc G) \to \mc Ext^{1}(\mc F,\mc G)$ is surjective (See (\ref{eqea-relext})).
\item $\mc Ext_{f}^{1}(\mc F,\mc G)$ is locally free. 
\item $H^{2}(Y,f_{\ast}\mc Hom(\mc F,\mc G) \otimes \mc Ext^{1}_{f}(\mc F,\mc G)^{\vee})=0$. 
\end{enumerate}
Then there exists a locally free $f$-universal extension $\mc H$ of $\mc F$ by $\mc G$. 
\end{lem}
\begin{proof}
By (3), the morphism $\theta$ in (\ref{eqt-relext}) is surjective. 
By (4) and the exact sequence (\ref{eq5-relext}), the morphism $\wt{\b}$ in (\ref{eqwtb-relext}) is surjective. 
Therefore, the map $\t$ in (\ref{eqt-relext}) is surjective and hence there exists an $f$-universal extension $\mc H$. 
By (1), (2) and Lemma~\ref{lem-locfr}, this $\mc H$ is locally free. 
\end{proof}

\begin{lem}\label{lem-nosurj}
Let $X$ be a proper variety over a field and $\mc F$ be a coherent sheaf on $X$. 
Let $\mc L$ be a invertible sheaf on $X$. 
Suppose that there exists a universal extension of $\mc F$ by $\mc L$: 
\begin{align}
\label{eq-genc}
0 \to \mc L \otimes \Ext^{1}(\mc F,\mc L)^{\vee} \to \mc H \to \mc F \to 0. 
\end{align}
If there are no surjections $\mc F \to \mc L$, then there are no surjections $\mc H \to \mc L$. 
\end{lem}
\begin{proof}
To obtain a contradiction, assume that there are no surjections $\mc F \to \mc L$ and there is a surjection $a \colon \mc H \to \mc L$. 
We consider the following commutative diagram:
\[\xymatrix{
0 \ar[r] & \mc L \otimes \Ext^{1}(\mc F,\mc L)^{\vee} \ar[r] \ar[d]_{b}& \mc H \ar[r] \ar[d]_{a}& \mc F \ar[r] \ar[d]& 0 \\
0 \ar[r] & \mc L \ar@{=}[r] & \mc L \ar[r] & 0,
}\]
where $b$ is the composition of the morphisms $\mc L \otimes \Ext^{1}(\mc F,\mc L)^{\vee} \to \mc H$ and $a$. 
If $b=0$, then there exists a surjection $\mc F \to \mc L$ by the snake lemma, which is contradiction and hence $b \neq 0$. 
Since $X$ is a proper variety and $\mc L$ is invertible, $b$ is surjection and splits. 
Let $V$ be the linear subspace of $\Ext^{1}(\mc F,\mc L)^{\vee}$ such that $\Ker b=\mc L \otimes V$. 
Then we obtain an exact sequence 
\[0 \to \mc L \otimes V \to \Ker a \to \mc F \to 0\]
and the exact sequence (\ref{eq-genc}) is the push-forward of this exact sequence, which implies that the identity morphism $\id \colon \Ext^{1}(\mc F,\mc L) \to \Ext^{1}(\mc F,\mc L)$ factors through $V$, which is a contradiction. 
Hence we are done. 
\end{proof}

\subsection{Proof of Theorem~\ref{thm-univext}}

In order to prove Theorem~\ref{thm-univext}, 
we show the following lemma, which was essentially proved in \cite[(6.2.2)]{Stankova-Frenkel00}. 
The author gives the proof for reader's convenience. 

\begin{prop}[{\cite[(6.2.2)]{Stankova-Frenkel00}}]\label{prop-surj}
Let $Y$ and $Z$ be noetherian schemes and $g \colon Z \to Y$ be a non-isomorphic surjective finite flat morphism. 
Set $\mc E=\Cok(\mc O_{Y} \to g_{\ast}\mc O_{Z})$ and let $a \colon g_{\ast}\mc O_{Z} \to \mc E$ be the natural surjective homomorphism. 
Let $(g_{\ast}\mc O_{Z})^{\vee} \simeq g_{\ast}\omega_{g}$ be the isomorphism by the duality of the finite flat morphism $g$. 
Let $e \colon g^{\ast}g_{\ast}\omega_{g} \to \omega_{g}$ be the natural morphism.
Then the composition $c:=e \circ g^{\ast}a^{\vee} \colon g^{\ast}\mc E^{\vee} \to \omega_{g}$ is surjective: 
\end{prop}

\begin{proof}
Note that $\mc E$ is locally free. 
Let $\Spec A \simeq U \subset Y$ be an affine open subset and $V:=g^{-1}(U)=\Spec B$. 
Consider the corresponding injective homomorphism $A \to B$. If we set $E=B/A$, then $E \simeq \G(U,\mc E)$ and we may assume that $E$ is a free $A$-module by shrinking $U$. 
If we set $\deg f|_{Z}=n+1$, then there exists $b_{1},\ldots,b_{n} \in B$ such that $E=\bigoplus_{i=1}^{n} A \cdot b_{i}$ and $\Tr(b_{i})=0$ for any $i$, where $\Tr \colon B \to A$ denotes the trace map. 

Set $\omega_{B/A}:=\G((f|_{Z})^{-1}(U),\omega_{g})$. 
Then $\omega_{B/A} \simeq \!_{B}(\Hom_{A}(B,A))$ as $B$-modules by the duality of a finite flat morphism $\Spec B \to \Spec A$. 
Note that $\Hom_{A}(B,A)$ has the natural $B$-module structure that is defined by 
\[(b.f)(x):=f(bx) \text{ for } f \in \Hom_{A}(B,A),\ b \in B \text{ and } x \in B.\]
Then the map $B \ni b \mapsto (x \mapsto \Tr(b.x)) \in \Hom_{A}(B,A)$ is a $B$-module isomorphism. 
Based on this setting, we identify the morphism $c|_{V}$ with 
\[c_{B} \colon \Hom_{A}(E,A) \otimes_{A} B \ni (f \otimes b) \mapsto b.f \in \!_{B}(\Hom_{A}(B,A)).\]
It suffices to show the surjectivity of $c_{B}$. 
Let $1^{\vee},b_{1}^{\vee},\ldots,b_{n}^{\vee} \in \Hom_{A}(B,A)$ be the dual elements of $1,b_{1},\ldots,b_{n} \in B$ respectively. 
If $n=1$, then we have $c_{B}(b_{1}^{\vee} \otimes (b_{1}-b_{1}^{\vee}(b_{1}^{2})))=1^{\vee}$ and hence $c_{B}$ is surjective since $\!_{B}(\Hom_{A}(B,A))$ is generated by $1^{\vee}$ as a $B$-module. 
If $n \geq 2$, then it holds that
\[c_{B} \left( b_{1}^{\vee} \otimes (b_{1}-b_{1}^{\vee}(b_{1}^{2})) + \sum_{i=2}^{n} b_{i}^{\vee} \otimes (-b_{1}^{\vee}(b_{1}b_{i})) \right) = 1^{\vee}. \]
Therefore, $c_{B}$ is surjective for any $n \geq 1$. 
We complete the proof. 
\end{proof}

\begin{proof}[Proof of Theorem~\ref{thm-univext}]
Consider the exact sequence 
$0 \to \mc O_{Y} \to f_{\ast}\mc O_{Z} \overset{a}{\to} \mc E \to 0$. 
Then there exists an injection $\mc E \hra R^{1}f_{\ast}\mc I_{Z}$ by seeing the exact sequence $0 \to \mc I_{Z} \to \mc O_{X} \to \mc O_{Z} \to 0$. 
Since $Z$ is Cohen-Macaulay, $f|_{Z}$ is flat. 
We fix the isomorphism $(f_{\ast}\mc O_{Z})^{\vee} \simeq f_{\ast}\omega_{f|_{Z}}$ which comes from the duality of the finite flat morphism $f|_{Z}$. 
Then Lemma~\ref{prop-surj} gives that the composition 
$\ds \wt{c} \colon  f^{\ast}\mc E^{\vee} \mathop{\to}^{f^{\ast}a^{\vee}} f^{\ast}(f_{\ast}\mc O_{Z})^{\vee} \simeq f^{\ast}(f_{\ast}\omega_{f|_{Z}}) \mathop
{\to}^{e} \omega_{f|_{Z}}$ 
is surjective, where $e$ is the natural morphism. 
Since every fiber of $f$ and $f|_{Z}$ is Cohen-Macaulay, \cite[Theorem~(21)]{Kleiman80} gives the following commutative diagram:
\[\xymatrix@C=0.97em{
& 0 \ar[rd]& & & & \\
& & \mc E^{\vee} \ar[rd]& & & \\
0 \ar[r]&(R^{1}f_{\ast}\mc O_{X})^{\vee} \ar[r]& (R^{1}f_{\ast}\mc I_{Z})^{\vee} \ar[r] \ar[u] & (f_{\ast}\mc O_{Z})^{\vee} \ar[r] & (f_{\ast}\mc O_{X})^{\vee} \ar[r] & 0  \\
0 \ar[r] & \mc Ext^{1}_{f}(\mc O_{X},\omega_{f}) \ar[r] \ar[u]^{\simeq}& \mc Ext^{1}_{f}(\mc I_{Z},\omega_{f}) \ar[r] \ar[u]^{\simeq}& \mc Ext^{2}_{f}(\mc O_{Z},\omega_{f}) \ar[r] \ar[u]^{\simeq}& \mc Ext^{2}_{f}(\mc O_{X},\omega_{f}) \ar[u]^{\simeq}.  &
}\]
Noting that $\mc Ext^{i}_{f}(\mc O_{X},\omega_{f})=R^{i}f_{\ast}\omega_{f}$, we obtain the following exact sequence: 
\[0 \to R^{1}f_{\ast}\omega_{f} \to \mc Ext^{1}_{f}(\mc I_{Z},\omega_{f}) \mathop{\to}^{\vp} \mc E^{\vee} \to 0.\]
\begin{claim}
(1) -- (4) in Lemma~\ref{lem-univlocfr} hold for $\mc F=\mc I_{Z}$ and $\mc G=\omega_{f}$. 
\end{claim}
\begin{proof}
(1) and (4) hold by the assumption. 
Since $R^{1}f_{\ast}\mc O_{X}$ is locally free and hence so are $R^{1}f_{\ast}\mc I_{Z}$ and $\mc Ext^{1}_{f}(\mc I_{Z},\omega_{f})$, (3) also holds. 
The item (2) holds since $\wt{c} \circ f^{\ast}\vp \colon f^{\ast}\mc Ext^{1}_{f}(\mc I_{Z},\omega_{f}) \to \mc Ext^{1}(\mc I_{Z},\omega_{f}) \simeq \omega_{f|_{Z}}$ coincides with $\e \circ f^{\ast}\a$ (See (\ref{eqea-relext})) and $\wt{c}$ and $f^{\ast}\vp$ are surjective. 
\end{proof}
Therefore, we obtain an $f$-universal extension $\mc H$ which is locally free:
\begin{align}\label{eq1-univext}
0 \to \omega_{f} \otimes f^{\ast}\mc Ext^{1}_{f}(\mc I_{Z},\omega_{f})^{\vee} \to \mc H \to \mc I_{Z} \to 0. 
\end{align}
\begin{claim}\label{claim-univext}
For any $y \in Y$, $\mc H|_{X_{y}}$ is a universal extension of $\mc I_{Z \cap X_{y}}$ by $\omega_{X_{y}}$. 
\end{claim}
\begin{proof}
Since $f$ and $f|_{Z}$ is flat, the restriction of the exact sequence (\ref{eq1-univext}) to $X_{y}$ is also exact for every $y \in Y$:
\[0 \to \omega_{X_{y}} \otimes (\mc Ext^{1}_{f}(\mc I_{Z},\omega_{f})
^{\vee}\otimes k(y)) \otimes \mc O \to \mc H|_{X_{y}} \to \mc I_{Z \cap X_{y}} \to 0.\]
Since $\Ext^{2}(\mc I_{Z \cap X_{y}},\omega_{X_{y}})=H^{0}(X_{y},\mc I_{Z \cap X_{y}})=0$ holds, the natural morphism $\mc Ext^{1}_{f}(\mc I_{Z},\omega_{f}) \otimes k(y) \to \Ext^{1}(\mc I_{Z \cap X_{y}},\omega_{X_{y}})$ is isomorphism by \cite[Satz~3]{BPS80}. 
Hence the extension on $X_{y}$ is also universal, which is the assertion.
\end{proof}

Tensoring the sequence (\ref{eq1-univext}) with $\omega_{X}^{-1}$ and identifying 
$\mc Ext^{1}_{f}(\mc I_{Z},\omega_{f})$ with $(R^{1}f_{\ast}\mc I_{Z})^{\vee}$, we have the exact sequence 
\[0 \to f^{\ast}((R^{1}f_{\ast}\mc I_{Z})(-K_{Y})) \to \mc F \to \mc I_{Z}(-K_{X}) \to 0.\]
Then $\mc F$ is an $f$-universal extension of $\mc I_{Z}(-K_{X})$ by $\mc O_{X}$. 
Claim~\ref{claim-univext} gives that $\mc F|_{X_{y}}$ is a universal extension of $\mc I_{Z \cap X_{y}}(-K_{X_{y}})$ by $\mc O_{X_{y}}$ for every $y$. 
Since there are no surjections from $\mc I_{Z \cap X_{y}}(-K_{X_{y}})$ to $\mc O_{X_{y}}$, the property (2) follows from Lemma~\ref{lem-nosurj}, which completes the proof. 
\end{proof}

\providecommand{\bysame}{\leavevmode\hbox to3em{\hrulefill}\thinspace}
\providecommand{\MR}{\relax\ifhmode\unskip\space\fi MR }
\providecommand{\MRhref}[2]{%
  \href{http://www.ams.org/mathscinet-getitem?mr=#1}{#2}
}
\providecommand{\href}[2]{#2}

\end{document}